\documentclass[10pt]{article}
\usepackage{amsfonts,amsmath,latexsym,amsthm,amssymb,graphicx,xypic}

\input{mathdg2015-mwcqm.sty}

\usepackage{hyperref}
\usepackage{tikz}
\usepackage{tikz-cd}  
\usetikzlibrary{decorations.pathreplacing,3d} 
\usetikzlibrary{positioning} 
\usetikzlibrary{arrows}
\usepackage{xypic}
\usepackage{adjustbox}      

\usepackage{enumerate} 

\numberwithin{equation}{section}

\newcommand{\dom}{\operatorname{dom}}

\newcommand{\CinfD}{C^{\infty}_{\mathrm {D}}}
\newcommand{\Cinftr}{C^{\infty,\mathrm{tr}}}  
\newcommand{\CinfDtr}{\Cinftr_{\mathrm{D}}}  

\newcommand{\dDir}{\partial_{\mathrm{D}}}  

\newcommand{\LP}{\operatorname{LP}}  

\renewcommand{\Omegabar}{\overline{\Omega}}

\def\bupr{.4}   
\def\pcaxis{.2} 

\newtheorem*{inductivesteplemma}{Inductive step lemma}

\newtheorem*{leading part and model operator lemma}{Leading part and model operator lemma}

\newenvironment{indlist}
{
\begin{list}{}{\setlength{\itemsep}{2mm}}
}
{ 
\end{list}
 }

\AtEndDocument{\bigskip{\footnotesize%
  \textsc{Institut f\"ur Mathematik, Carl von Ossietzky Universit\"at Oldenburg, 26111 Oldenburg} \par  
  \textit{E-mail address:} \texttt{daniel.grieser@uni-oldenburg.de}
}}

\begin{document}
 \title
 {Scales, blow-up and quasimode constructions}
 \author{Daniel Grieser}
 \date{\today}



%
%

 \maketitle

\begin{abstract}
In this expository article we show how the  concepts of manifolds with corners, blow-ups and resolutions can be used effectively for the construction of quasimodes, i.e.\ of approximate eigenfunctions of the Laplacian on certain families of spaces, mostly exemplified by domains $\Omega_h\subset\R^2$, that degenerate as $h\to0$. These include standard adiabatic limit families and also families that exhibit several types of scaling behavior.  
An introduction to manifolds with corners and resolutions, and how they relate to the ideas of (multiple) scales and matching, is included. 
\end{abstract}

\setcounter{tocdepth}{2}
 \tableofcontents
 
\section{Introduction}
This article gives an introduction to the ideas of blow-up and resolution, and how they can be used for the construction of quasimodes for the Laplacian in singular perturbation problems. Blow-up is a rigorous geometric tool for describing multiple scales, which appear in many analytic problems in pure and applied mathematics. 
The construction of quasimodes is a low-tech yet non-trivial problem where this tool can be used effectively.

\smallskip

\noindent{\em The idea of scales.} One of the fundamental ideas in analysis is {scale.} As an illustration consider the function
\begin{equation}
 \label{eqn:fh def}
 f_h(x) = \frac x{x+h}\,,\quad x\in [0,1]
\end{equation}
where $h$ is a \lq small' positive number, see Figure \ref{fig:fh graph} for $h=0.1$ and $h=0.01$. Observe that at $x=0$ the function takes the value 0 while for \lq most' values of $x$ it is \lq close' to 1. On the other hand, taking $x=h$ we get $f_h(h)=\frac12$, and more generally if $x$ is \lq on the order of $h$' then $f_h(x)$ will be somewhere \lq definitely between 0 and 1'.

This may be the way a physicist describes the function $f_h$, even without the quotation marks; to a mathematician 
the quotes create a sense of uneasiness, so we search for a precise statement. We then realize that we are really talking about the {\em family} of functions $(f_h)_{h>0}$ and its limiting behavior as $h\to0$.  
More precisely, we first have the pointwise limit
\begin{equation}
 \label{eqn:fh pointwise}
 \lim_{h\to0}f_h(x)  = f_0(x):=
\begin{cases}
 0&\text{ if }x=0\\
 1&\text{ if }x>0.
\end{cases}
\end{equation}
On the other hand, we have the {\em rescaled limit} where we set $x=hX$ and fix $X$ while letting $h\to0$:
\begin{equation}
\label{eqn:fh rescaled}
\lim_{h\to0} f_h(hX) = g(X) := \frac{X}{X+1}\,,\quad X\geq0\,.
\end{equation}
The function $g$ shows how the transition from the value 0 to almost 1 happens in $f_h$. We call this the limit of $f_h$ at the scale $x\sim h$, while \eqref{eqn:fh pointwise} is the limit at the scale $x\sim 1$. 
We could also consider other scales, i.e. limits $\lim_{h\to0} f_h(h^aX)$ with $a\in\R$, but they don't give new insights in this case: if $a<1$ then we get the jump function $f_0$ while for $a>1$ we just get zero.

Summarizing, we see that the family $(f_h)$ has non-trivial behavior at two scales, $x\sim1$ and $x\sim h$, for $h\to0$. 
\begin{figure}
\def\dash{.05}  
\begin{tikzpicture}[scale=2]
  \pgfgettransform\mytrafo  
    
\matrix[row sep=-1cm, column sep=3cm,execute at begin cell=\pgfsettransform\mytrafo]
{
&
\coordinate (limit1) at (-.2,0.2); 
 \draw[->] (0,0) -- (1.2,0) node[right]{$x$};
 \draw[->] (0,0) -- (0,1.2);
 \draw (1,-\dash) node[below]{1} -- (1,\dash);
 \draw (-\dash,1) node[left]{1} -- (\dash,1);
 
 \draw (0,1) -- (1,1) node[above]{$f_0$};
\\ 
\coordinate (fh1) at (1.5,.7);
\coordinate (fh2) at (1.5,.3);

 \draw[->] (0,0) -- (1.2,0) node[right]{$x$};
 \draw[->] (0,0) -- (0,1.2);
 \draw (1,-\dash) node[below]{1} -- (1,\dash);
 \draw (-\dash,1) node[left]{1} -- (\dash,1);

\node[above] at (1,1){$f_h$};

\foreach  \h in {.1, .01}
{
 \draw[smooth,samples=100,domain=0:1] plot( {\x},{\x/(\x+\h)} );
};
\\
&
\coordinate (limit2) at (-.2,.8);
 \draw[->] (0,0) -- (2,0) node[right]{$X$};
 \draw[->] (0,0) -- (0,1.2);
 \draw (-\dash,1) node[left]{1} -- (\dash,1);
 
\node[above] at (1.5,1){$g$};
 
  \draw[smooth,samples=100,domain=0:1.5] plot( {\x},{\x/(\x+.1)} );
  \draw[dashed,samples=100,domain=1.5:2] plot( {\x},{\x/(\x+.1)} );

 \\
}; 
\draw[->] (fh1) -- node[above,sloped]{$h\to0$}(limit1);
\draw[->] (fh2) -- node[above,sloped]{$h\to0$} node[below,sloped]{\small rescaled $x=hX$}(limit2);
\end{tikzpicture}
 \caption{Graph of  $f_h$ for $h=0.1$ and $h=0.01$, and limits at two scales}
 \label{fig:fh graph}
\end{figure}
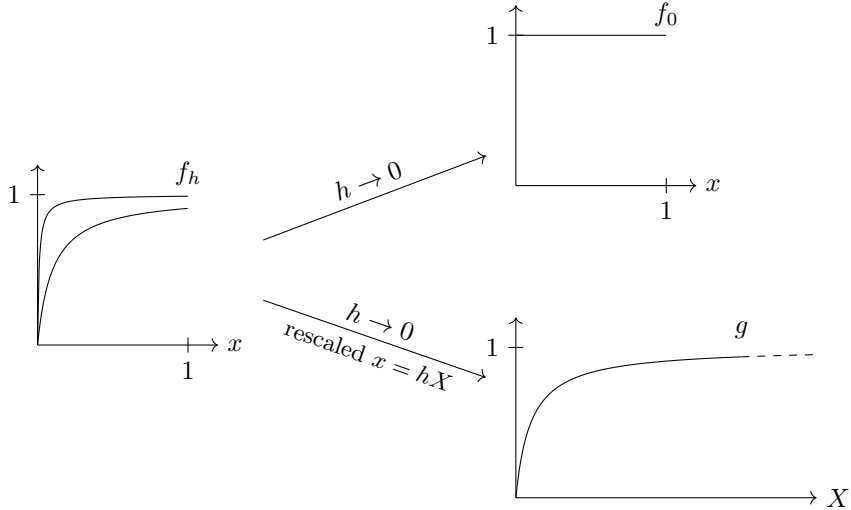

\medskip

\noindent{\em Geometric resolution analysis and matched asymptotic expansions.}
This rough first explanation of scales will be made more precise in Section \ref{sec:intro mwc}. But let us now turn to real problems: Consider a differential equation whose coefficients depend on a parameter $h$, and have non-trivial behavior at several scales as $h\to0$. We then ask how the solutions behave as $h\to0$. Of course we expect them to exhibit several scales also.\footnote{Although it is not essential for this article, as a warm-up exercise you may analyze the behavior as $h\to0$ of the solution of the differential equation $u'+f_h(x)u=0$, $u(0)=1$, or (more difficult) of $u''+f_h(x)u=0$, $u(0)=0,u'(0)=1$.}
The same phenomenon arises in so-called singular perturbation problems, where the type of the equation changes at $h=0$.\footnote{As an example, consider the equation $hu'+u=0$, $u(0)=1$. At $h=0$ this is not even a differential equation! For $h>0$ it has the solution $u_h(x)=e^{-x/h}$, which exhibits scaling behavior as $h\to0$ similar to $f_h$.}
Similarly, we could think of a partial differential equation on a domain which depends on $h$ and has parts that scale in different ways, or which degenerates to a lower-dimensional domain as $h\to0$, or both.
For examples see Figures \ref{fig:adiab var domain} and \ref{fig:example OmegaL}.
Such problems arise frequently in global and geometric analysis as well as in applied analysis (sometimes under the name of boundary layer problems). We will call them {\em singular problems.}

There is a standard method to attack such problems, called {\em matched asymptotic expansions} (MAE) and commonly used in applied analysis since the mid-1900s (see e.g. \cite{Hol:IPM}): 
Roughly speaking, for each scale appearing in the problem you make an ansatz for the Taylor expansion (in $h$) of the solution at this scale and plug it into the equation. This yields  recursive sets of equations for the Taylor coefficients. The fact that the solutions at different scales must \lq fit together' yields boundary conditions that make these equations well-posed (and often explicitly solvable). 

Of more recent origin is a different but closely related method, which has been used frequently in global and geometric analysis and which we call {\em geometric resolution analysis}\footnote{As far as I know, no name has been coined for the method in the literature. This name must not be confused with the so-called geometric multi-resolution analysis, a method for the analysis of high dimensional data.} (GRA):
the starting point is a shift in perspective, which in the example above is to consider $f:(x,h)\mapsto f_h(x)$ as a function of two variables rather than as a family of functions of one variable.  Then $f$ has singular behavior at $(x,h)=(0,0)$, and the scaling considerations above can be restated as saying that this singularity can be {\em resolved by blowing up} the point $(0,0)$ in $(x,h)$-space, as will be explained in Section \ref{sec:intro mwc}. In order to analyze the solutions of a singular differential equation we first resolve its singularities by suitably blowing up $(x,h)$-space; then the asymptotic behavior of solutions is obtained by solving model problems at the $h=0$ boundary faces of the blown-up space. The model problems are simpler than the original problem and correspond to the recursive sets of equations of MAE.

\medskip

\noindent{\em Eigenfunctions and quasimodes.} The purpose of this article is to introduce the concepts needed for geometric resolution analysis and apply them to problems in spectral theory. The needed concepts are manifolds with corners, blow-up and resolution. The spectral problem is to analyze solutions $\lambda\in\R$, $u:\Omega\to\R$ of the equation
$$ -\Delta u = \lambda u$$
where $\Delta$ is the Laplacian on a bounded domain $\Omega\subset\R^2$, and the Dirichlet boundary condition $u=0$ at $\partial\Omega$ is imposed. This problem has natural generalizations to higher dimensions, manifolds and other boundary conditions, some of which will occasionally also be considered. The eigenvalues form a sequence $0<\lambda_1\leq\lambda_2\leq\dots\to\infty$ and can usually not be calculated explicitly. But if we look at families of domains $\Omega_h$ which degenerate to a line segment as $h\to0$ then we have a chance to analyze the asymptotic behavior of $\lambda_k(h)$ (and associated eigenfunctions) as $h\to0$. Here we fix $k$ while letting $h\to0$. Other regimes are also interesting, e.g. $k$ going to $\infty$ like $h^{-1}$, but we don't consider them here. One expects that the leading term in the asymptotics can be calculated by solving a one-dimensional (ODE) problem. This is indeed the case also for higher order terms, but the details of how this works depend crucially on {\em how} $\Omega_h$ degenerates (the \lq shape' of $\Omega_h$). 
We will analyze several interesting cases of such degenerations. 

A standard approach to analyzing such eigenvalue problems is to first construct so-called {\em quasimodes}, i.e.\ pairs $(\lambda,u)$ which solve the eigenvalue equation up to a small, i.e. $O(h^N)$, error, and then to show that the quasimodes are close to actual solutions.
The construction yields the full asymptotics (i.e. up to errors $O(h^N)$ for any $N$) of quasimodes, and then of actual eigenvalues and eigenfunctions as $h\to0$.
It is in the construction of quasimodes where GRA (or MAE) is used, and we will focus on this step in this article. The second step is quite straight-forward if the operator is scalar and the limit problem is one-dimensional, as is the case for all problems considered here. See Remark \ref{rem:first vertical mode},  \cite{FriSol:SDLNS}, \cite{GriJer:AEPD}, \cite{Our:DEAFT} and point (V) below. For higher dimensional limit problems quasimodes need not be close to modes, see \cite{Arn:MQ}.

\medskip

\noindent{\em Why GRA?}
The methods of geometric resolution analysis and matched asymptotic expansions are closely related: they are really different ways to encode the same calculational base. GRA requires you to learn and get used to some new concepts, like manifolds with corners and blow-up, while MAE is very \lq down-to-earth'. Here are some points why it may be worth to invest the effort to learn about GRA. I hope they will become clear while you read this article.
\begin{enumerate}[(I)]
 \item
 GRA provides a rigorous framework for the powerful idea of MAE. For example, the \lq expansions at different scales' of a putative solution $u(x,h)$ are simply Taylor expansions at different faces of $u$ when considered on (i.e. pulled back to) the blown-up space.
 \item
 GRA provides conceptual clarity. In GRA the \lq singular' aspects of a problem are dealt with in the geometric operation of blow-up. Then the analysis (solution of differential equations) is reduced to non-singular model problems, and to a version of the standard Borel lemma. 
In this way essential structures of a problem are clearly visible, while notationally messy (but essentially trivial) calculations involving multiple Taylor series  
run invisibly in the background.
This also helps to identify common features of seemingly different problems.
 \item
 GRA helps to stay sane in complex settings. Often more than two model problems appear, and remembering how they fit together (the \lq matching conditions' of MAE) may be a torturous task. In GRA each model problem corresponds to a boundary face of the resolved space, and their relations can be read off from how these faces intersect. 
 \item
 GRA may guide the intuition. 
 The true art in solving singular problems is to identify the scales that can be expected to appear in the solutions.  
 The geometric way of thinking about singularities often helps to \lq see' how to proceed, see Section \ref{sec:adiabatic with ends} for a nice example. An added complication is that sometimes solutions exhibit more scales than the data (i.e.\ the coefficients or the  domains), as the setting in Section \ref{sec:adiab variable} shows. It is desirable to have systematic methods to find these. These are beyond the scope of this article however, and we refer to \cite{Mel:DAMWC}.
 \item
GRA can be refined to provide
a systematic way to extend modern PDE methods like the pseudodifferential calculus to singular problems, and embeds them in a larger mathematical framework, see \cite{Mel:APSIT}, \cite{Mel:DAMWC}, \cite{Gri:BBC}. In this way one may also carry out the second step mentioned above, proving that quasimodes are close to actual solutions, in the framework of GRA, by analyzing the resolvent on a blown-up double space, see e.g. \cite{MazMel:ALHCLSSF}.
\end{enumerate}

In this article we 
use simple examples to explain structures which also arise in more elaborate contexts. We consider planar domains and the scalar Laplacian, but the methods generalize without much extra work to manifolds and systems of elliptic PDEs (for example the Hodge Laplacian on differential forms). This is indicated at the end of each section. 
The methods can also be extended to study many other types of singular degenerations (with more work!), for example families of triangles degenerating to a line (ongoing work with R. Melrose, see also \cite{DauRay:PWCSAL}, \cite{Our:DEAFT}), domains from which a small ball is removed etc.

The results presented here are not new, and in some cases more precise or more general results have been obtained by other methods, as is indicated in the subsections on generalizations. 
In the PDE literature blow-up methods have mostly been used in the context of microlocal analysis. Our purpose here is to illustrate their use on a more elementary level, and to introduce a systematic setup for applying them to quasimode constructions. 
A minor novelty seems to be the use of the quasimode and remainder spaces $\calE(M), \calR(M)$ and their associated leading part maps, see Section \ref{sec:main steps} and Definitions \ref{def:adiab limit rem space}, \ref{def:adiab ends qm space} and \ref{def:adiab ends rem space}, although it is reminiscent of and motivated by the rescaled bundles used for example in  \cite{MazMel:ALHCLSSF}.

\subsection*{Outline of the paper}
In Section \ref{sec:intro mwc} we introduce the main objects of geometric resolution analysis (manifolds with corners, blow-up and resolution) and explain how they relate to the idea of scales. 
If you are mostly interested in quasimode constructions it will suffice to skim this section and only use it for reference; however, for Section \ref{sec:adiabatic with ends} more of this material will be needed.
In the remaining sections we show how quasimodes can be constructed using geometric resolution analysis. 
The examples are ordered to have increasing complexity, so that later examples use ideas introduced in previous examples plus additional ones. 
For easier reading the main steps of the constructions are outlined in  Section \ref{sec:main steps}.
To set the stage, we first consider regular perturbation problems in Section \ref{sec:regular pert}.
All further problems are eigenvalue problems on families of domains $\Omega_h$ which degenerate to a line segment as $h\to0$. 
Such problems are sometimes called \lq adiabatic limit problems\rq.
The simplest setting for these, where the cross section has constant lowest eigenvalue, is considered in Section \ref{sec:adiab limit const}. 
The treatment is general enough to apply to fibre bundles with Riemannian submersion metrics. 
Variable eigenvalues of the cross section, which occur for example when $\Omega_h$ is an ellipse with half axes $1$ and $h$, will introduce new scales, and this is analyzed in Section \ref{sec:adiab variable}. Then in Section \ref{sec:adiabatic with ends} we consider a problem where $\Omega_h$ scales  differently in some parts than in others. Here it will be especially apparent how the geometric way of thinking guides us to the solution.
The quasimode results are formulated in Theorems \ref{thm:regular quasimodes}, \ref{thm:adiab quasimodes}, \ref{thm:adiabatic variable}, \ref{thm:adiab ends}.
In Section \ref{sec:summary} we summarize the main points of the various quasimode constructions.

\subsection*{Related literature}
The book \cite{Mel:DAMWC} (unfinished, available online) introduces and discusses in great generality and detail manifolds with corners and blow-ups and their use in analysis. The big picture is outlined in \cite{Mel:PDOCSL}. 
The focus in the present article is on problems depending on a parameter $h$, where singularities only appear as $h\to0$ (so-called singular perturbation problems). Closely related are problems which do not depend on a parameter but where the underlying space (or operator) is singular, and the methods of geometric resolution analysis can be and have been applied extensively in this context. A basic introduction to this is given by the author in \cite{Gri:BBC}, with applications to microlocal analysis, including many references to the literature. Other frameworks for manifolds with corners have been proposed, see for example \cite{Joy:MAC} and references there.

Blow-up methods have also been used in the context of dynamical systems, e.g. in celestial mechanics \cite{McG:SCCM}, for analyzing geodesics on singular spaces \cite{GraGri:EMCS} or in multiple time scale analysis, see for example \cite{DumRou:CCCM}, \cite{KueSzm:MGOMNCRO}, \cite{GilKruSzm:AEUB} and the book \cite{Kue:MTSD}, which gives an excellent overview and many more references.

The survey \cite{Gri:TTMPGASG} discusses various types of \lq thin tube\rq\ problems including the ones discussed here; their origin as well as various methods and results are explained.
The books \cite{MazNazPla:ATEBVPSPDI}, \cite{MazNazPla:ATEBVPSPDII} discuss many singular perturbation problems of geometric origin and their solution by a method called \lq compound asymptotic expansions\rq\ there, which is similar to matched asymptotic expansions.

More references are given at the end of each section.

\subsection*{Acknowledgements}
These notes are based on a series of lectures that I gave at the summer school \lq Geometric and Computational Spectral Theory\rq\ at the Centre de Recherches Math\'ematiques in Montreal. I am grateful to the organizers of the school for inviting me to speak and for suggesting to write lecture notes. I thank Leonard Tomczak for help with the pictures and D. Joyce, I. Shestakov, M. Dafinger and the anonymous referee for useful comments on previous versions of these notes. My biggest thanks go to Richard Melrose for introducing many of the concepts discussed here, and for many inspiring discussions.

%

\section{A short introduction to manifolds with corners and resolutions}
\label{sec:intro mwc}
In this section the basic concepts of geometric resolution analysis are introduced: manifolds with corners, polyhomogeneous functions, blow-up, resolutions. 
We emphasize ideas and introduce concepts mostly by example or picture (after all, we are talking geometry here!), hoping that the interested reader will be able to supply precise definitions and proofs herself, if desired. Many details can be found in \cite{Mel:APSIT} and \cite{Mel:DAMWC}.

To see where we're heading consider the example from the introduction:
$$ f(x,h) = \frac{x}{x+h}\,,\quad x,h\geq0,\ (x,h)\neq(0,0)\,.$$
Recall its $h\to0$ limits at two scales:
\begin{equation*}
f_0(x)= \lim_{h\to0}f(x,h)  = 
\begin{cases}
 0&(x=0)\\
 1&(x>0)
\end{cases}
,\qquad
g(X)= \lim_{h\to0} f(hX,h) =  \frac{X}{X+1}\,,
\end{equation*}
see Figure \ref{fig:fh graph}. The \lq geometry\rq\ (of geometric resolution analysis) resides in the spaces on which these functions are defined, i.e.\ their domains:
$$ \dom f = \R_+^2,\ \dom f_0 = \R_+,\ \dom g = [0,\infty]$$
where
$$ \R_+ := [0,\infty)\,.$$
See Figure \ref{fig:domains}.
Actually, $f$ is not defined at $(0,0)$, but we ignore this for the moment.
For $g$ we have added $\infty$ to its domain, where we set $g(\infty)=\lim_{X\to\infty}g(X)=1$. We'll see in a moment why this makes sense.

\begin{figure}
 
\centering
\begin{tikzpicture}

\begin{scope}[scale=2]

\draw[fill=gray!20,draw=white]
   (0,0) rectangle (1,1);
   
\draw[->] (0,0) -- (1,0) node[right]{$x$};
\draw[->] (0,0) -- (0,1) node[left]{$h$}; 
\draw[dashed] (30:0) -- (30:.7) node[right]{$R_2$};
\draw[dashed] (45:0) -- (45:.7);
\node at (.7,.6) {$R_1$};

\node at (.5,1.2){$\dom f$};
 
\end{scope}

\node at (4,1){$\stackrel{\beta}\longleftarrow$};

\begin{scope}[xshift=6cm, scale=2]
 
\draw[fill=gray!20,draw=white]
   (0,0) rectangle (1,1)
   (0,0) -- (\bupr,0) arc (0:90:\bupr) -- (0,0);
   
\draw (1,0) -- (\bupr,0) arc (0:90:\bupr) -- (0,1); 
\draw[dashed] (30:\bupr) -- (30:\bupr+.7);
\draw[dashed] (45:\bupr) -- (45:\bupr+.7);

\node at (.5,1.2){$[\dom f, (0,0)]$};

\def\y{.5}  
 \draw (-.8,-\y) -- (-.4,-\y) node[below]{$\dom g$} -- (0,-\y);
 \draw[fill] (-.8,-\y) circle (.3pt);
 \draw[fill] (0,-\y) circle (.3pt);
 
 \draw (\bupr,-\y) -- (.7,-\y) node[below]{$\dom f_0$} -- (1,-\y);
 \draw[fill] (\bupr,-\y) circle (.3pt);

 \draw[dotted,->] (-.4, -\y+.1) -- (\bupr/2,\bupr/2);
 \draw[dotted,->] (.7, -\y+.1) -- (.7,-.1);
\end{scope}

\end{tikzpicture}

\caption{Domains of $f$ and of its rescaled limits $f_0$ and $g$, and how they relate to each other. Dotted arrows mean \lq identify with\rq.}
\label{fig:domains}
\end{figure}
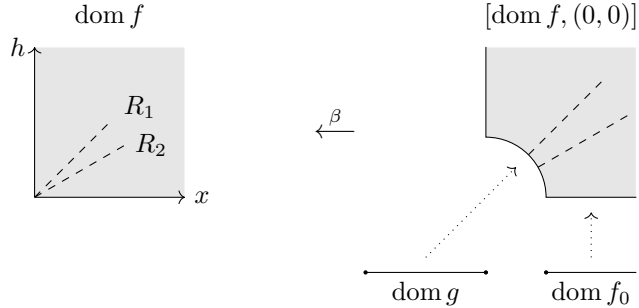
These spaces are simple examples of manifolds with corners.
We ask:
\begin{quote}
{\em Can we understand $f_0$ and $g$ as restrictions of $f$ to suitable subsets of its domain?} 
\end{quote}

For $f_0$ this is easy: if we identify $\dom f_0$ with the lower edge of
$\dom f$ then $f_0$ is simply the restriction of $f$.
(Again we should exclude the point $(0,0)$.)

How about understanding $g$ as a restriction of $f$? Can we reasonably identify $\dom g$ as a subset of $\dom f$?
This is less obvious.
Note that, for any $X\geq0$, $h\mapsto f(hX,h)$ is the restriction of $f$ to the ray
$$ R_X := \{(x,h)\in\dom f: x = hX,\ h>0\}\,. $$
By definition, $g(X)$ is the limit of this restriction as $h\to0$, so it should be the value of $f$ at the endpoint of $R_X$. This remains true for $X=\infty$ if we set $R_\infty=\{(x,0):x>0\}$.
Now this endpoint is $(0,0)$, so we have two problems: First,  $f$ is not defined there, and second, 
the endpoints of all rays $R_X$ (with different $X$) coincide.

That's why we don't find $\dom g$ in $\dom f$. But there is a way out, and this is {\em the idea of blow-up:
we simply add a separate endpoint for each ray $R_X$ to the picture}.

That is, we remove $(0,0)$ from $\dom f$ and replace it by a quarter circle as in Figure \ref{fig:domains}.
This produces a new space, denoted $[\dom f,(0,0)]$ and called the blow-up of $(0,0)$ in $\dom f$. A precise definition is given in Section \ref{subsec:blow-up}. It involves polar coordinates, and the quarter circle corresponds to $r=0$. We denote the quarter circle by $\ff$ (\lq front face\rq).
Each point of the blown-up space corresponds to a point of $\dom f$, as is indicated in Figure \ref{fig:domains} by the dashed rays.
We encode this by  a map
$$ \beta: [\dom f,(0,0)] \to \dom f$$
which maps $\ff$ to $(0,0)$ and is bijective between the complements of these sets.
Under this correspondence, $f$ translates into the function  $\beta^*f:=f\circ\beta$ on $[\dom f,(0,0)]$. Essentially, we will see that $\beta^*f$ is \lq $f$ written in polar coordinates\rq.
This simple construction solves all our problems:
\begin{itemize}
  \item
  $\beta^*f$ is defined on all of $[\dom f,(0,0)]$, including its full boundary. It is actually smooth, once we define what smoothness means on $[\dom f,(0,0)]$.
  \item
  If we identify $\ff$  with $[0,\infty]$ (the endpoint of the ray $R_X$ being identified with $X\in[0,\infty]$)
  then $g$ is the restriction of $\beta^*f$ to $\ff$.
  \item
  The pointwise limit $f_0(x)=\lim_{h\to0} f(x,h)$ is, for $x>0$, still the restriction of $f$ to the lower part of the boundary of $[\dom f,(0,0)]$. 
\end{itemize}
In addition, as we will see later, $\beta^*f$ also encodes how $f_0$ and $g$ relate to each other (so-called \lq matching\rq).

Summarizing, the multiple scales behavior of $f$ is completely encoded by the behavior of $\beta^*f$ near the boundary of $[\dom f,(0,0)]$, and different scales correspond to different segments (later called boundary hypersurfaces) of the boundary.

%
%
%
%

\subsection{Manifolds with corners}
Even if we wanted to study  problems on domains in $\R^n$ only, 
the natural setting for our theory is that of manifolds, for (at least) two reasons:
\begin{enumerate}
 \item Just as finite dimensional vector spaces are like $\R^n$ without choice of a basis, manifolds are locally like $\R^n$ without choice of a (possibly non-linear) coordinate system -- and foregoing such a choice leads to greater conceptual clarity. To put it more mundanely, it will be useful to use different coordinate systems (e.g. polar coordinates, projective coordinates), and it is reassuring to know that all constructions are independent of such choices.
 \item
Globally, a manifold represents how various local objects fit together -- 
and one of our goals is to fit different scales together.
In fact, even if the problem to be studied is topologically trivial, there may be non-trivial topology (or combinatorics) in the way that different scales relate to each other. 
\end{enumerate}
To get an idea what a manifold with corners is, look at Figure \ref{fig:mwc examples}. The most complicated specimen appearing in this text is on the right in Figure \ref{fig:blow-up}. 
\medskip

Recall that  a manifold is a space which can locally be parametrized by coordinates. For a manifold with corners some coordinates will be restricted to take only non-negative values. 
As before we use the notation
$$ \R_+ := [0,\infty)$$
and write $\R_+^k = (\R_+)^k$.
\begin{definition}
\label{def:mwc}
 A \defin{manifold with corners (mwc)} of dimension $n$ is a space $M$ which can locally be parametrized by open subsets of the \defin{model spaces} $\R_+^k\times\R^{n-k}$, for various $k\in\{0,\dots,n\}$.
 
 In addition, we require that  the boundary hypersurfaces be embedded, as  explained below.
\end{definition}
The model space condition is meant as in the standard definition of manifolds, for which only $k=0$ is allowed.
So for each point $p\in M$ there is  $k\in\{0,\dots,n\}$ and a neighborhood $U$ of $p$ with a coordinate map $U\to \tilde U$, with $\tilde U\subset \R_+^k\times\R^{n-k}$ open, and it is required that  coordinate changes are smooth.\footnote{\label{footnote: smooth}
Open means relatively open, that is, there is an open subset $\tilde U'\subset\R^n$ with $\tilde U' \cap (\R_+^k\times\R^{n-k})=\tilde U$. 
For example, $[0,1)$ is open in $\R_+$. 
A 
\defin{smooth function} on an open subset $\tilde U\subset\R_+^k\times\R^{n-k}$ is a function which extends to a smooth function on such a $\tilde U'$. A map $\tilde U\to \R_+^k\times\R^{n-k}$ is smooth if each component function is smooth. The space of smooth functions on $M$  (which are sometimes called \lq smooth up to the boundary\rq\ for emphasis) is denoted by $\Cinf(M)$.
}
The smallest $k$ which works for a fixed $p$ is called the \defin{codimension} of $p$.
See Figures \ref{fig:mwc examples}, \ref{fig:mwc non-examples} for some examples and non-examples of mwc.

The set of points of codimension 0 is the interior $\interior{M}$ of $M$. The closure of a connected component of the set of points of codimension $k$ is called a \defin{boundary hypersurface (bhs)} if $k=1$, and a \defin{corner of codimension $k$} if $k\geq2$. 
So the examples in Figure \ref{fig:mwc examples} have 1, 2, 3, 2 boundary hypersurfaces. 

It is clear that each boundary hypersurface itself satisfies the local model condition, with $n$ replaced by $n-1$. However, as in the example on 
the right in Figure \ref{fig:mwc non-examples}, it may happen that a boundary hypersurface \lq intersects itself\rq, 
that is, it is an immersed rather than an embedded submanifold (with corners). 
So according to our definition it is not a manifold with corners.

The embeddedness requirement is equivalent to the existence of a \defin{boundary defining function}
for each bhs $H$, i.e. a smooth function $x:M\to\R_+$ which vanishes precisely on $H$ and whose differential at any point of $H$ is non-zero. 
A boundary defining function $x$ can be augmented to a  \defin{trivialization near $H$}, i.e. an identification of  a neighborhood $U$ of $H$ with $[0,\eps)\times H$ for some $\eps>0$, where $x$ is the first component and each $y\in H\subset U$ corresponds to $(0,y)$.

Each bhs and each corner of a mwc $M$ is a mwc. But if $M$ has corners then its full boundary is not a manifold with corners.

Some authors, e.g. D.\ Joyce \cite{Joy:GMWC}, define manifolds with corners without the embeddedness condition on boundary hypersurfaces. Also, Joyce defines the notion of boundary of a mwc differently, so that it is also a mwc.

\begin{figure}
\begin{tikzpicture}
\matrix[column sep=1cm]
{
\draw[fill] (0,0) node[left]{\footnotesize 1} circle (1pt) -- (1,0) node[above]{\footnotesize 0} -- (2,0);

\node at (1,-.5){$\R_+$};
&
\fill[gray!20] (0,0) rectangle (2,2);
\draw[fill] (0,0) node[left]{\footnotesize 2} circle (1pt) -- (1,0) node[above]{\footnotesize 1} -- (2,0);
\draw (0,0) -- (0,1) node[left]{\footnotesize 1} -- (0,2);
\node at (1,1) {\footnotesize 0};

\node at (1,-.5){$\R_+^2$};
&
\fill[fill=gray!20] (0,0) rectangle (2,2)
                         (0,0) -- (1,0) arc (0:90:1) -- (0,0);  
\draw (1,0) arc (0:90:1);
\draw[fill] (1,0) node[left]{\footnotesize 2} circle (1pt) -- (1.8,0) node[above]{\footnotesize 1} -- (2,0);
\draw[fill] (0,1) node[left]{\footnotesize 2} circle (1pt) -- (0,1.8) node[left]{\footnotesize 1} -- (0,2);
\node at (.6,.6){\footnotesize 1};
\node at (1.2,1.2){\footnotesize 0};
&
\draw[fill] (0,1) node[left]{\footnotesize 2} circle (1pt)  (2,1) node[right]{\footnotesize 2} circle (1pt);
\node at (1,.6){\footnotesize 1};
\node at (1,1.4){\footnotesize 1};

\filldraw[fill=gray!20,draw=black] (0,1) arc (240:300:2);
\filldraw[fill=gray!20,draw=black] (2,1) arc (60:120:2);
\node at (1,1){\footnotesize 0};

\\
};
 
\end{tikzpicture}
 \caption{Examples of manifolds with corners, with codimensions of points indicated}
 \label{fig:mwc examples}
\end{figure}
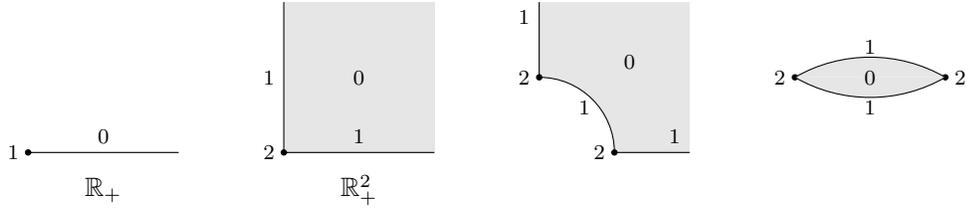
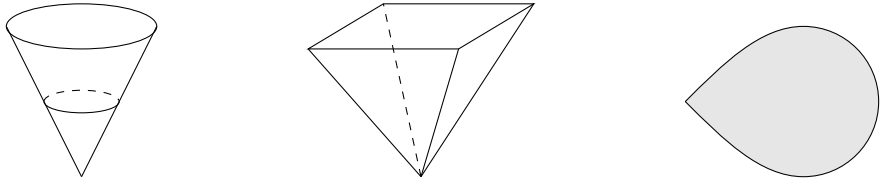
\begin{figure}
\begin{tikzpicture}
\matrix[column sep=2cm]
{
\begin{scope}[x={(1cm,0)},y={(0cm,-.3cm)},z={(0,1cm)}]
 \draw (0,0,0) -- (-1,0,2);
 \draw (0,0,0) -- (1,0,2);

 \draw[canvas is xy plane at z=2] (0,0) circle (1);
 \draw[canvas is xy plane at z=1,dashed] (-.5,0) arc (-180:0:.5);
 \draw[canvas is xy plane at z=1] (.5,0) arc (0:180:.5);

\end{scope}
&
\begin{scope}[x={(1cm,0)},y={(-.5cm,-.3cm)},z={(0,1cm)}]
 \draw (0,0,0) -- (1,1,2);
 \draw (0,0,0) -- (-1,1,2);
 \draw (0,0,0) -- (1,-1,2);
 \draw[dashed] (0,0,0) -- (-1,-1,2);

 \draw (1,1,2) -- (1,-1,2) -- (-1,-1,2) -- (-1,1,2) -- (1,1,2); 
\end{scope}
&
\begin{scope}[yshift=1cm]
\draw[fill=gray!20] (-1.5707,0) sin (0,-1) arc (-90:90:1)  cos (-1.5707,0); 
 
\end{scope}

\\
};
 
\end{tikzpicture}
 \caption{Not manifolds with corners. The cone and pyramid are understood as 3-dimensional bodies. The teardrop satisfies the local condition of a mwc, but the boundary line is not embedded.}
 \label{fig:mwc non-examples}
\end{figure}

Taylor's theorem implies the following simple fact which we need later. 
\begin{lemma} \label{lem:divide by bdf}
Let $M$ be a manifold with corners and $\calS$ a finite set of boundary hypersurfaces of $M$. Let $h$ be a total boundary defining function for $\calS$, i.e.\ the product of defining functions for all $H\in\calS$.

Then any $u\in\Cinf(M)$ which vanishes at each $H\in\calS$ can be written as $u=h\utilde$ with $\utilde\in\Cinf(M)$. 
\end{lemma}
Exercise: Prove this. Show that the analogous statement would not be true for the pyramid in Figure \ref{fig:mwc non-examples}.%
\footnote{If you understand this then you understand one of the main points about manifolds with corners!}


\begin{remark}
{\bfseries The corners of a mwc should not be considered as a problem, but as (part of) a solution}
 -- of all kinds of problems involving singularities.
 They should not be thought of as corners in a metric sense, only in a differential sense (i.e. some coordinates are $\geq0$). 
  
 For example, suppose you want to analyze the behavior of harmonic functions near the vertex of $\R^2_+$ or of a cone or of the pyramid in  Figure \ref{fig:mwc non-examples} (where the Laplacian is the standard Laplacian for the Euclidean metric on these spaces).
 The essential first step towards a solution would be to introduce polar coordinates around the vertex, and in the case of the pyramid also cylindrical coordinates around the edges. 
Geometrically this corresponds to the operation of blow-up, discussed below. This results in manifolds with corners. 
The fact that the original (metric) $\R^2_+$ happens to be a mwc also is irrelevant.

\end{remark}
\begin{remark} 
\label{rem:mwc divisors}
 Manifolds with corners are an oriented analogue of 
manifolds with normal crossings divisors as used in real algebraic geometry. \lq Oriented\rq\ means that the boundary hypersurfaces, which correspond to the components of the divisor, have a relative orientation, i.e. possess a transversal vector field.
The use of manifolds with corners allows for greater flexibility 
in many analytic problems.
See also Remark \ref{rem:blow-up alg geom}.
\end{remark}

\subsection{Polyhomogeneous functions}
All functions we consider will be smooth in the interior of their domains. 
Our interest will lie in their boundary behavior -- partly because we have a much better chance to analyze their boundary behavior than their interior properties. Functions smooth up to the boundary
(see Footnote \ref{footnote: smooth})
have the following important properties:
\begin{itemize}
 \item[1.]
A smooth function on $\R_+$ has a Taylor expansion $f(x)\sim \sum_{k=0}^\infty a_k x^k$ as $x\to0$, i.e. at $\partial\R_+$. 
\item[2a.]
A smooth function on $\R^2_+$ has Taylor expansions
\begin{equation}
\label{eqn:smooth expansions} 
 f(x,y) \sim \sum_{k=0}^\infty a_k(y)x^k\ \text{as }x\to0,\quad
f(x,y) \sim \sum_{l=0}^\infty b_l(x)y^l\ \text{as }y\to0
\end{equation}
at the boundary hypersurfaces $x=0$ and $y=0$ of $\R^2_+$,
with $a_k,b_l$ smooth on $\R_+$.
\item[2b.] (Matching) For each $k,l\in\N_0$ the  $l$-th Taylor coefficient of $a_k$ at $y=0$ equals (\lq matches\rq) the $k$-th Taylor coefficient of $b_l$ at $x=0$. This corresponds to the Taylor expansion of $f$ at the corner $(0,0)$.
 \item[2c.] (Borel lemma) Conversely, given $a_k,b_l\in \Cinf(\R_+)$ satisfying these matching (or compatibility) conditions for all $k,l$, there is a function $f\in\Cinf(\R^2_+)$ satisfying
  \eqref{eqn:smooth expansions}, and it is unique modulo functions vanishing to infinite order at the boundary of $\R^2_+$.
\end{itemize}
It turns out that requiring smoothness up to the boundary  is too restrictive for many purposes. The class of {polyhomogeneous} functions\footnote{These are called \lq nice functions\rq\ in \cite{Gri:BBC}.} is obtained by replacing the powers $x^m$, $m\in\N_0$  in these expansions by terms $x^z \log^jx$ where $z\in\C$ and $j\in\N_0$, and is big enough for many problems.

Apart from this, polyhomogeneous functions enjoy the analogous properties as listed above.
Properties 2b. and 2c. will be essential for our purpose of analyzing multiple scale solutions of PDEs.

\subsubsection{Definition and examples}
We will define the space of polyhomogeneous functions on a manifold with corners $M$. 
The essence of the definition can be grasped from two special cases: $M=\R_+\times\R^n$ where $n\in\N_0$ and $M=\R_+^2$. The terms permitted in an expansion are characterized by a set  $E\subset \C\times\N_0$ satisfying
\begin{equation}
 \label{eqn:index set loc fin}
 \{(z,j)\in E:\ \Re z \leq r\}\quad\text{is finite for every }r\in\R\,.
\end{equation}
This guarantees that the expansion \eqref{eqn:phg def2} below makes sense.
\begin{definition}
A \defin{polyhomogeneous function} on $M=\R_+\times\R^n$ or $M=\R_+^2$ is a smooth function $u$ on $\interior{M}$ satisfying:
\begin{enumerate}[(a)]
 \item For $M=\R_+\times\R^n$: $u$ has an asymptotic expansion
 \begin{equation}
\label{eqn:phg def2}
 u(x,y) \sim \sum_{(z,j)\in E} a_{z,j}(y)\, x^z \log^j x \quad\text{ as }x\to 0
\end{equation}
for each $y\in\R^n$, for a set $E$ as above, where each $a_{z,j}\in\Cinf(\R^n)$. 

The set of these functions with $E$ fixed is denoted $\calA^E(\R_+\times\R^n)$.
\smallskip

\item
For $M=\R_+^2$: $u$ has an asymptotic expansion
\eqref{eqn:phg def2} for each $y>0$, where each
$a_{z,j}\in\calA^F(\R_+)$, for sets $E,F\subset\C\times\N_0$ satisfying \eqref{eqn:index set loc fin}.
\\
Also, the same condition is required to hold with $x,E$ and $y,F$  interchanged.

The set of these functions with $E,F$ fixed is denoted $\calA^{E,F}(\R_+^2)$.
\end{enumerate}
By definition, we understand  asymptotic expansions always \lq with derivatives\rq, i.e.   $\partial_x u$ has the asymptotic series 
 with each term differentiated, and similarly for $\partial_yu$ and higher derivatives. In addition, certain uniformity conditions are required.
\end{definition}
All asymptotic expansions occuring in the problems in this article have no logarithms, so $E\subset\C\times\{0\}$.\footnote{%
\label{footnote:logs}
However, logarithms are included in the definition since they appear in the solutions of many differential equations even if they don't appear in their coefficients. For example, the equation
$u' = f_h, \ u(0)=0 $
with $f_h$ as in \eqref{eqn:fh def} has solution $u_h(x)=x-h\log\left(\frac xh + 1\right)$ which for fixed positive $x$ has the expansion
$$u_h(x) \sim x + h\log h -h\log x + O(h^2)$$
as $h\to0$. The appearance of the log term here can be predicted without calculating integrals, using geometric resolution analysis via the push-forward theorem of Melrose \cite{Mel:CCDMWC}, as is explained in \cite{Gri:BBC} for the related example where $f_h(x)=\sqrt{x^2+h^2}$, see also \cite{GriGru:SALPFT}.
}

The \lq asymptotics with derivatives\rq\ condition is equivalent to
\begin{equation}
 \label{eqn:asymp precise}
 \left| (x\partial_x)^\alpha \partial_y^\beta\left( u(x,y) - \sum_{(z,j)\in E,\Re z\leq r} a_{z,j}(y)\, x^z \log^j x\right) \right| \leq C_{r,\alpha,\beta} \,x^r
\end{equation}
for all $r\in\R$ and all $\alpha\in\N_0,\beta\in\N_0^n$. Here $C_{r,\alpha,\beta}$ may depend on $y$. 
For $M=\R_+\times\R^n$ the local uniformity condition is that for any compact $K\subset\R^n$  the same constant can be chosen for all $y\in K$. \\
For $M=\R_+^2$ this is required for all compact $K\subset(0,\infty)$, plus a local uniformity near $(x,y)=(0,0)$: there is $N\in \R$ so that estimate \eqref{eqn:asymp precise} holds for all $y\in(0,1)$, with $\partial_y$ replaced by $y\partial_y$ and $C_{r,\alpha,\beta}$ by $C_{r,\alpha,\beta} y^{-N}$.

We now give examples and then formulate the general definition. 

\medskip
\begin{examples}\mbox{}
\label{ex:phg fcns}
\begin{enumerate}
 \item $u(x)=\frac1x$ is in $\calA^E(\R_+)$ for $E=\{(-1,0)\}$.
 \item If $E=F=\N_0\times\{0\}$ then $u\in\calA^{E,F}(\R^2_+)$ if and only if $u$ extends smoothly to the boundary of $\R^2_+$.\footnote{Exercise: prove this.}
 \item \label{ex:phg fcn not}
 $u(x,y) = \frac x{x+y}$ is smooth on $\R^2_+\setminus\{(0,0)\}$, but not polyhomogeneous (for any index sets) on $\R^2_+$. To see this, we expand $u$ as $x\to0$ for fixed $y>0$:
\begin{equation}
\label{eqn:f expansion 2} 
 u(x,y) = \frac x{x+y} = \frac1{y}\frac x{1+\frac{x}{y}} = \frac1{y}x - \frac1{y^2}x^2 + \frac1{y^3}x^3-+\cdots
\end{equation}
 We see that $u$ has an expansion as in \eqref{eqn:phg def2}, but the coefficients $a_{k,0}(y)=(-1)^k y^{-k}$ become more and more singular (for $y\to0$) as $k$ increases, so there is no index set $F$ for which all coefficients lie in $\calA^F(\R_+)$.

Note that this is precisely our first example 
\eqref{eqn:fh def}.
\end{enumerate}
\end{examples}

A set $E\subset\C\times\N_0$ satisfying \eqref{eqn:index set loc fin} and in addition $(z,j)\in E, l\leq j\Rightarrow (z,l)\in E$ is called an \defin{index set}. This condition guarantees that $\calA^E(\R_+\times\R^n)$ is invariant under the operator $x\partial_x$. If, in addition, 
$(z,j)\in E\Rightarrow (z+1,j)\in E$ then $E$ is called a \defin{smooth (or $\Cinf$) index set}. This guarantees coordinate independence, i.e.\ any self-diffeomorphism of $\R_+\times\R^n$ preserves the space 
$\calA^E(\R_+\times\R^n)$. The index set $E$ in Example \ref{ex:phg fcns}.1 is not smooth; the smallest smooth index set containing $E$ is $\{-1,0,1,\dots\}\times\{0\}$.  
\smallskip 

We now consider general manifolds with corners. Of course we want to say a function is polyhomogeneous if it is so in any coordinate system. Since we want to allow corners of higher codimension, we give an inductive definition.

 An
\defin{index family}
for $M$ is an assignment $\calE$ of a $\Cinf$ index set $\calE(H)$ to each boundary hypersurface $H$ of $M$.
Recall that there is a trivialization near each $H$, i.e. we may write points near $H$ as pairs $(x,y)$ where $x\in[0,\eps)$  and $y\in H$, for some $\eps>0$.

\begin{definition}
Let $M$ be a manifold with corners and $\calE$ an index family for $M$. A \defin{polyhomogeneous function} on $M$ with index family $\calE$ is a smooth function $u$ on $\interior{M}$ which has  an expansion as in \eqref{eqn:phg def2} at each boundary hypersurface $H$, 
in some trivialization near $H$, where 
 $E=\calE(H)$ and the functions $a_{z,j}$ are polyhomogeneous on $H$ with the induced index family for $H$.%
 \footnote{The index family $\calE$ for $M$ induces an index family $\calE_H$ for the mwc $H$ as follows: Any boundary hypersurface $H'$ of $H$ is a component of a set $H\cap G$ where $G$ is boundary hypersurface of $M$ uniquely determined by $H'$. Then we let $\calE_H(H') := \calE(G)$. We require $a_{z,j}\in\calA^{\calE_H}(H)$ for each $(z,j)\in\calE(H)$ and each $H$. If this is true in one trivialization then it is true in any other, since each $\calE(H)$ is a $\Cinf$ index set.  Local uniformity is also required, analogous to the explanation after equation \eqref{eqn:asymp precise}.

This definition is inductive over the highest codimension of any point in $M$.}
 
\noindent The set of these functions is denoted $\calA^\calE(M)$.
\end{definition}

Again, if $\calE(H)=\N_0\times\{0\}$ for all $H$ then $u\in \calA^\calE(M)$ if and only if $u$ extends to a smooth function on all of $M$.

\begin{remark}
  \label{rem:functions int}
In our terminology a
\lq polyhomogeneous function on a manifold with corners $M$\rq\ needs to be defined on the interior $\interior{M}$ only.
The terminology is justified since its behavior near the boundary is prescribed;  $\Cinf(\interior M)$ is the much larger space of functions without prescribed boundary behavior. More formally, $\calA^\calE$ defines a sheaf over $M$, not over $\interior M$. 
\end{remark}

\subsubsection{Matching conditions and Borel lemma}
\label{subsubsec:phg product}
A central point of polyhomogeneity is to have \lq product type\rq\ asymptotic expansions at corners. This is most clearly seen in the case of $\R^2_+$.
To ease notation we formulate this only for the case without logarithms.

\begin{lemma}[Matching conditions]
\label{lem:compatibility conditions}
 Let $E,F\subset\C\times\{0\}$ be index sets for $\R^2_+$.
 Suppose $u\in \calA^{E,F}(\R^2_+)$, and assume $u$ has expansions
\begin{equation}
 \label{eqn:u expansions}
\begin{aligned}
  u(x,y) &\sim \sum_{(z,0)\in E} a_{z}(y)\, x^z  &\text{ as }x\to 0\\
 u(x,y) &\sim \sum_{(w,0)\in F} b_{w}(x)\, y^w   & \text{ as } y\to 0
\end{aligned}
\end{equation}
where $a_{z}\in\calA^F(\R_+)$, $b_{w}\in\calA^E(\R_+)$ for each $(z,0)\in E$, $(w,0)\in F$.
Expand
\begin{equation}
 \label{eqn:a,b expansion}
 a_z(y) \sim \sum_{(w,0)\in F} c_{z,w}\, y^w,\quad
b_w(x) \sim \sum_{(z,0)\in E} c'_{z,w}\, x^z 
\end{equation}
as $y\to0$ resp. $x\to 0$. Then
\begin{equation}
 \label{eqn:compatibility condition}
 c_{z,w} = c'_{z,w} \quad\text{ for all }z,w.
\end{equation}
\end{lemma}
This has a converse, which is a standard result:
\begin{lemma}[Borel lemma]
\label{lem:borel}
 Let $E,F$ be as in the previous lemma, and assume that functions $a_z,b_w$ satisfying \eqref{eqn:a,b expansion} are given.
 
 If \eqref{eqn:compatibility condition} holds then there is $u\in \calA^{E,F}(\R^2_+)$ satisfying \eqref{eqn:u expansions}. It is uniquely determined up to errors vanishing to infinite order at the boundary.
\end{lemma}
This will be a central tool in our analysis since it allows us to construct approximate solution of a PDE from solutions of model problems.

\subsection{Blow-up and resolution}
\label{subsec:blow-up}
We now introduce blow-up, which is what makes the whole manifolds with corners business interesting.
Here are the most important facts about blow-up. They will be explained in this section:
\medskip

\begin{itemize}
\itemsep2mm
 \item {\bf Blow-up is a geometric and coordinate free way to introduce polar coordinates.}
 \item {\bf Blow-up serves to desingularize singular objects.}
 \item {\bf Blow-up helps to understand scales and transitions between scales -- and therefore to solve PDE problems involving different scales.}
\end{itemize}

We first explain the idea in the case of blow-up of 0 in $\R^2$ and then give the general definition in Subsection \ref{subsubsec:b-up def}. After discussing resolutions and projective coordinates we return to our motivating example \eqref{eqn:fh def} in Example \ref{ex:proj coords resol}. There is also a short discussion of quasihomogeneous blow-up, which occurs naturally in Section \ref{sec:adiab variable}.

\begin{figure}
\centering
\begin{tikzpicture}
\def\length{1.5}  
 
\begin{scope}  
\draw[fill=gray!20,draw=white] (-2,-2) rectangle (2,2);
\draw[fill] (0,0) circle (1pt);

\foreach \angle in {0,72,144,216,288}  
\draw[dashed] (0,0) -- (\angle:\length); 
 
\draw[dotted] (0,0) circle (.6);
 
\node at (0,-2.5) {$\R^2$};
\end{scope}

\node at (3,0) {$\xleftarrow{\beta}$}; 

\begin{scope}[xshift=6cm] 
 \draw[fill=gray!20,draw=white] (-2,-2) rectangle (2,2);
 \draw[fill=white] (0,0) circle (\bupr);

\foreach \angle in {0,72,144,216,288}  
{
\draw[dashed] (\angle:\bupr) -- (\angle:\length+\bupr); 
\draw[fill] (\angle:\bupr) circle (.7pt);
} 

\draw[dotted] (0,0) circle (.6+\bupr);

\node at (0,-2.5) {$[\R^2,0]$};
 
\end{scope}
\end{tikzpicture}

 \caption{Blow-up of 0 in $\R^2$, with a few rays (dashed) and a pair of corresponding circles (dotted) drawn; the white disk is not part of $[\R^2,0]$; its inner boundary circle is the front face
}
 \label{fig:def blow-up}
\end{figure}
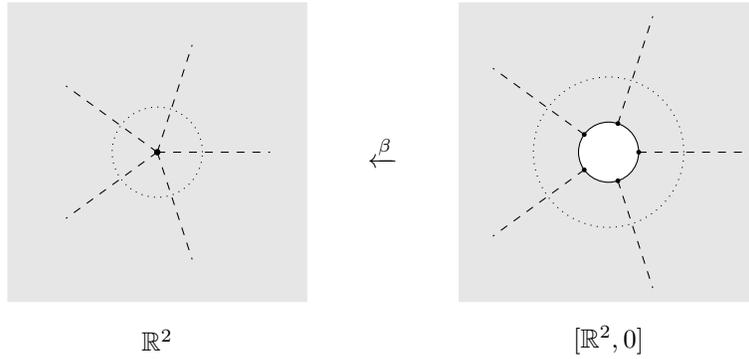

\begin{figure}
 \centering
\begin{tikzpicture}

\matrix[row sep=.2cm, column sep=1.5cm]
{

\draw[fill=gray!20,draw=white]
   (-1,0) rectangle (1,1)
   (\bupr,0) arc (0:180:\bupr);
   
\draw (1,0) -- (\bupr,0) arc (0:180:\bupr) -- (-1,0); 
\draw[dashed] (30:\bupr) -- (30:\bupr+.7);
& 
\draw[fill=gray!20,draw=white]
   (0,0) rectangle (1,1)
   (0,0) -- (\bupr,0) arc (0:90:\bupr) -- (0,0);
   
\draw (1,0) -- (\bupr,0) arc (0:90:\bupr) -- (0,1); 
\draw[dashed] (30:\bupr) -- (30:\bupr+.7);

&
\begin{scope}[x={(1cm,0)},y={(-.5cm,-.5cm)},z={(0,1cm)}]
 \draw (\bupr,0,0) -- (1,0,0);
 \draw (0,\bupr,0) -- (0,1,0);
 \draw (0,0,\bupr) -- (0,0,1); 
 
  \draw[dashed] (\bupr/1.732,\bupr/1.732,\bupr/1.732) -- (1,1,1);
 
 \draw[canvas is xy plane at z=0] (\bupr,0) arc (0:90:\bupr);
 \draw[canvas is yz plane at x=0] (\bupr,0) arc (0:90:\bupr);
 \draw[canvas is zx plane at y=0] (\bupr,0) arc (0:90:\bupr);
\end{scope}

&
\begin{scope}[x={(1cm,0)},y={(-.5cm,-.5cm)},z={(0,1cm)}]
 \draw (0,0,\bupr) -- (0,0,1);
 \draw (0,\bupr,0) -- (0,1,0);
 \draw (0,0,\bupr) -- (1,0,\bupr); 
 \draw (0,\bupr,0) -- (1,\bupr,0);
 
 \draw[canvas is yz plane at x=0] (\bupr,0) arc (0:90:\bupr);
 
 \draw[dashed] (.8,\bupr,\bupr) -- (.8,1+\bupr,1+\bupr);

\end{scope}

\\
\node at (0,0) {$\downarrow$};
&
\node at (0.5,0) {$\downarrow$};
&
\node at (0,0) {$\downarrow$};

&
\node at (0,0) {$\downarrow$};
\\

\draw[fill=gray!20,draw=white]
   (-1,0) rectangle (1,1);

\draw (1,0)  -- (-1,0); 
\draw[dashed] (30:0) -- (30:.7);

\draw[fill] (0,0) circle (1pt);
&
\draw[fill=gray!20,draw=white]
   (0,0) rectangle (1,1);
   
\draw (1,0) -- (0,0) -- (0,1); 
\draw[dashed] (30:0) -- (30:.7);

\draw[fill] (0,0) circle (1pt);
&
\begin{scope}[x={(1cm,0)},y={(-.5cm,-.5cm)},z={(0,1cm)}]
 \draw (0,0,0) -- (1,0,0);
 \draw (0,0,0) -- (0,1,0);
 \draw (0,0,0) -- (0,0,1); 
 
 \draw[dashed] (0,0,0) -- (1,1,1);
 
 \draw[fill] (0,0,0) circle (1pt);
\end{scope}

&
\begin{scope}[x={(1cm,0)},y={(-.5cm,-.5cm)},z={(0,1cm)}]
 \draw (0,0,0) -- (1,0,0);
 \draw (0,0,0) -- (0,1,0);
 \draw (0,0,0) -- (0,0,1); 
 
 \draw[very thick] (0,0,0) -- (1,0,0);
 
 \draw[dashed] (.8,0,0) -- (.8,1,1);
\end{scope}

\\ 
\node at (0,0) {(a)}; & 
\node at (0.5,0) {(b)}; & 
\node at (0,0) {(c)}; & 
\node at (0,0) {(d)};
\\
};
\end{tikzpicture}
\caption{Some examples of blow-up; 
in each bottom picture the submanifold being blown up is drawn fat, and the blown-up space is in the top picture. The vertical arrow is the blow-down map. The third and fourth example are 3-dimensional, and only the edges are drawn.}
\label{fig:blow-up examples}
\end{figure}
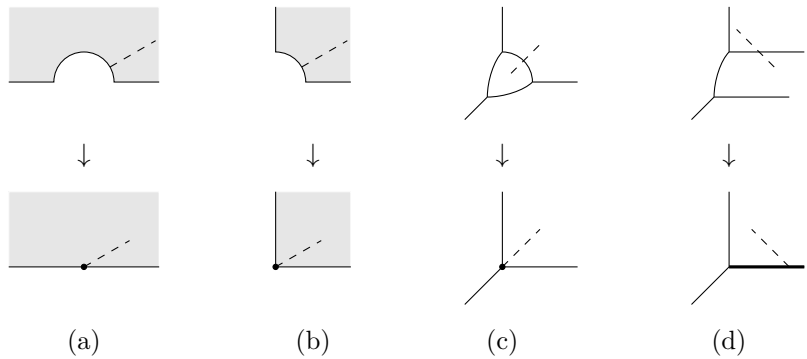

\subsubsection{The idea}
\label{subsubsec:b-up idea}
We first explain the idea in the case of blowing up the point 0 in $\R^2$, see Figure \ref{fig:def blow-up}:
Consider the set of rays (half lines) in $\R^2$ emanating from 0. They are pairwise disjoint except that they all share the common endpoint 0. 
{\em The \defin{blow-up of 0 in $\R^2$} is the space constructed from $\R^2$ by removing 0 and replacing it by one separate endpoint for each ray.} This space is denoted by $[\R^2,0]$. 
So we replace 0 by a circle, and each point on the circle corresponds to a direction of approach to 0. This circle is called the \defin{front face} of the blow-up.
Thus, blowing-up 0 in $\R^2$ means taking out $0$ from $\R^2$ and then choosing a new  \lq compactification at 0\rq\ of $\R^2\setminus0$, by adding the front face instead of $0$.
\medskip

Here is a concrete mathematical model realizing this idea: 
As a space we take $[\R^2,0] = \R_+\times S^1$, where $S^1=\{\omega\in\R^2:|\omega|=1\}$ is the unit circle.
The front face is $\ff := \{0\}\times S^1$. In Figure \ref{fig:def blow-up} the rays are the sets $\omega=$ const, the unbroken circle is $\ff$ and the dotted circle is $\{1\}\times S^1$.
We then need to specify how points of $[\R^2,0]$ correspond to points of $\R^2$. This is done using the map
$$\beta:[\R^2,0]\to\R^2,\quad \beta (r,\omega) = r\omega$$
called the \defin{blow-down map.}
Note that $\beta$ is a diffeomorphism from $(0,\infty)\times S^1$ to $\R^2\setminus\{0\}$; this means that it provides an identification of $[\R^2,0]\setminus\ff$ with $\R^2\setminus\{0\}$. 
The sets $\omega=$ const are mapped to rays, and two different such sets have different endpoints on $\ff$. All these endpoints are mapped to 0 by $\beta$.
%
Thus, this model and $\beta$ do precisely what they were supposed to do.

In addition, the model gives $[\R^2,0]$ a differentiable structure, making it a smooth manifold with boundary and $\beta$ a smooth map.
\smallskip

Note that if we parametrize $S^1$ by $\omega=(\cos\varphi,\sin\varphi)$ then $\beta$ is just the polar coordinates map 
\begin{equation}
\label{eqn:polar coord}
 (r,\varphi)\mapsto(x,y),\quad
 x= r\cos\varphi,\ y=r\sin\varphi\,.
\end{equation}
Recall that \lq polar coordinates on $\R^2$\rq\ are not coordinates at the origin. So $[\R^2,0]$ is the space on which polar coordinates are actual coordinates -- also at $r=0$. 

\begin{exercise}
 Show that points of $\ff$ correspond to directions at 0 not only of rays, but of any regular curve. That is: Let $\gamma:[0,1)\to\R^2$ be a smooth curve with $\gamma(0)=0$, $\dot\gamma(0)\neq0$ and $\gamma(t)\neq0$ for $t\neq0$. Show that there is a unique smooth curve $\tilde\gamma:[0,1)\to[\R^2,0]$ lifting $\gamma$, i.e.\ satisfying $\beta\circ\tilde\gamma=\gamma$, and that $\tilde\gamma(0) = \frac{\dot\gamma(0)}{\|\dot\gamma(0)\|}$.
\end{exercise}

\subsubsection{Definition and examples}
\label{subsubsec:b-up def}
The general operation of blow-up 
associates to any manifold $X$ and submanifold $Y\subset X$ a manifold with boundary, denoted $[X,Y]$, and a surjective smooth map $\beta:[X,Y]\to X$. 
We say that $[X,Y]$ is obtained from \defin{blowing up $Y$ in $X$} and call $\beta$ the \defin{blow-down map}.\footnote{For this to be defined $Y$ must have codimension at least one. We will always assume that the codimension is at least two, the other case being less interesting.}
$X,Y$ may also be manifolds with corners, then a local product assumption (see below) must be placed on $Y$, and $[X,Y]$ is a manifold with corners.
The preimage $\beta^{-1}(Y)$ is called the {\bf front face} $\ff$ of the blow-up. It is a boundary hypersurface of $[X,Y]$, and $\beta$ maps diffeomorphically $[X,Y]\setminus\ff\to X\setminus Y$. See Figure \ref{fig:blow-up examples} for some examples.

To define blow-up we use local models as in the previous subsection, but you should always keep the original idea of adding endpoints of rays in mind. 
We start with blow-up of an interior point, then  generalize this in two ways:  blow-up of a point on the boundary, and blow-up of a subspace (by taking products). Finally both generalizations are combined to yield the most general case.
\medskip

\noindent{\bf Definition of blow-up for the local models.}

Recall that a {\em model space} is a space of the form $\R^{n-k}\times\R_+^k$ (or $\R_+^k\times\R^{n-k}$). We consider these first.
\begin{enumerate}
 \item
 Blow-up of\footnote{To simplify notation we often write $0$ instead of $\{0\}$. Also $0$ denotes the origin in any $\R^k$.} $0$ in $\R^n$: Define 
$$ [\R^n,0] := \R_+\times S^{n-1}, \quad \beta(r,\omega)=r\omega$$
where $S^{n-1}=\{\omega\in\R^n:|\omega|=1\}$ is the unit sphere, $n\geq1$.

Note that $[\R^n,0]$ is a manifold with boundary.
  \item
Blow-up of $0$ in the upper half plane $\R\times\R_+$: Define
  $$[\R\times\R_+,0] := \R_+\times S^{1}_+,\ 
  \beta(r,\omega) = r\omega $$
  where $S^{1}_+ =S^{1}\cap (\R\times\R_+)$ is the upper half circle. See Figure \ref{fig:blow-up examples}(a).
  
  This is simply the upper half of case (1) with $n=2$. 
 This generalizes in an obvious way to the blow-up of zero in any model space:
 $$[\R^{n-k}\times\R_+^k,0] := \R_+\times S^{n-1}_k, \ \beta(r,\omega)=r\omega$$
 where $S^{n-1}_k:=S^{n-1}\cap (\R^{n-k}\times\R_+^k)$.
See Figure \ref{fig:blow-up examples}(b) for $n=k=2$ and Figure \ref{fig:blow-up examples}(c) for $n=k=3$.

Note that $[\R^{n-k}\times\R_+^k,0]$ has corners if $k\geq1$.
 \item
Blow-up of the $x$-axis $\R\times\{0\}$ in $\R^3$:
Define
  $$[\R^3,\R\times\{0\}]:=\R\times[\R^2,0],\ \beta(x,y,z) = (x,\beta_0(y,z))$$
  with $\beta_0:[\R^2,0]\to\R^2$ from case (1).
  So the line $Y=\R\times\{0\}$ is blown up to a cylinder, the front face of this blow-up.  
  Any point $p\in Y$ is blown up to a circle $\beta^{-1}(p)$. Points on the front face correspond to a pair consisting of a point $p\in Y$ and a direction of approach to $p$, modulo directions tangential to $Y$.

This generalizes in an obvious way to the blow-up of $\R^{n-m}\times\{0\}$ in $\R^n$:
$$[\R^n,\R^{n-m}\times\{0\}] = \R^{n-m}\times[\R^m,0]$$
(write $\R^n=\R^{n-m}\times\R^m$ and take out the common factor $\R^{n-m}$).
\item
Blow-up of $\R_+\times\{0\}$ in $\R_+^3$. Combining cases (2) and (3) we define
$$ [\R_+^3,\R_+\times\{0\}] = \R_+\times[\R_+^2,0]$$
see Figure \ref{fig:blow-up examples}(d).

The main point here is the product structure. In general, for model spaces $X,W,Z$,
\begin{equation}
 \label{eqn:model b-up}
\text{for } X=W\times Z,\ Y=W\times\{0\}\quad\text{define}\quad [X,Y]= W\times[Z,0] 
\end{equation}
with $[Z,0]$ defined in (2). In the example, $X=\R_+^3$, $W=\R_+$, $Z=\R_+^2$.
\end{enumerate}

\noindent{\bf Definition of blow-up for manifolds (possibly with corners).}
It can be shown (see \cite{Mel:RB}, \cite{Mel:DAMWC}) that these constructions are invariant in the following sense: for model spaces $X,Y$ as in case (4), any self-diffeomorphism of $X$ fixing $Y$ pointwise lifts to a unique self-diffeomorphism of $[X,Y]$.\footnote{In the case of $[\R^2,0]$ this can be rephrased as follows: let $x,y$ be standard cartesian coordinates and $r,\varphi$ corresponding polar coordinates. 
Let $x',y'$ be some other coordinate system defined near 0 (possibly non-linearly related to $x,y$), with $x'=y'=0$ corresponding to the point $0$. Define polar coordinates in terms of $x',y'$, i.e. $x'=r'\cos\varphi'$, $y'=r'\sin\varphi'$. Then $(r,\varphi)\mapsto(r',\varphi')$ is a smooth coordinate change on $[\R^2,0]$. 

It is in this sense that blow-up is a coordinate free way of introducing polar coordinates: the result does not depend on the (cartesian) coordinates chosen initially.
This is important, for example, for knowing that we may choose coordinates at our convenience. For example, when doing an iterated blow-up we may choose projective coordinates after the first blow-up, or polar coordinates, and will get the same mathematics in the end.
}  
Now if $X$ is a manifold and $Y\subset X$ a submanifold, then $Y\subset X$ is locally $\R^{n-m}\times\{0\}\subset\R^n$, in suitable coordinates. Therefore, the blow-up $[X,Y]$ is well-defined as a manifold, along with the blow-down map $\beta:[X,Y]\to X$.%
\footnote{The original idea that points on the front face correspond to directions at 0 can be used directly as an invariant definition: Let $M$ be a manifold and $p\in M$. The set of directions at $p$ is $S_pM := (T_pM\setminus\{0\})/\R_{>0}$ where $T_pM$ is the tangent space and $\R_{>0}$ acts by scalar multiplication.
Then $[M,p] = (M\setminus\{p\}) \cup S_pM$, with $\beta$ the identity on $M\setminus\{p\}$ and mapping $S_pM$ to $p$. One still needs local coordinates to define the differentiable structure on $[M,p]$.}

If $X$ is a manifold with corners then a subset $Y\subset X$ is called a \defin{p-submanifold} if it is everywhere locally like the models \eqref{eqn:model b-up} (p is for product). Therefore, the blow-up $[X,Y]$ is defined for p-submanifolds $Y\subset X$.
For example, the fat subsets in the bottom line of Figure \ref{fig:blow-up examples} are p-submanifolds, as are the dashed rays in the top line. However, the dashed rays in (b), (c) and (d) in the bottom line are not p-submanifolds.

Put differently, a subset $Y\subset X$ is a p-submanifold if near every $q\in Y$ there are local coordinates centered at $q$ so that $Y$ and every face of $X$ containing $q$ is a coordinate subspace, i.e. a linear subspace spanned by some coordinate axes, locally. 

The preimage $\beta^{-1}(Y)\subset[X,Y]$ is a boundary hypersurface of $[X,Y]$, called the \defin{front face} of the blow-up. The other boundary hypersurfaces of  $[X,Y]$ are in 1-1 correspondence with those of $X$.

\begin{remark}
\label{rem:blow-up alg geom}
This notion of blow-up, sometimes called \defin{oriented blow-up}, is closely related to (unoriented) blow-up as defined in real algebraic geometry, where one \lq glues in\rq\ a real projective space instead of a sphere. Unoriented blow-up can be obtained from oriented blow-up by identifying pairs of antipodal points of this sphere. This results in an interior hypersurface (usually called exceptional divisor) rather than a new boundary hypersurface as front face. Compare Remark \ref{rem:mwc divisors}.

Unoriented blow-up has the virtue of being definable purely algebraically, so it extends to other ground fields, e.g. to complex manifolds. See \cite{Har:AG}, where also a characterization of blow-up by a universal property is given (Proposition 7.14). 
\end{remark}
\subsubsection{Multiple blow-ups}
Due to the geometric nature of the blow-up operation, it can be iterated. So if $X$ is a manifold with corners and $Y$ a p-submanifold, we can first form the blow-up $\beta_1:[X,Y]\to X$. Next, if $Z$ is a p-submanifold of $[X,Y]$ then we can form the blow-up $\beta_2:[[X,Y],Z]\to [X,Y]$. The \defin{total blow-down map} is then the composition
$$ \beta=\beta_1\circ\beta_2 : [[X,Y],Z]\to X\,.$$
See Figure \ref{fig:double blow up} for a simple example.
Of course one may iterate any finite number of times. 

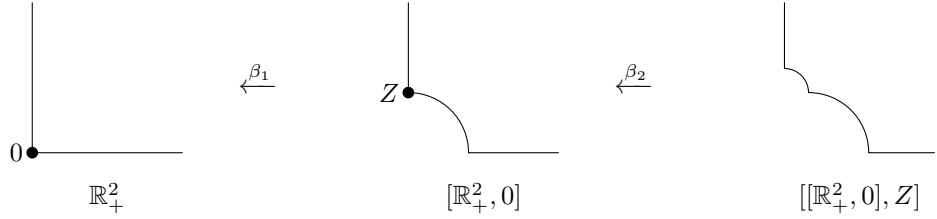
\begin{figure}
\centering 
\begin{tikzpicture}
 
\begin{scope}[scale=2]
 \draw (1,0) -- (0,0) -- (0,1);
 \draw[fill] (0,0) circle (1pt);
 
 \node at (.5,-.3){$\R^2_+$};
 \node at (0,0)[left]{$0$};
\end{scope}
 
 \node at (3,1) {$\xleftarrow{\beta_1}$};
 
\begin{scope}[xshift=5cm,scale=2]
 \draw (1,0) -- (\bupr,0) arc (0:90:\bupr) -- (0,1);
 \draw[fill] (0,\bupr) circle (1pt);
 
 \node at (.5,-.3){$[\R^2_+,0]$};
 \node at (0,\bupr)[left]{$Z$};
\end{scope}

 \node at (8,1) {$\xleftarrow{\beta_2}$};
 
\begin{scope}[xshift=10cm,scale=2]
 \draw (1,0) -- (1.4*\bupr,0) arc (0:90:\bupr) arc (0:90:.4*\bupr) -- (0,1);
 
 \node at (.5,-.3){$[[\R^2_+,0],Z]$};
\end{scope}
  
\end{tikzpicture}
\caption{A double blow-up}
\label{fig:double blow up}
\end{figure}

\subsubsection{Resolutions via blow-up}
The main use of blow-ups is that they can be used to resolve singular objects, for example functions and sets. 

\begin{definition}[Resolving functions]
\label{def:resol functions}
Let $\beta:X'\to X$ be a (possibly iterated) blow-down map of manifolds with corners and $f:\interior X\to\C$ a function. We say that \defin{$f$ is resolved by $\beta$} if $\beta^*f$ is a polyhomogeneous function on $X'$. Here $\beta^*f:=f\circ\beta$ is the pull-back.  
\end{definition}
Recall that
$\beta^*f$ need only be defined on the interior of $X'$, compare Remark \ref{rem:functions int}.
In Example \ref{ex:proj coords resol} we will see that the function $f(x,y)=\frac x{x+y}$ on $\R^2_+\setminus\{0\}$ is resolved by blowing up zero.

For subsets we need a slight generalization of p-submanifolds.
A \defin{d-submanifold} of a manifold with corners $X$ is a subset $Y\subset X$ which is everywhere locally modelled on
\begin{equation}
 \label{eqn:d-submanifold}
 X = W\times Z\times\R^l\,,\quad Y = W\times\{0\}\times\R^l_+ 
\end{equation}
for some $l\geq0$ and model spaces $W,Z$ (d means decomposable). This is a p-submanifold iff $l=0$, see \eqref{eqn:model b-up}.
For example, $\R_+^2\subset\R^2$ is a d-submanifold which is not a p-submanifold.
\begin{definition}[Resolving subsets]
\label{def:resol subsets}
Let $\beta:X'\to X$ be a (possibly iterated) blow-down map of manifolds with corners and $S\subset X$ a subset. We say that \defin{$S$ is resolved by $\beta$} if $\beta^*S$ is a d-submanifold of $X'$. Here the \defin{lift}\footnote{The lift is also called the \defin{strict transform} in the algebraic geometry literature.} $\beta^*S$  under a blow-down map $[X,Y]\to X$ is defined as
$$ \beta^*S = \overline{\beta^{-1}(S\setminus Y)}\ \text{ if }S\not\subset Y,\quad \beta^*S = \beta^{-1}(S)\ \text{ if }S\subset Y.$$ 
For an iterated blow-down map $\beta=\beta_1\circ\dots\circ\beta_k$ we define $\beta^*S = \beta_k^*\dots\beta_1^*S$.
\end{definition}
For example, the solid cone $S\subset\R^3$ (left picture in Figure \ref{fig:mwc non-examples}) is resolved by blowing up $0$ in $\R^3$.
Here  $\beta^*S\subset[\R^3,0]$ is a manifold with corners, the local model at the corner is \eqref{eqn:d-submanifold} with $W=\R\times\R_+$, $Z=\{0\}$ and $l=1$. The boundary of the cone is also resolved by $\beta$, its lift is even a p-submanifold.

Note that in general the lift $\beta^*S$ is almost the preimage, but not quite.
In the cone example, the preimage $\beta^{-1}S$ would be the union of $\beta^*S$ and the front face of the blow-up, which is a 2-sphere. We consider $\beta^*S$ since it contains the only interesting information about $S$.
 
See Figure \ref{fig:proj coords examples}(d) for another example (dashed lines) and Figure \ref{fig:blow-up} for an example of a resolution by a multiple blow-up. Both of them will be used later.

Of course we can combine Definitions \ref{def:resol subsets} and \ref{def:resol functions}: If $S\subset X$ then a function $f$ on $S\cap\interior X$ is resolved by $\beta:X'\to X$ if $S$ is resolved and $\beta^*f$ is polyhomogeneous on $\beta^*S$.
\medskip

Note that in these definitions we consider polyhomogeneous functions and d-submanifolds as \lq regular\rq\ and more general functions resp.\ subsets as \lq singular\rq. Regular objects in this sense remain regular after blow-up, as is easy to see using projective coordinates, introduced below.\footnote{For a d-submanifold $S\subset X$ to lift to a d-submanifold under blow-up of $Y\subset X$ we must require that $S$ and $Y$ intersect {\em cleanly} (which might be called \lq normal crossings\rq\ by algebraic geometers), i.e.\ near every intersection point there are coordinates in which $X$, $S$ and $Y$ are given by model spaces.}

\begin{remark}
 By a deep famous theorem of Hironaka every algebraic variety $S\subset\C P^n$ can be resolved by a sequence of blow-ups (in the algebraic geometric sense, see Remark \ref{rem:blow-up alg geom}). Similar statements hold for algebraic (or even semi- or subalgebraic) subsets of $\R^n$, see \cite{Hir:RSAVFCZ} and \cite{Hau:HTRSPAWU} for a more entertaining and low-tech survey.
\end{remark}
\begin{remark}
 There is a generalization of blow-up which is sometimes useful when resolving several scales simultaneously, see \cite{Joy:GMWC}, \cite{KotMel:GBCFP}.
\end{remark}
\subsubsection{Projective coordinates}
\begin{figure}
\centering
\begin{tikzpicture}[scale=1.7]
  \pgfgettransform\mytrafo  

\def\xmin{.55}  
\def\xmax{.95} 
\def\ymax{.5}
    
\matrix[row sep=.2cm, column sep=1.5cm,execute at begin cell=\pgfsettransform\mytrafo]
{


\draw[fill=gray!20,draw=white]
   (0,0) rectangle (1,1)
   (0,0) -- (\bupr,0) arc (0:90:\bupr) -- (0,0);
   
\draw (1,0) -- (.9,0) node[below]{$\rf$} -- (\bupr,0) arc (0:90:\bupr) --
          (0,.9) node[left]{$\lf$} -- (0,1); 

\draw (.7*\bupr,.7*\bupr) node[below left]{$\ff$};

%

\begin{scope}[shift={(\bupr,0)}]  
 \draw[->] (0,0) -- (\pcaxis,0) node[above right=-1pt and -4pt]{$\scriptstyle x$};
 \draw[->] (0,0) arc (0:30:\bupr) node[right=-1pt]{$\frac yx$};
\end{scope}

\begin{scope}[shift={(0,\bupr)}]  
 \draw[->] (0,0) arc (90:60:\bupr) node[above]{$\frac xy$};
 \draw[->] (0,0) -- (0,\pcaxis) node[left]{$\scriptstyle y$};
\end{scope}

&
\begin{scope}[x={(1cm,0)},y={(-.5cm,-.5cm)},z={(0,1cm)}]
 \draw (\bupr,0,0) -- (1,0,0);
 \draw (0,\bupr,0) -- (0,1,0);
 \draw (0,0,\bupr) -- (0,0,1);

 \draw[canvas is xy plane at z=0] (\bupr,0) arc (0:90:\bupr);
 \draw[canvas is yz plane at x=0] (\bupr,0) arc (0:90:\bupr);
 \draw[canvas is zx plane at y=0] (\bupr,0) arc (0:90:\bupr);


 \coordinate (A) at (0,0,\bupr);
 \draw[->] (A) -- (0,0,\bupr+\pcaxis) node[above right=-1pt]{$\scriptstyle z$};
 \draw[->,canvas is xz plane at y=0] (A) arc (90:60:\bupr) node[right=-1pt,yshift=4pt]{\small$\frac xz$};
 \draw[->,canvas is yz plane at x=0] (A) arc (90:60:\bupr) node[left=-1pt]{\small$\frac yz$};
 

\end{scope}

&
\begin{scope}[x={(1cm,0)},y={(-.5cm,-.5cm)},z={(0,1cm)}]
 \draw (0,0,\bupr) -- (0,0,1);
 \draw (0,\bupr,0) -- (0,1,0);
 \draw (0,0,\bupr) -- (1,0,\bupr); 
 \draw (0,\bupr,0) -- (1,\bupr,0);
 
 \draw[canvas is yz plane at x=0] (\bupr,0) arc (0:90:\bupr);

 \coordinate (A) at (0,0,\bupr);
 \draw[->] (A) -- (0,0,\bupr+\pcaxis) node[above right=-1pt]{$\scriptstyle z$};
 \draw[->] (A)  -- (\pcaxis,0,\bupr) node[yshift=4pt]{$\scriptstyle x$};
 \draw[->,canvas is yz plane at x=0] (A) arc (90:60:\bupr) node[left=-1pt]{\small$\frac yz$};

\def\bupR{.9} 
\def\yR{\ymax} 

\pgfmathsetmacro{\alpha}{atan(\ymax/\bupR)} 

\pgfmathsetmacro{\yff}{\bupr*sin(\alpha)} 
\pgfmathsetmacro{\zff}{\bupr*cos(\alpha)}  

 \foreach \x in{\xmin,\xmax}
 \draw[dashed,canvas is yz plane at x=\x] 
 (0,\bupr) arc (90:90-\alpha:\bupr) -- (\yR,\bupR) -- (0,\bupR) -- cycle;

 \draw[dashed] (\xmin,\yff,\zff) -- (\xmax,\yff,\zff);
 \draw[dashed] (\xmin,\yR,\bupR) -- (\xmax,\yR,\bupR);
 \draw[dashed] (\xmin,0,\bupR) -- (\xmax,0,\bupR);

\end{scope}

\\
\node at (0.5,0) {$\downarrow$};
&
\node at (0,0) {$\downarrow$};

&
\node at (0,0) {$\downarrow$};
\\

\draw[fill=gray!20,draw=white]
   (0,0) rectangle (1,1);

\draw[->] (0,0) -- (1,0) node[right]{$x$};    
\draw[->] (0,0) -- (0,1) node[right]{$y$}; 

\draw[fill] (0,0) circle (1pt);
&
\begin{scope}[x={(1cm,0)},y={(-.5cm,-.5cm)},z={(0,1cm)}]
 \draw[->] (0,0,0) -- (1,0,0) node[right]{$x$};
 \draw[->] (0,0,0) -- (0,1,0) node[right=+1pt]{$y$};
 \draw[->] (0,0,0) -- (0,0,1) node[right]{$z$};

 \draw[fill] (0,0,0) circle (1pt);
\end{scope}

&
\begin{scope}[x={(1cm,0)},y={(-.5cm,-.5cm)},z={(0,1cm)}]
 \draw[->] (0,0,0) -- (1,0,0) node[right]{$x$};
 \draw[->] (0,0,0) -- (0,1,0) node[right=+1pt]{$y$};
 \draw[->] (0,0,0) -- (0,0,1) node[right]{$z$}; 
 
 \draw[very thick] (0,0,0) -- (1,0,0);
 
 \draw[dashed] (\xmin,0,0) -- (\xmin,\ymax,.7) -- (\xmin,0,.7) -- cycle;
 \draw[dashed] (\xmax,0,0) -- (\xmax,\ymax,.7) -- (\xmax,0,.7) -- cycle;
 \draw[dashed] (\xmin,\ymax,.7) -- (\xmax,\ymax,.7);
 \draw[dashed] (\xmin,0,.7) -- (\xmax,0,.7); 

\end{scope}

\\ 
\node at (0,0) {(b)}; & 
\node at (0,0) {(c)}; & 
\node at (0,0) {(d)};
\\
};

\end{tikzpicture}
\caption{Projective coordinates for examples (b), (c), (d) of Figure \ref{fig:blow-up examples}. Dashed lines in (d) indicate a singular subset (below) and its resolution (above).}
\label{fig:proj coords examples}
\end{figure}
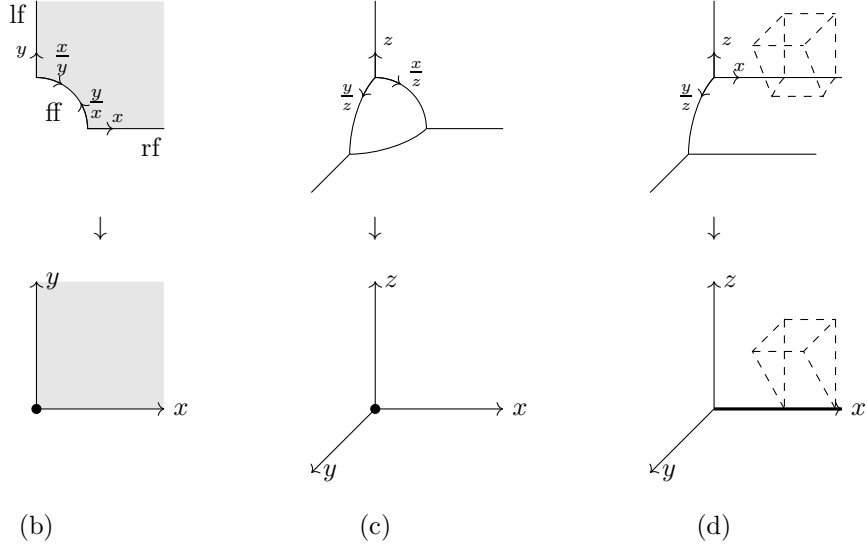

Projective coordinates simplify calculations with blow-ups
 and also provide the link of blow-ups to the discussion of scales.

We first discuss this for the space $[\R^2_+,0]$, see Figure \ref{fig:proj coords examples}(b).
Recall that points of $[\R^2_+,0]$ correspond to pairs consisting of a ray (in $\R^2_+$, emanating from 0) and a point on that ray.
Now, with $x,y$ standard coordinates on $\R^2_+$,
$$ \text{rays }\leftrightarrow \text{ values of }\frac yx,
\qquad \text{points on a ray }\leftrightarrow\text{ values of }x$$
except if the ray is the $y$-axis (which would correspond to $\frac yx=\infty$). Here $\frac yx\geq0$ and $x\geq0$, and $x=0$ is the endpoint of the ray.

This means that $\frac yx$ and $x$ provide a coordinate system for $[\R^2_+,0]\setminus\lf$, where $\lf$ (\lq left face\rq) is the lift of the $y$-axis:\footnote{It would be formally better to write $\frac{\beta^*y}{\beta^*x}$ and $\beta^*x$ instead of $\frac yx$ and $x$, but this quickly becomes cumbersome. Note that $\beta^*x$ vanishes on $\lf\cup\ff$ and $\beta^*y$ vanishes on $\ff\cup\rf$.}
\begin{equation}
 \label{eqn:proj coord}
 (x,\frac yx):[\R_+^2,0]\setminus\lf \to \R^2_+
\end{equation}
\begin{quote}
 \small
 We need to check that this is a smooth coordinate system. This means:
 
\begin{enumerate}
 \item[1.] The function $\frac yx$, which is defined and smooth on $[\R_+^2,0]\setminus(\lf\cup\ff)$, extends smoothly to $[\R_+^2,0]\setminus\lf $.
 \item[2.] The map \eqref{eqn:proj coord} is a diffeomorphism. 
\end{enumerate}
Both statements refer to the differentiable structure on $[\R^2_+,0]$, which was defined by writing $[\R^2_+,0] = \R_+\times S^1_{++}$ where $S^1_{++}$ is the quarter circle. If we use the angle coordinate $\varphi\in[0,\frac\pi2]$ on $S^1_{++}$ then we need to check that the map $(r,\varphi)\to (x,\frac yx)$ extends smoothly from $r>0,\varphi<\frac\pi2$ to $r\geq0,\varphi<\frac\pi2$ and is a diffeomorphism $\R_+\times[0,\frac\pi2)\to\R_+\times\R_+$. This can be seen from the explicit formulas
 $x=r\cos\varphi$, $\frac yx=\tan\varphi$, and for the inverse map
 $r=x\sqrt{1+\left(\frac yx\right)^2}$, $\varphi=\arctan\frac yx$.
\end{quote}

By symmetry, we have another smooth coordinate system
given by $\frac xy$ and $y$ on the set $[\R^2_+,0]\setminus \rf$, where $\rf$ (\lq right face\rq) is the lift of the $x$-axis.

Note that  in the coordinate system $x,\frac yx$
 the boundary defining function of the front face is $x$, and  in the coordinate system $y,\frac xy$ it is $y$.

Projective coordinates can be used to check that a function is resolved under a blow-up:
\begin{lemma}
\label{lem:proj coord phg}
  A function $f$ on $\R_{>0}^2$ is resolved by the blow up of $0$ if and only if $f$ is polyhomogeneous as a function of $\frac xy,y$ and as a function of $x,\frac yx$.
\end{lemma}
This is clear since polyhomogeneity (or smoothness) of a function on a manifold means polyhomogeneity (or smoothmess) in each coordinate system of an atlas.
\begin{example}
\label{ex:proj coords resol}

We consider the function $f(x,y)=\frac x{x+y}$ on $\R^2_+\setminus0$ again.
We saw in Example \ref{ex:phg fcns}\eqref{ex:phg fcn not} that $f$ is not polyhomogeneous at $0$. However, 
\begin{align*}
 \text{in coordinates }X=\frac xy,\ y:\quad & \beta^*f = \frac X{X+1} \\
 \text{in coordinates }x,\ Y=\frac yx:\quad & \beta^*f= \frac1{1+Y}
\end{align*}
and both of these functions are smooth for $(X,y)\in\R^2_+$ resp. $(x,Y)\in\R^2_+$, the respective ranges of these coordinates.
So $f$ is resolved by $\beta$, and $\beta^*f$ is even smooth on $[\R^2_+,0]$. 

As another example, consider $f_2(x,y)=\frac x{x+y+xy}$. Here $\beta^*f_2=\frac X{X+1+Xy}$ and $\beta^*f_2=\frac1{1+Y+xY}$ in the two coordinate systems, so $f_2$ is also resolved by $\beta$. Note that these agree with $\beta^*f$ at $y=0$ and $x=0$ respectively, which means $\beta^*f_2=\beta^*f$ at the front face. This is clear a priori since $xy$ vanishes to second order at $x=y=0$.
\end{example}
\begin{remark}[Relation of projective coordinates to scaled limit]
Suppose a function $f$ on $\R_{>0}^2$ is resolved by $\beta:[\R^2_+,0]\to\R^2_+$, and assume $\beta^*f$ is even smooth.
To emphasize the relation to the discussion of scales, we denote coordinates by $x,h$ and write $f_h(x)=f(x,h)$.
\begin{enumerate}
 \item The rescaled limit $g(X)=\lim_{h\to0} f_h(hX)$ is simply the restriction of $\beta^*f$ to the front face $\ff$, when parametrizing
 $\ff$ by the projective coordinate $X$. 
 
To see this, note that in the projective coordinate system $X,h$ the map $\beta$ is given by $\beta(X,h)=(hX,h)$ (this is the meaning of writing $X=\frac xh$), so $(\beta^*f)(X,h) = f(hX,h)$, and $h=0$ is the front face.
 \item That $f$ is resolved by $\beta$ contains additional information beyond existence of this scaled limit: information on derivatives as well as information on the behavior of $g(X)$ as $X\to\infty$. Note that $X=\infty$ corresponds to the \lq lower\rq\ corner in $[\R^2_+,0]$.
 More precisely, $g$ is smooth at $\infty$ in the sense that $\eta\mapsto g(\frac1\eta)$ is smooth at $\eta=0$. Here $\eta$ is the coordinate $\frac hx$ in the second projective coordinate system.
\end{enumerate}
\end{remark}
\noindent For more general blow-ups it is useful to have:  
\begin{quote} 
{\bf Quick practical guide to finding projective coordinate systems:}

Point blow-up of 0 in  $\R^k_+\times\R^{n-k}$: Near the (lift of the) $x$-axis projective coordinates are $x$ and $\frac{y_j}x$, where $y_j$ are the variables other than $x$.
These are coordinates except on the (lift of the) set $\{x=0\}$. Similarly for any other axis.

Blow-up of coordinate subspace $Y \in\R_+^k\times\R^{n-k}$: Apply the previous to variables $x,y_j$ vanishing on $Y$. Other variables remain unchanged.
\end{quote}
It may be useful to think of $x$ as \lq dominant\rq\ variable on the coordinate patch: for any compact subset of the patch there is a constant $C$ so that $|y_j|\leq Cx$. So $\frac{y_j}x$ is bounded there.
Note:
\begin{quote}
 dominant variable = boundary defining function of front face
\end{quote}

For the examples in Figure \ref{fig:blow-up examples}(a),(c),(d)  we get the projective coordinate systems, see also Figure \ref{fig:proj coords examples}(c),(d) (where only one system is indicated): 
\begin{enumerate}
 \item[(a)] $[\R\times\R_+,0]$:  near the interior of the front face: $y,\frac xy$; in a neighborhood of the lift of the $x$-axis: $x,\frac yx$.\footnote{The latter are really two coordinate patches, one for $x\geq0$ (near right corner) and one for $x\leq0$ (near left corner). Near the left corner it is more customary to use $|x|, \frac y{|x|}$ instead so the dominant variable is positive.}
 \item[(c)]
 $[\R_+^3,0]$: outside the left boundary hypersurface: $x,\frac yx,\frac zx$; outside the back boundary hypersurface: $y,\frac xy,\frac zy$; outside the bottom boundary hypersurface: $z,\frac xz,\frac yz$.
 \item[(d)]
 $[\R_+^3,\R_+\times\{0\}]$: outside the back boundary hypersurface: $x,y,\frac zy$;  outside the bottom boundary hypersurface: $x,z,\frac yz$.
\end{enumerate}

\begin{exercise}
\label{exer:double blow-up}
 Show that the function $f(x,y)=\sqrt{x^2+xy+y^3}$ on $\R^2_+$ is resolved by the double blow-up in Figure \ref{fig:double blow up}, but not by the simple blow-up of $0\in\R^2_+$.\footnote{Solution:
In coordinates $x$, $Y=\frac yx$ the function $\beta_1^*f=x\sqrt{1+Y+xY^3}$ is polyhomogeneous since it is smooth. In coordinates $X=\frac xy$, $y$ the function $\beta_1^*f=y\sqrt{X^2+X+y}$ is polyhomogeneous outside $X=y=0$, but not at this point. 
\\
Therefore we blow up $X=y=0$, which is the point $Z$ in Figure \ref{fig:double blow up}. Let $\beta=\beta_1\circ\beta_2$.  
In coordinates $X$, $\eta=\frac yX$ the function $\beta^*f=X^{3/2}\eta\sqrt{X+1+\eta}$ is polyhomogeneous. In coordinates $\xi=\frac Xy$, $y$ the function $\beta^*f=y^{3/2}\sqrt{\xi^2y+\xi+1}$ is polyhomogeneous. So $f$ is resolved by $\beta$.
 }
\end{exercise}


\subsubsection{Quasihomogeneous blow-up}
\label{subsubsec:blow-up qh}
In many problems scalings other than $x\sim y$ appear, for example $x\sim \sqrt y$ in the function $f(x,y) = \frac 1{x^2+y}$. These can be understood either by multiple blow-ups, as in Exercise \ref{exer:double blow-up}, or by the use of quasihomogeneous blow-up. 
 
This occurs, for example, in Section \ref{sec:adiab variable}, and also for the heat kernel (where $y$ is time), see e.g. \cite{Mel:APSIT}, \cite{Gri:NHKA}, \cite{MazRow:HTAP}. 

For simplicity we only consider the quasihomogeneous blow-up of $0$ in $\R^2_+$, with $x$ scaling like $\sqrt y$. We denote it by $[\R^2_+,0]_q$. This is sometimes called parabolic blow-up.
The idea is analogous to regular blow-up, except that the rays in $\R^2_+$ through 0 are replaced by \lq parabolas\rq, by which we mean the sets $\{y=Cx^2\}$ including the cases $C=0$,  i.e.\ the $x$-axis, and $C=\infty$, i.e.\ the $y$-axis.
Then the blown-up space is constructed by removing 0 and replacing it by one separate endpoint for each parabola. These endpoints can be thought of as forming a quarter  circle again, so the blown-up space looks just like, and in fact will be diffeomorphic to, $[\R^2_+,0]$. However, the blow-down map $\beta$ will be different.

Here is a local model realizing this idea: Let $r(x,y)=\sqrt{x^2+y}$ and $S^1_q = \{(\omega,\eta)\in\R^2_+:\, r(\omega,\eta)=1\}$. Then we let
$[\R^2_+,0]_q = \R_+\times S^1_q$ with blow-down map
$$ \beta(r,(\omega,\eta)) = (r\omega,r^2\eta).$$
This is constructed so that $\beta$ maps each half line $\{(\omega,\eta)=\const\}$ to a parabola, so that indeed endpoints of parabolas correspond to points of $\ff:=\{0\}\times S^1_q$. 
Also, $\beta$ maps $\ff\subset[\R^2_+,0]_q$ to $0\in\R^2_+$ and is a diffeomorphism between the complements of these sets.%
\footnote{Maybe you ask: why this model, not another one? In fact, the precise choice or $r$ and $S^1_q$ are irrelevant -- any choice of  positive smooth function $r$ which is 1-homogeneous when giving $x$ the weight 1 and $y$ the weight 2, and any section transversal to all parabolas which stays away from the origin
 will do, with the same definition of $\beta$. Choosing $S^1_q=r^{-1}(1)$ has the nice feature that use of the letter $r$ is consistent in that $r(\beta(R,(\omega,\eta)))=R$.}

Projective coordinates are as shown in Figure \ref{fig:proj coord qh}. 
The coordinates near $A$ seem quite natural: $x$ smoothly parametrizes the points on each parabola $\{y=Cx^2\}$ (except $C=\infty$), and the parabolas are parametrized by the value of $C=\frac y{x^2}$, so pairs $(\frac y{x^2},x)$ parametrize pairs (parabola, point on this parabola). On the other hand, the coordinates near $B$ require explanation.
One way to understand them is to check in the model that these are indeed coordinates (compare the explanation after \eqref{eqn:proj coord}; do it!). Without reference to the model the exponents that occur can be understood from three principles: 
\begin{enumerate}
 \item[(a)] 
 The coordinate \lq along the front face' should reflect the  scaling 
$x\sim\sqrt y$.
 \item[(b)] 
$\beta$ should be smooth, so both $x$ and $y$ must be expressible 
as monomials in the coordinates,\footnote{This means that we require $\beta$ to be a b-map, a condition stronger than smoothness, see \cite{Mel:DAMWC}.} near $A$ and near $B$
 \item[(c)] 
 The smooth structure on $[\R^2_+,0]_q$ should be the minimal one satisfying (a) and (b), i.e. the exponents should be maximal possible.
\end{enumerate}
So for the system near $A$, (b) implies that in the coordinate along $\ff$ the exponent of $y$ must be  $\frac1m$ for some $m\in\N$, and then (c) implies $m=1$. Hence the coordinate must be $\frac y{x^2}$ by (a). The exponent of $x$ in the other coordinate must be 1 by (b) and (c). Similarly, near $B$ in the coordinate along $\ff$ we need $x$ in first power by (b) and (c), and (a) gives $\frac x{\sqrt y}$. Then (b) and (c) leave no choice but to have $\sqrt y$ as the other coordinate.\footnote{A different way to understand the coordinates $\frac y{x^2}$, $\frac x{\sqrt y}$ along $\ff$ is to note that $\frac y{x^2}$ is a defining function of $\rf$ in its interior $x>0$, and $\frac{x}{\sqrt y}$ is a defining function of $\lf$ in its interior $y>0$.

This reflects the fact that only the point $0\in\R^2_+$ is affected by the blow-up, that is, that $\beta$ is a diffeomorphism between the complements of $\beta^{-1}(0)$ and $\{0\}$. In particular, quasihomogeneous blow-up is {\em not} the same as first replacing the variable $y$ by $\sqrt y$ and then doing a standard blow-up. 
}

Projective coordinates can be used as in Lemma \ref{lem:proj coord phg} to check whether quasihomogeneous blow-up resolves a function.

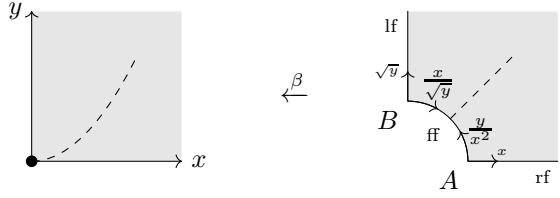
\begin{figure}
\centering
\begin{tikzpicture}

\begin{scope}[scale=2]
  \draw[fill=gray!20,draw=white]
   (0,0) rectangle (1,1);

\draw[->] (0,0) -- (1,0) node[right]{$x$};    
\draw[->] (0,0) -- (0,1) node[left]{$y$}; 

\draw[fill] (0,0) circle (1pt);

\draw[dashed] (0,0) parabola (.7,.7);

\end{scope}

 \node at (3.5,1) {$\xleftarrow{\beta}$};

\begin{scope}[xshift=5cm,scale=2]
 \draw[fill=gray!20,draw=white]
   (0,0) rectangle (1,1)
   (0,0) -- (\bupr,0) arc (0:90:\bupr) -- (0,0);
   
\draw (1,0) -- (.9,0) node[below]{$\scriptstyle\rf$} -- (\bupr,0) node[below left]{$A$} arc (0:90:\bupr) node[below left]{$B$} --
          (0,.9) node[left]{$\scriptstyle\lf$} -- (0,1); 
\draw (.7*\bupr,.7*\bupr) node[below left]{$\scriptstyle\ff$};
\draw[dashed] (\bupr*.707,\bupr*.707) -- (.7,.7);

%

\begin{scope}[shift={(\bupr,0)}]  
 \draw[->] (0,0) -- (\pcaxis,0) node[above right=-1pt and -4pt]{\tiny $x$};
 \draw[->] (0,0) arc (0:30:\bupr) node[right=-1pt]{\small $\frac y{x^2}$};
\end{scope}

\begin{scope}[shift={(0,\bupr)}]  
 \draw[->] (0,0) arc (90:60:\bupr) node[above]{\small $\frac x{\sqrt y}$};
 \draw[->] (0,0) -- (0,\pcaxis) node[left]{\tiny $\sqrt y$};
\end{scope}

\end{scope}
 
\end{tikzpicture}

 \caption{Projective coordinate systems for quasihomogeneous blow-up $[\R^2_+,0]_q$}
 \label{fig:proj coord qh}
\end{figure}

For more details, including the question  of coordinate invariance, see
\cite{EpsMelMen:RLSPD} and \cite{Mel:DAMWC}; see also \cite{GriHun:POCGQLSSI}. A more general blow-up procedure is introduced in \cite{KotMel:GBCFP}, see also \cite{Joy:GMWC}, \cite{Kot:BUMWGC}. This is closely related to blow-up in toric geometry, see \cite{CoxLitSchen:TV}.

\subsection{Summary on blow-up and scales; further examples} \label{subsec:resol-ex}

We first summarize our discussion of the function
$f(x,h) = \frac{x}{x+h}$: $f$ is smooth on $\R^2_+\setminus0$ but has no continuous extension to 0. The behavior of $f$ near 0 can be described by saying that the scaling limit $\lim_{h\to0} f(hX,h)=g(X)$ exists for all $X$. 
This can be restated in terms of the blow-up of 0 in $\R^2_+$ with blow-down map $\beta:[\R^2_+,0]\to\R^2_+$ and front face $\ff=\beta^{-1}(0)$: 
the function $\beta^*f$, defined on $[\R^2_+,0]\setminus\ff$, extends continuously to $\ff$, and $g$ is the restriction of this extension to $\ff$ when $\ff$ is parametrized by $X$. Here $X=\frac xh$ is part of the  projective coordinate system $X,h$. 

In fact, we saw that the extension of $\beta^*f$ is not only continuous but even smooth on $[\R^2_+,0]$. That is, $f$ is resolved by $\beta$
in the sense of Definition \ref{def:resol functions}.

The fact that $g$ is not constant leads to the discontinuity of $f$ at 0.

\medskip

The example suggests that the vague idea of scaling behavior is captured by the notion of resolution, which is defined rigorously in Definition \ref{def:resol functions}. We note a few details of this definition:
\begin{enumerate}
 \item
 The resolved function $\beta^*f$ is required to be polyhomogeneous, which means in particular:
\begin{itemize}
 \item the asymptotics holds with all derivatives
 \item full asymptotics is required, not just leading order asymptotics
\end{itemize}
To include derivatives is natural since we want to deal with differential equations. To require full asymptotics is then natural since for example smoothness at a boundary point means having a full asymptotic series (the Taylor series). Only the combination of both conditions yields a unified theory.\footnote{Of course one could define finite order (in number of derivatives or number of asymptotic terms) theories, and this may be useful for some problems. However, many problems do admit infinite order asymptotics -- once the scales are correctly identified. Requiring less than the best possible sometimes obfuscates the view towards the structure of a problem.}
\item
On the other hand, requiring $\beta^*f$ to be smooth would be too restrictive (compare Footnote \ref{footnote:logs}). What really matters is the product structure near corners as explained in Section \ref{subsubsec:phg product}.
\item
Of course any function can be \lq over-resolved\rq, for example if $f$ is smooth on $\R^2_+$ then we may still look at $\beta^*f$ which is still smooth. This would correspond to \lq looking at $f$ at scale $x\sim h$\rq.\footnote{So really we should not say that a function \lq exhibits the scale $x\sim h$\rq, since every function does. More appropriate may be \lq $f$ requires scale $x\sim h$\rq, or \lq The scale $x\sim h$ is relevant for $f$\rq. In any case, \lq $f$ is resolved by $\beta$\rq\ is a well-defined statement giving an upper bound on the \lq badness\rq\ of $f$.}
\end{enumerate}
We give some more examples to illustrate these points. 
\begin{examples}
In these examples we denote coordinates on $\R^2_+$ by $x,h$ to emphasize the relation to scaling. $\beta$ is always the blow-down map for the blow-up of 0 in $\R^2_+$.
\begin{enumerate}
 \item
 $f(x,h) = x+h$ is smooth on $\R^2_+$. In scaled coordinates
 $f(hX,h) = h(X+1)$ is the expansion of $\beta^*f$ at the front face.
 \item
 $f(x,h)=\sqrt{x+h}$ is not polyhomogeneous on $\R^2_+$ as can be seen from the Taylor expansion as $h\to0$ for fixed $x>0$:
 $$\sqrt{x+h}=\sqrt x \sqrt{1+\frac hx} = \sum_{k=0}^\infty \binom{\frac12}k
 x^{\frac12-k} h^k $$
 compare \eqref{eqn:f expansion 2}. 
 However, note that $f_h=f(\cdot,h)$ converges uniformly to $f_0$ on $\R_+$. But already $f_h'$ does not converge uniformly to $f_0'$. The same is true for $f_2(x,h)=\sqrt{x^2+h^2}$ even though $f_0$ is smooth.
 
 Both $f$ and $f_2$ are resolved by blowing up $0$ in $\R^2_+$.
 
 These examples show that non-trivial scaling behavior may only be visible in the derivatives.
 \item
 $f(x,h)=\sqrt{x^2+xh+h^3}$ is resolved by the double blow-up in Figure \ref{fig:double blow up}, see Exercise \ref{exer:double blow-up}. The two front faces correspond to the scales $x\sim h$ and $x\sim h^2$.
 Any problem involving $f$ needs to take into account both of these scales.
\end{enumerate}
\end{examples}
\medskip

To end this section we consider an example in three dimensions where a set is resolved by two blow-ups. This will be used in Section \ref{sec:adiabatic with ends}.

 Consider the family of plane domains $\Omega_h\subset\R^2$, $h>0$, shown in Figure \ref{fig:example Omegascale}: The $1\times h$ rectangle $[0,1)\times(0,h)$ with a fixed triangle (e.g. a right-angled isosceles triangle), scaled to have base $h$,  attached at one end. Again we want to describe the behavior of $\Omega_h$ as $h\to0$. As in the first example, different features emerge at different scales:
\begin{enumerate}
\item
We can consider $B:=\lim_{h\to0}\Omega_h$. This is just an interval.\footnote{The precise meaning of the limit is irrelevant for this motivational discussion. You may think of Hausdorff limits.}
Many features of $\Omega_h$ are lost in the limit: the thickness $h$,  the triangular shape at the end.
\item
More information is retained by noting that $y$ scales like $h$, hence considering
\begin{equation}
\label{eqn:Omega_h, scale A}
 A_h := \{(x,Y):\ (x,hY) \in \Omega_h\},
\end{equation}
the domain obtained from stretching by the factor $h^{-1}$ in the $y$-direction.
Then $A:=\lim_{h\to0}A_h$ is the square $(0,1)\times(0,1)$. This still forgets the triangular shape at the end.
\item
At the left end, both $x$ and $y$ scale like $h$. So we consider
\begin{equation}
\label{eqn:Omega_h, scale S}
 S_h := \{(X,Y):\ (hX, hY) \in \Omega_h\} = h^{-1}\Omega_h\,.
\end{equation}
Then $S:=\lim_{h\to0}S_h$ is a half infinite strip of width one with a triangle attached at the left end. This limit remembers the triangle, but not that $\Omega_h$ has essentially length 1 in the $x$-direction.
\end{enumerate}

For the asymptotic analysis of the eigenvalue problem on $\Omega_h$ in Section \ref{sec:adiabatic with ends} it will be essential to understand $A$ and $S$ as parts of one bigger space, which arises as resolution of the closure of $\Omega=\bigcup\limits_{h>0}\Omega_h\times\{h\} \subset\R^3$. This resolution is shown in Figure \ref{fig:blow-up} and explained there. Note that $A$ and $S$ are boundary hypersurfaces. The  limit interval $B$ occurs as the base of a natural fibration of the face $A$.

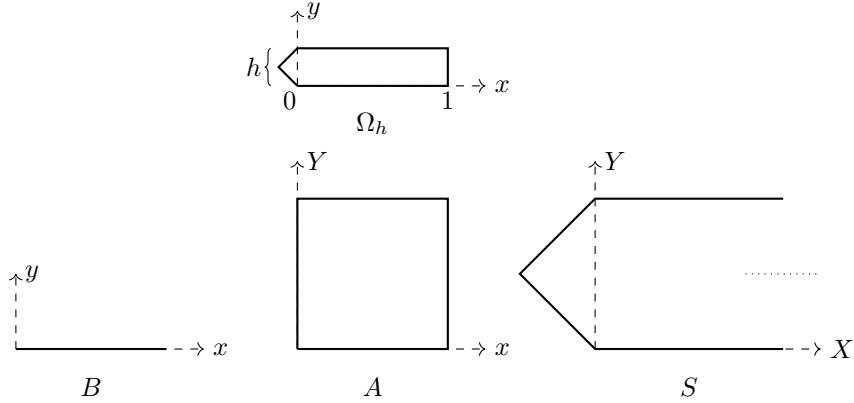
\begin{figure}
\centering
 \begin{tikzpicture}
\matrix
{
;& 
 \draw[thick]  (0,0) -- (2,0)  -- (2,.5) -- (0,.5) -- (-.25,.25) -- (0,0);
 \draw[dashed,->] (0,0) -- (2.5,0) node[right]{$x$};
 \draw[dashed,->] (0,0) -- (0,1) node[right]{$y$};  
 \node at (1,-.5) {$\Omega_h$};
 \node at (2,-.2){1};
 \node at (-.1,-.2){0};
\draw[decorate,decoration={brace,amplitude=2pt},xshift=-3pt,yshift=0pt]
(-.25,0) -- (-.25,.5) node [midway,left]{$h$};

& ;
\\
 \draw[thick]  (0,0) -- (2,0);
 \draw[dashed,->] (0,0) -- (2.5,0) node[right]{$x$};
 \draw[dashed,->] (0,0) -- (0,1) node[right]{$y$};  
 \node at (1,-.5) {$B$};
&  
 \draw[thick]  (0,0) -- (2,0)  -- (2,2) -- (0,2)  -- (0,0);
 \draw[dashed,->] (0,0) -- (2.5,0) node[right]{$x$};
 \draw[dashed,->] (0,0) -- (0,2.5) node[right]{$Y$};  
  \node at (1,-.5) {$A$};
&
 \draw[thick]  (2.5,0) -- (0,0) -- (-1,1)  -- (0,2) -- (2.5,2);
 \draw[dashed,->] (0,0) -- (3,0) node[right]{$X$};
 \draw[dashed,->] (0,0) -- (0,2.5) node[right]{$Y$};  
 \draw[dotted] (2,1) -- (3,1);
  \node at (1.25,-.5) {$S$};
\\
};

%
\end{tikzpicture}
\caption{A family of domains $\Omega_h$ and three rescaled limits as $h\to0$}
\label{fig:example Omegascale}
\end{figure}

\section{Generalities on quasimode constructions; the main steps}
\label{sec:main steps}
In this section we give an outline of the main steps of the quasimode constructions that will be carried out in the following sections. 

For each $h>0$ let $\Omega_h$ be a bounded domain in $\R^2$, and let $P_h=-\Delta$ be the Laplacian on $\Omega_h$, acting on functions that vanish at the boundary $\partial\Omega_h$. We assume that $\partial\Omega_h$ is piecewise smooth.\footnote{%
More generally one can consider families of compact manifolds with (or without) boundary and differential operators on them which are elliptic and self-adjoint with respect to given measures and for given boundary conditions. The methods are designed to work naturally in this context. Non-smooth boundary may require extra work.}

A \defin{quasimode} for the family $(\Omega_h)_{h>0}$ is a family $(\lambda_h,u_h)_{h>0}$ 
where $\lambda_h\in\R$ and $u_h$ is in the domain of $P_h$ (in particular, $u_h=0$ at $\partial\Omega_h$), so that
\begin{equation}
\label{eqn:quasimode estimate} 
 (P_h-\lambda_h)u_h = O(h^\infty)\quad\text{ as }h\to 0.
\end{equation}
Here $O(h^\infty)$ means $O(h^N)$ for each $N$. 
We are ambitious in that we require these estimates to hold uniformly,
also for all derivatives with respect to $x\in\Omega_h$ and with respect to $h$.

We reformulate this as follows: Consider the \defin{total space}
$$ \Omega = \bigcup_{h>0} \Omega_h\times\{h\} \subset\R^2\times\R_+.$$
We assume that $\Omega_h$ depends continuously on $h$ in the sense that $\Omega$ is open.
A family of functions $u_h$ on $\Omega_h$ corresponds to a single function $u$ on $\Omega$ defined by $u(x,h)=u_h(x)$. The operators $P_h$ define a single operator $P$ on $\Omega$ via
$$ (Pu)(\cdot,h) = P_h (u_h).$$
The operator $P$ differentiates only in the $\Omega_h$ directions, not in $h$. Then  a quasimode is a pair of functions $\lambda:\R_{>0}\to\R$, $u:\Omega\to\R$ satisfying the boundary conditions and
$$ (P-\lambda)u = O(h^\infty)\quad\text{ as }h\to0.$$
\medskip

How can we find quasimodes?
Since the only issue is the behavior as $h\to0$, one expects that finding $\lambda$ and $u$ reduces to solving PDE problems \lq at $h=0$\rq, along with an iterative construction: 
first solve with
 $O(h)$ as right hand side, then improve the solution so the error is $O(h^2)$ etc.

 This is straightforward in the case of a regular perturbation, i.e.\ if the family $(\Omega_h)_{h>0}$ has a limit $\Omega_0$ at $h=0$, and the resulting family is smooth for $h\geq0$. This essentially means that the closure $\Omegabar$ of the total space $\Omega$  is a manifold with corners, see Section \ref{subsec:reg pert setup} for details.  In particular, $\Omega_0$ is still a bounded domain in $\R^2$.
Then the problem at $h=0$ is the \defin{model problem}
$$ (P_0-\lambda_0)v = g \quad\text{on }\Omega_0,\quad v=0\text{ at }\partial\Omega_0$$
where $P_0=-\Delta$ on $\Omega_0$.
Solving the model problem is the only analytic input in the quasimode construction. 
As we recall in Section \ref{sec:regular pert} the iterative step reduces to solving this equation, plus some very simple algebra.

However, our main focus will be on {\bf singular perturbations}, where a limit $\Omega_0$ exists but $\Omega_h$ does not depend smoothly on $h$ at $h=0$, so $\Omegabar$ has a singularity at $h=0$. 
For example, if $\Omega_0$ 
is an interval or a curve, then this singularity looks approximately like an edge, see Figures \ref{fig:adiab limit b-up}, \ref{fig:adiab var resolution} and \ref{fig:blow-up}. 
We will consider several concrete such families. Their common feature is that this singularity can be resolved by (possibly several) blow-ups, yielding a manifold with corners $M$ and a smooth map
$$ \beta: M\to\Omegabar\,.$$
As explained in Section \ref{sec:intro mwc}  this corresponds to a certain scaling behavior in the family $(\Omega_h)_{h>0}$ as $h\to0$.
The boundary hypersurfaces of $M$ at $h=0$, whose union is   $$\partial_0 M:=\beta^{-1}(\Omegabar\cap\{h=0\}),$$ will now take the role of $\Omega_0$, i.e.\ they will carry the
model problems whose solution is used for constructing quasimodes.

Since in the singular case several model problems are involved, the algebra needed for the quasimode construction is more complicated than in the regular case. However, 
this can be streamlined, and unified, by cleverly defining function spaces $\calE(M)$ and $\calR(M)$ which will contain putative quasimodes $u$ and remainders $f=(P-\lambda)u$, respectively, along with suitable notions of {\bf leading part} (at $h=0$). The leading parts will lie in spaces $\calE(\partial_0 M)$ and $\calR(\partial_0 M)$, and are, essentially, functions on $\partial_0 M$. 
All model problems together define the {\bf model operator} $(P-\lambda)_0:\calE(\partial_0 M)\to\calR(\partial_0M)$. 
Denoting for the moment by $\LP$ the leading part map, the needed algebra will be summarized in a {\em Leading part and model operator lemma}, which states\footnote{This is analogous to the algebra needed for the parametrix construction in the classical pseudodifferential calculus, as explained in \cite{Gri:BBC}: $\LP$ corresponds to the symbol map, the model operator is the constant coefficient operator obtained by freezing coefficients at any point. Invertibility of the model operator (which amounts to ellipticity) allows construction of a parametrix, which is the analogue of the construction of a quasimode.}
\begin{enumerate}
 \item[a)] the exactness of the sequences
\begin{equation}
\label{eqn:exact sequences}
\begin{array}{lllllllll}
 0 &\to& h\calE(M) &\to& \calE(M) &\xrightarrow{\LP} &\calE(\partial_0 M) & \to &0 \\
 0 &\to& h\calR(M) &\to& \calR(M) &\xrightarrow{\LP} &\calR(\partial_0 M) &\to &0\,;
\end{array}
\end{equation}
The main points here are exactness at $\calR(M)$ and at $\calE(\partial_0 M)$, explicitly:
$$f\in\calR(M),\ \LP(f)=0\Rightarrow f\in h\calR(M),$$
and any $v\in\calE(\partial_0 M)$ is the leading part of some $u\in\calE(M)$;
\item[b)] the commutativity of the diagram
\begin{equation}
\label{eqn:comm diagram model op}
 \xymatrixcolsep{3pc}
 \xymatrix{
 \calE(M) \ar[d]^{P-\lambda} \ar[r]^{\LP}        & \calE(\partial_0 M) \ar[d]^{(P-\lambda)_0}  \\
 \calR(M) \ar[r]^{\LP}                                     & \calR(\partial_0 M)}
\end{equation}
That is, $(P-\lambda)_0$ encodes the leading behavior of $P-\lambda$ at $h=0$.
\end{enumerate}

Summarizing, the main steps of  the quasimode constructions are:
\begin{enumerate}
 \item
 Resolve the geometry, find the relevant scales
 \item
 Find the correct spaces for eigenfunctions and remainders
 \item
 Find the correct \lq leading part' definition for eigenfunctions and remainders. Identify model operators, prove Leading part and model operator lemma.
 \item
Study model operators (solvability of homogeneous/non-homogeneous PDE problems)
 \item
 Carry out the construction: Initial step, inductive step 
\end{enumerate}
The examples are progressively more complex, so that some features will occur only in later examples. Of course the process of finding the correct spaces etc. may be non-linear, as usual.

The \lq meat\rq\ is in step 4. After this, step 5 is easy. 
Steps 1-3 are the conceptual work needed to reduce the construction of quasimodes to the study of model operators.

The results are formulated in Theorems \ref{thm:regular quasimodes}, \ref{thm:adiab quasimodes}, \ref{thm:adiabatic variable}, \ref{thm:adiab ends}. They all have the same structure: given data for $\lambda$ and $u$ at $h=0$ there is a unique quasimode having this data. For $\lambda$ the data is the first or first two asymptotic terms, for $u$ the data is the restriction to the boundary hypersurfaces of $M$ at $h=0$. Both cannot be freely chosen but correspond to a boundary eigenvalue problem.

There are many other types of singular perturbations which can be treated by the same scheme. For example, $\Omega_0$ could be a domain with a corner and $\Omega_h$ be obtained from $\Omega_0$ by rounding the corner at scale $h$. Or $\Omega_h$ could be obtained from a domain $\Omega_0$ by removing a disk of radius $h$. 
\begin{remark}[Are the blow-ups needed?]
Our constructions yield precise asymptotic information about $u$ as a function of $h,x,y$. Different boundary hypersurfaces of $M$ at $h=0$ correspond to different asymptotic regimes in the family $\Omega_h$. 
This is nice, but is it really needed if we are only interested in $\lambda$, say?

The leading asymptotic term for $\lambda$ and $u$ as $h\to0$ is often easier to come by and does not usually require considering different regimes.
But in order to obtain higher order terms  of $\lambda$, it is necessary to obtain this detailed information about $u$ along the way. As we will see, all regimes of the asymptotics of $u$ will \lq influence\rq\ the asymptotics of $\lambda$, often starting at different orders of the expansion. Another mechanism is that justifying a formal expansion up to a certain order usually requires knowing the expansion to a higher order (as is explained in \cite{GriJer:AEPD}, for example).
\end{remark}

\section{Regular perturbations}
\label{sec:regular pert}

To set the stage we first consider the case of a regular perturbation. Here basic features of any quasimode construction are introduced: 
the reduction to an initial and an inductive step, the
identification of a model operator, and the use of the solvability properties of the model operator for carrying out the initial and inductive steps. 

\subsection{Setup}
\label{subsec:reg pert setup}
Let $\Omega_h$, $h\geq0$ be a family of bounded domains in $\R^2$ with smooth boundary.\footnote{Everything works just as well in $\R^n$ or in a smooth Riemannian manifold. Also the smoothness of the boundary can be relaxed, for example the $\Omega_h$ could be domains with corners, then the requirement (B) below is that $M$ be a d-submanifold of $\R^n\times\R_+$, as defined before Definition \ref{def:resol subsets}.} 
We say that this family is a regular perturbation of $\Omega_0$ if one of the following equivalent conditions is satisfied:
\begin{enumerate}
 \item[(A)]
 There are diffeomorphisms $\Phi_h:\Omegabar_0\to\Omegabar_h$ so that $\Phi_h$ is smooth in $x\in\Omegabar_0$ and $h\geq0$, and  $\Phi_0=\Id_{\Omegabar_0}$.
 \item[(B)]
 The closure of the total space
  $$M=\Omegabar = \bigcup_{h\geq0}\Omegabar_h \times \{h\} \subset\R^2\times\R_+ $$ is a manifold with corners, with boundary hypersurfaces 
 $$ X:= \Omegabar_0,\quad \dDir M:=\bigcup_{h\geq0}\partial\Omegabar_h.$$ 
\end{enumerate}
See Figure \ref{fig:reg pert}.\footnote{To prove the equivalence of (A) and (B) note that $\Omegabar_0\times\R_+$ is a manifold with corners and that the
 $\Phi_h$ define a trivialization (diffeomorphism) $\Phi:\Omegabar_0\times\R_+\to\Omegabar, (x,h)\mapsto (\Phi_h(x),h)$, and conversely a trivialization defines $\Phi_h$.
 
(B) could also be reformulated as: $M$ is a p-submanifold of $\R^2\times\R_+$ }
Note that the two boundary hypersurfaces play different roles: At $\dDir M$, the \lq Dirichlet boundary\rq, we impose Dirichlet boundary conditions. The quasimode construction proceeds at $X=\{h=0\}$.
To unify notation, we denote the boundary of $X$ by $\dDir X$.

In the geometric spirit of this article, and to prepare for later generalization, we use condition (B). For explicit calculations the maps $\Phi_h$ in (A) are useful, as we indicate in Subsection \ref{sec:reg explicit}.  

As explained in Section \ref{sec:main steps} the quasimode construction problem is to find $u$ and $\lambda$ satisfying
$(P-\lambda)u=O(h^\infty)$, where $u$ is required to satisfy Dirichlet boundary conditions. Here $P=-\Delta$ on each $\Omega_h$.

\begin{figure}
\centering
\begin{tikzpicture}
 \draw (0,0) ellipse (2cm and .5cm);
 \draw (.2,1.3) ellipse (3cm and .6cm);
 \draw (-2,0) -- (-2.8,1.3);
 \draw (2,0) -- (3.2,1.3);
 
\begin{scope}[xshift=-1cm] 
 \draw[->] (-3,0) -- (-3,2) node[left]{$h$};
 \draw[->] (-3,0) -- (-2,0);
 \draw[->] (-3,0) -- (-3.5,-.5);
\end{scope}
 \node at (0,0){$X$};
 \node at (3,.5){$\dDir M$};
 
\end{tikzpicture}
 
\caption{The total space $M$ for a regular perturbation}
\label{fig:reg pert}
\end{figure}
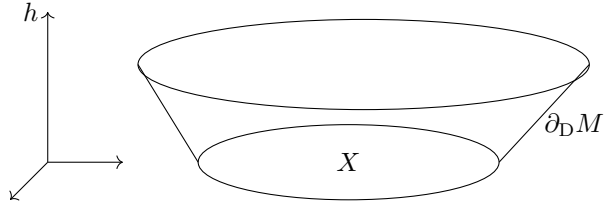
\subsection{Solution}
\label{subsec:reg solution}
The idea is this: Rather than solve $(P-\lambda)u=O(h^\infty)$ directly, we proceed inductively with respect to the order of vanishing of the right hand side:
\begin{quote}
\begin{tabular}{lll}
{\bf Initial step:} &  Find $\lambda,u$ satisfying & $(P-\lambda)u=O(h)$. \\[1mm]
{\bf Inductive step:} &  Given $\lambda,u$ satisfying & $(P-\lambda)u=O(h^k)$ where $k\geq1$,\\[.5mm] 
& find $\lambdatilde,\utilde$ satisfying & $(P-\lambdatilde)\utilde=O(h^{k+1})$.
\end{tabular}
\end{quote}
Before carrying this out, we prepare the stage. 
We structure the exposition of the details so that it parallels the later generalizations.
\subsubsection{Function spaces, leading part and model operator}
For a regular perturbation we expect  $\lambda$ and $u$ to be smooth up to $h=0$. Therefore we introduce the function spaces
\begin{align*}
\CinfD(M) &= \{u\in\Cinf(M):\ u=0\text{ at } \dDir M\} \\
\CinfD(X) &= \{v\in\Cinf(X):\ v=0\text{ at } \dDir X\} 
\end{align*}
and
$$ h^k\Cinf(M) = \{h^kf:\, f\in\Cinf(M)\},\quad h^\infty\Cinf(M) = \bigcap_{k\in\N} h^k\Cinf(M).$$
So $h^\infty\Cinf(M)$ is the space of smooth functions on $M$ vanishing to infinite order at the boundary $h=0$. For simplicity we always consider real-valued functions.

We seek $u\in\CinfD(M)$ for which the remainders $f=(P-\lambda)u$ lie in $h^k\Cinf(M)$ for $k=1,2,3,\dots$.
Our final goal is:
$$ \text{Find } \lambda\in\Cinf(\R_+),\ u\in\CinfD(M)
\text{ so that } (P-\lambda)u\in h^\infty\Cinf(M)\,.$$

The \defin{leading part} of  $u\in\CinfD(M)$ and of $f\in \Cinf(M)$ is defined to be the restriction to $h=0$:
$$ u_X :=u_{|X},\quad f_X:=f_{|X}.$$

The following lemma is obvious. In (a)  use Taylor's theorem.
\begin{leading part and model operator lemma}[regular perturbation]
\mbox{}
\begin{enumerate}
 \item[a)]
If $f\in\Cinf(M)$ then
$$ f\in h\,\Cinf(M) \ \text{ if and only if }\  f_X = 0.$$
 \item[b)]
 For $\lambda\in\Cinf(\R_+)$ we have
 $$ P-\lambda: \CinfD(M) \to \Cinf(M) $$
 and
 $$ [(P-\lambda)u]_X = (P_0-\lambda_0)u_X$$  
where $P_0=-\Delta$ is the Laplacian on $\Omega_0$ and $\lambda_0=\lambda(0)$.
\end{enumerate}
\end{leading part and model operator lemma}
We call 
$$P_0-\lambda_0:\CinfD(X)\to\Cinf(X)$$ 
the {\bf model operator} of $P-\lambda$, since it models its action at $h=0$.
Thus, the leading part of $(P-\lambda)u$ is obtained by applying the model operator to the leading part of $u$.
\begin{remark}
 In the uniform notation of Section \ref{sec:main steps}, see \eqref{eqn:exact sequences}, \eqref{eqn:comm diagram model op}, we have $\partial_0 M=X$ and
 $\calE(M) = \CinfD(M), \calE(\partial_0 M)=\CinfD(X), \calR(M)=\Cinf(M), \calR(\partial_0 M)=\Cinf(X)$ and $\LP(u)=u_X$, $\LP(f)=f_X$, $(P-\lambda)_0=P_0-\lambda_0$.
\end{remark}

\subsubsection{Analytic input for model operator}
The core analytic input in the construction of quasimodes is the following fact about $P_0$.
\begin{lemma}\label{lem:orth decomp}
Let $\lambda_0\in\R$ and $g\in\Cinf(X)$. Then there is a unique \mbox{$\gamma\in\Ker(P_0-\lambda_0)$} so that the equation
\begin{equation}
\label{eqn:reg pert model problem} 
 (P_0-\lambda_0)v = g + \gamma
\end{equation}
has a solution $v\in\CinfD(X)$.
Also, $\gamma=0$ if and only if $g\perp\Ker(P_0-\lambda_0)$.

The solution $v$ is unique up to adding an element of 
 $\Ker(P_0-\lambda_0).$
\end{lemma}

Note that the lemma is true for any elliptic, self-adjoint elliptic operator on a compact manifold with boundary.
\begin{proof}
By standard elliptic theory, self-adjointness of $P_0$ in $L^2(X)$ 
and elliptic regularity imply 
the orthogonal decomposition $\Cinf(X)= \Ran(P_0-\lambda_0)\oplus\Ker(P_0-\lambda_0)$. This implies the lemma.
\end{proof}

\subsubsection{Inductive construction of quasimodes}
\begin{indlist}
 \item[\bf Initial step]
 We want to solve 
\begin{equation}
 \label{eqn:intl step regular}
 (P-\lambda) u \in h\Cinf(M).
\end{equation}
 By the leading part and model operator lemma this is equivalent to \mbox{$[(P-\lambda)u]_X=0$} and then to
 $$ (P_0-\lambda_0)u_X=0.$$
 Therefore we choose
\begin{align*}
 \lambda_0 &= \text{ an eigenvalue of $P_0$} \\
 u_0 &= \text{ a corresponding eigenfunction}
\end{align*}
then any $u$ having $u_X=u_0$ will solve \eqref{eqn:intl step regular}.
For simplicity we make the\footnote{The method can be adjusted to the case $\dim \Ker(P_0-\lambda_0)>1$. The main difference is that generically, not every eigenfunction $u_0$ of $P_0$ will arise as a limit of quasimodes $u_h$ with $h>0$.} 
\begin{equation}
\label{eqn:simple ev}
 \text{{\bf Assumption:} the eigenspace } \Ker (P_0-\lambda_0)\text{ is one-dimensional.}
\end{equation}
 \item[\bf Inductive step]
\begin{inductivesteplemma}[regular perturbation]
Let $\lambda_0$, $u_0$ be chosen as in the initial step, and let $k\geq1$.
Suppose $\lambda\in\Cinf(\R_+)$, $u\in\CinfD(M)$ satisfy
$$ (P-\lambda)u \in h^k\Cinf(M)$$
and  $\lambda(0)=\lambda_0$, $u_X=u_0$.
Then there are $\mu\in\R$, $v\in \CinfD(M)$ so that
   $$  (P-\tilde\lambda)\tilde u \in h^{k+1}\Cinf(M)$$
 for $\tilde\lambda=\lambda+h^k\mu,$ $\tilde u=u+h^kv.$ The number
 $\mu$ is unique, and $v_X$ is unique up to adding constant multiples of $u_0$.
\end{inductivesteplemma}
More precisely, $\mu$ and $v_X$ (modulo $\R u_0$) are uniquely determined by $\lambda_0$, $u_0$ and the leading part of $h^{-k}(P-\lambda)u$.
\begin{proof}
 Writing $(P-\lambda)u=h^kf$ and $\tilde\lambda=\lambda+h^k\mu,\ \tilde u=u+h^kv$ we have
\begin{align*}
 (P-\lambdatilde)\utilde &= h^k [f - \mu u + (P-\lambda)v - h^k \mu v]
\end{align*}
This is in $h^{k+1}\Cinf(M)$ if and only if the term in brackets is in $h\Cinf(M)$, which by the leading part and model operator lemma (and by $k\geq1$) is equivalent to
$  f_X-\mu u_X + (P_0-\lambda_0)v_X = 0$, i.e.\ (using $u_X=u_0$) to
\begin{equation}
\label{eqn:f decomposition regular} 
 (P_0-\lambda_0)v_X =  -f_X + \mu u_0\,.
\end{equation}
This equation can be solved for $\mu$, $v_X$ by applying Lemma \ref{lem:orth decomp} to $g=-f_X$, since $\Ker(P_0-\lambda_0)=\{\mu u_0:\,\mu\in\R\}$ by \eqref{eqn:simple ev}. Having $v_X$ we extend it to a smooth function $v$ on $M$.
Lemma \ref{lem:orth decomp} also gives the uniqueness of $\mu$ and the uniqueness of $v_X$ modulo multiples of $u_0$.
\end{proof}

\end{indlist}

The initial and inductive steps give eigenvalues and quasimodes to any order $h^N$, and this is good enough for all purposes. It is still nice to go to the limit and also consider uniqueness. We get the final result:
\begin{theorem}[quasimodes for regular perturbation]
\label{thm:regular quasimodes}
Assume the setup of a regular perturbation as described in Section \ref{subsec:reg pert setup}.
 Given a simple eigenvalue $\lambda_0$ and associated eigenfunction $u_0$ of $P_0$, there are $\lambda\in\Cinf(\R_+)$, $u\in \CinfD(M)$  satisfying 
 $$(P-\lambda)u \in h^\infty \Cinf(M)$$
and
$$ \lambda(0)=\lambda_0,\quad u_X=u_0\,.$$
Furthermore, $\lambda$ and $u$ are unique in Taylor series at $h=0$, 
up to replacing $u$ by $a(h)u$ where $a$ is smooth and $a(0)=1$.
\end{theorem}
Clearly, $u$ cannot be unique beyond what is stated. 
\begin{proof}
Let $u_{(k)}$, $\lambda_{(k)}$ be as obtained in the initial step (if $k=0$) or the inductive step (if $k\geq1$), respectively. 
Then $u_{(k+1)}=u_{(k)} + O(h^k)$, $\lambda_{(k+1)}=\lambda_{(k)}+ O(h^k)$ for all $k$ by construction, so by asymptotic summation (Borel Lemma, cf.\ Lemma \ref{lem:borel})  we obtain $\lambda, u$ as desired.

To prove uniqueness, we show inductively that for $\lambda,\lambda'$ and $u,u'$ having the same leading terms, the assumptions
$ (P-\lambda)u\in h^k\Cinf(M),\ (P-\lambda')u'\in h^k\Cinf(M)$ imply that
$ \lambda-\lambda' = O(h^k)$ and $u-a_{(k)}(h)u'\in h^k\CinfD(M)$ for a smooth function $a_{(k)}$, $a_{(k)}(0)=1$.

For $k=1$ there is nothing to prove. Suppose the claim is true for $k$, and let
$(P-\lambda)u\in h^{k+1}\Cinf(M),\ (P-\lambda')u'\in h^{k+1}\Cinf(M)$. By the inductive hypothesis, we have $ \lambda-\lambda' = O(h^k)$ and $u-a_{(k)}(h)u'\in h^k\CinfD(M)$. Since the leading terms of $h^{-k}(P-\lambda)u$ and $h^{-k} (P-\lambda')u'$ both vanish, the uniqueness statement in the inductive step lemma implies that $\lambda-\lambda'=O(h^{k+1})$ and $u-a_{(k)}(h)u' - ch^ku_0\in h^{k+1}\CinfD(M)$ for some $c\in\R$. Then $a_{(k+1)}(h)=a_{(k)}(h)+ch^{k}$ satisfies $u-a_{(k+1)}(h)u'\in h^{k+1}\CinfD(M)$.
Now define $a$ from the $a_{(k)}$ by asymptotic summation.
%

\end{proof}

\subsection{Explicit formulas}
\label{sec:reg explicit}
The proof of Theorem \ref{thm:regular quasimodes} is constructive: it gives a method for finding $u(x,h)$ and $\lambda(h)$ to any order in $h$, under the assumption that the model problem \eqref{eqn:reg pert model problem} can be solved.
We present two standard alternative ways of doing the calculation.

We use the maps $\Phi_h:X\to \Omegabar_h$, see (A) in Section \ref{subsec:reg pert setup}, where $X=\Omegabar_0$. In fact, only the restriction of $\Phi_h$ to $\dDir X$ is needed, as will be clear from the first method presented below.
\subsubsection{Boundary perturbation}
We will compute $\dot\lambda$, where the dot denotes the first derivative in $h$ at $h=0$. This is the first order perturbation term since 
$\lambda(h) = \lambda(0)+h\dot\lambda + O(h^2)$.
Differentiating the equation $(P-\lambda)u\in h^2\Cinf(M)$ in $h$ at $h=0$ we obtain
\begin{equation}
 \label{eqn:reg perturb explicit 1}
 (\dot P-\dot\lambda)u_0 + (P_0-\lambda_0)\dot u =0\,.
\end{equation}
In our case $\dot P=0$. The boundary condition is
$u(x,h)=0$ for all $x\in\dDir\Omega_h$ and all $h$, so $u(\Phi_h(y),h)=0$ for $y\in\dDir X$. Differentiating in $h$ yields
the boundary condition for $\dot u$:
$$ Vu_0+\dot u = 0 \quad\text{on } \dDir X$$
where $Vu_0$ is the derivative of $u_0$ in the direction of the vector field $V=(\partial_{h}\Phi_h)_{|h=0}$. Now take the $L^2(X)$ scalar product of \eqref{eqn:reg perturb explicit 1} with $u_0$. 
We write the second summand using
Green's formula as
$$\langle (P_0-\lambda_0)\dot u,u_0\rangle = \int_{\dDir X} \left(-\partial_n \dot u\cdot u_0 + \dot u\cdot \partial_n u_0\right)\, dS + \langle \dot u, (P_0-\lambda_0)u_0\rangle $$
where $\partial_n$ denotes the outward normal derivative. Using $u_{0|\dDir X}=0$, $(P_0-\lambda_0)u_0=0$ we obtain 
$$ \dot\lambda = - \frac1{\|u_0\|^2} \int_{\dDir X} Vu_0\cdot \partial_n u_0 \,dS $$
where $\|u_0\|$ is the $L^2(X)$-norm of $u_0$.
Commonly one chooses $\Phi_h$ so that $V=a\partial_n$ for a function $a$ on $\dDir X$. This means that the boundary is perturbed in the vertical direction at velocity $a$. For $L^2$-normalized $u_0$ this yields Hadamard's formula (see \cite{Had:MPAREPEE}) $\lambdadot=- \int_{\dDir X} a (\partial_n u_0)^2 \,dS$. Higher order terms are computed in a similar way.

Note that we did not need to solve the model problem. Its solution is only needed to compute $\dot u$ or higher derivatives of $\lambda$ and $u$.
\subsubsection{Taylor series ansatz}
\label{subsubsec:taylor}
Here is a different method where in a first step all operators are transferred to the $h$-independent space $X$.
Using the maps $\Phi_h:X\to\Omegabar_h$ pull back the operator $P_h$ to $X$:
$$ P_h' = \Phi_h^*P_h\,.$$
Now $P_h'$ is a smooth family of elliptic operators on $X$, so we can write
$$ P_h' \sim P_0+hP_1+\dots$$
Here $P_0$ is the Laplacian on $X$ since $\Phi_0$ is the identity.
We also make the ansatz
$$u\sim u_0+hu_1+\dots,\quad \lambda\sim\lambda_0+h\lambda_1+\dots $$
where all $u_i\in\CinfD(X)$, multiply out the left side of 
$$ (P_0+hP_1+\dots - \lambda_0-h\lambda_1-\dots)(u_0+hu_1+\dots) \sim 0, $$
order by powers of $h$ and equate each coefficient to zero.
The $h^0$ term gives the initial equation
\begin{equation}
 \label{eqn:initial}
 (P_0-\lambda_0) u_0 = 0
\end{equation}
and the $h^k$ term, $k\geq1$, gives the recursive set of equations
\begin{equation}
 \label{eqn:iteration0}
\begin{aligned}
 (P_0-\lambda_0)u_k &= - (P_1-\lambda_1)u_{k-1}-\dots-(P_k- 
\lambda_k)u_0\\
&=: - f_k + \lambda_k u_0
\end{aligned}
\end{equation}
where $f_k$ is determined by $u_0,\dots,u_{k-1}$ and $\lambda_0,\dots,\lambda_{k-1}$.
This is the decomposition of Lemma \ref{lem:orth decomp} for $g=f_k$, so it can be solved for $\lambda_k$, $u_k$.

\smallskip

We can solve \eqref{eqn:iteration0} explicitly as follows:
Taking the scalar product with $u_0$ and using $\langle (P_0-\lambda_0)u_k,u_0\rangle = \langle u_k,(P_0-\lambda_0)u_0\rangle =0$ we get
\begin{equation}
\label{eqn:choice lambdak} 
 \lambda_k = \frac{\langle f_k,u_0\rangle}{\|u_0\|^2} \,,
\end{equation}
for example
$$ \lambda_1 = \frac{\langle P_1 u_0,u_0\rangle}{\|u_0\|^2},\quad \lambda_2 = \frac{\langle (P_1-\lambda_1)u_1 + P_2 u_0,u_0\rangle}{\|u_0\|^2}$$
Here $u_1$, $u_2$ etc. are computed as
$$ u_k = (P_0-\lambda_0)^{-1} (-f_k + \lambda_k u_0)$$
where $(P_0-\lambda_0)^{-1}$ is a generalized inverse of $P_0-\lambda_0$, i.e. a left inverse defined on $\Ran(P_0-\lambda_0)$. The choice \eqref{eqn:choice lambdak} of $\lambda_k$ guarantees that $-f_k + \lambda_k u_0\in\Ran(P_0-\lambda_0)$.

\begin{remark}
This method seems simpler and more effective than the one presented in Section \ref{subsec:reg solution}.  However, in the context of singular perturbations, where several model problems occur, it will pay off to have a geometric view and not to have to write down asymptotic expansions.

The relation between these two methods becomes clearer if we formulate the present one in terms of the operator $P'$ on the space $\Omegabar'=X\times\R_+$. The product structure of $\Omegabar'$ allows us to extend functions on $X$ to functions on $\Omegabar'$ in a canonical way (namely, constant in $h$). This yields the explicit formulas. In comparison, for $\Omegabar$ there is no such canonical extension. 
 
\end{remark}

\subsection{Generalizations}
\label{subsec:reg pert ex}
Theorem \ref{thm:regular quasimodes} generalizes to any smooth family of uniformly elliptic operators $P_h$ with elliptic boundary conditions on a compact manifold with boundary, supposing  $P_0$ is self-adjoint. Note that $P_h$ for $h>0$ need not be self-adjoint. If $P_h$ has complex coefficients then $u$ and $\lambda$ will be complex valued, and if all $P_h$ are self-adjoint then $\lambda$ can be chosen real-valued.

The method in Subsection \ref{subsubsec:taylor} can be formulated abstractly for any family of operators $P_h'$ on a Hilbert space which has a regular Taylor expansion in $h$ as $h\to0$.
Using contour integration one may find  the asymptotics of eigenfunctions and eigenvalues, not just quasimodes, directly and show that they vary smoothly in the parameter $h$ under the simplicity assumption \eqref{eqn:simple ev}. See \cite{Kat:PTLO}. 

\section{Adiabatic limit with constant fibre eigenvalue}
\label{sec:adiab limit const}

The adiabatic limit\footnote{The word adiabatic originally refers to physical systems that change slowly. In their quantum mechanical description structures similar to the ones described here occur, where $x$ corresponds to time and $h^{-1}$ to the time scale of unit changes of the system. This motivated the use of the word adiabatic limit in global analysis in this context.}
 is a basic type of singular perturbation which will be part of all settings considered later. Its simplest instance is the Laplacian on the family of domains
\begin{equation}
\label{eqn:adiab ex0} 
 \Omega_h =(0,1) \times (0,h) \subset\R^2\,.
\end{equation}
Since the domain of the variable $y$ is $(0,h)$ it is natural to use the variable $Y=\frac yh\in (0,1)$ instead. Then 
\begin{equation}
\label{eqn:adiab ex0 Delta} 
 \Delta = \partial_x^2+\partial_y^2 = h^{-2} \partial_Y^2 + \partial_x^2\,.
\end{equation}

Although it is not strictly needed for understanding the calculations below, we explain how this is related to blow-up,
in order to prepare for later generalizations:
The closure of the total space $\Omega=\bigcup_{h>0} \Omega_h\times\{h\}\subset\R^2\times\R_+$ has a singularity (an edge) at $h=0$.\footnote{%
\label{footnote:not dsubman}
The precise meaning of this is that $\Omegabar$ is not a d-submanifold of $\R^2\times\R_+$, as defined before Definition \ref{def:resol subsets}. This is what distinguishes it from a regular perturbation. Note that $\Omegabar$ happens to be a submanifold with corners of $\R^3$, but this is irrelevant here.}
This singularity can be resolved by blowing up the $x$-axis $L=\{y=h=0\}$ in $\R^2\times\R_+$. 
If $\beta:[\R^2\times\R_+,L]\to\R^2\times\R_+$ is the blow-down map then the lift
$$ M = \beta^*\Omegabar$$
is contained in the domain of the projective coordinates system $x,Y=\frac yh, h$, compare Figure \ref{fig:proj coords examples}(d). In these coordinates the set $M$ is given by $x\in[0,1], Y\in[0,1], h\in\R_+$. See Figure \ref{fig:adiab limit b-up}.
Note that the operators $\Delta$ turn into the \lq singular' family of operators \eqref{eqn:adiab ex0 Delta} on $M$.
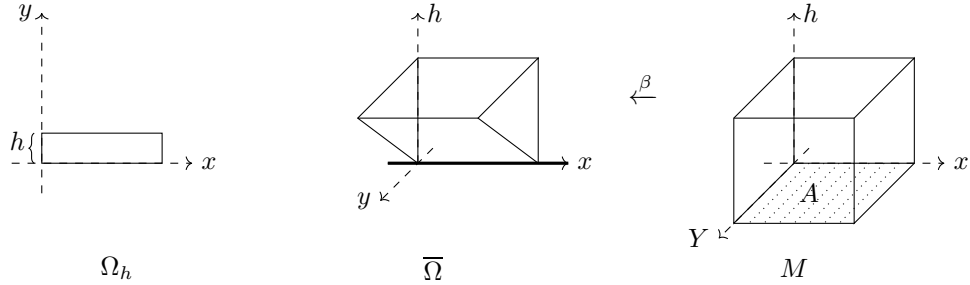
\begin{figure}

\begin{tikzpicture}

\def\xmin{0}
\def\xmax{.8}
\def\ymax{.8}
\def\yh{.2}  

\begin{scope}[xshift=-5cm,scale=2]

 \draw[dashed,->] (-.2,0,0) -- (1,0,0) node[right]{$x$};
 \draw[dashed,->] (0,-.2,0) -- (0,1,0) node[left]{$y$};

 \draw (\xmin,0) -- (\xmin,\yh) -- (\xmax,\yh) -- (\xmax,0) -- cycle;
 
 \draw[decorate,decoration={brace,amplitude=2pt}]
(-.05,0) -- (-.05,.2) node [near end,left]{$h$};

 
 \node at (.5, -.7){$\Omega_h$};
\end{scope}

\begin{scope}[scale=2,x={(1cm,0)},y={(-.5cm,-.5cm)},z={(0,1cm)}]

 \draw[dashed,->] (-.2,0,0) -- (1,0,0) node[right]{$x$};
 \draw[dashed,->] (0,-.2,0) -- (0,.5,0) node[left]{$y$};
 \draw[dashed,->] (0,0,0) -- (0,0,1) node[right]{$h$}; 
 
 \draw[very thick] (-.2,0,0) -- (1,0,0);
 
 \draw(\xmin,0,0) -- (\xmin,\ymax,.7) -- (\xmin,0,.7) -- cycle;
 \draw (\xmax,0,0) -- (\xmax,\ymax,.7) -- (\xmax,0,.7) -- cycle;
 \draw (\xmin,\ymax,.7) -- (\xmax,\ymax,.7);
 \draw (\xmin,0,.7) -- (\xmax,0,.7); 

 \node at (.1,0, -.7){$\Omegabar$};

\end{scope}

  \node at (3,1) {$\xleftarrow{\beta}$};
  
\begin{scope}[xshift=5cm,scale=2,x={(1cm,0)},y={(-.5cm,-.5cm)},z={(0,1cm)}]

 \draw[dashed,->] (-.2,0,0) -- (1,0,0) node[right]{$x$};
 \draw[dashed,->] (0,-.2,0) -- (0,1,0) node[left]{$Y$};
 \draw[dashed,->] (0,0,0) -- (0,0,1) node[right]{$h$};

 \draw(\xmin,0,0) -- (\xmin,\ymax,0) -- (\xmin,\ymax,.7) --  (\xmin,0,.7) -- cycle;
 \draw (\xmax,0,0) -- (\xmax,\ymax,0) -- (\xmax,\ymax,.7) -- (\xmax,0,.7) -- cycle;
 \draw (\xmin,\ymax,.7) -- (\xmax,\ymax,.7);
 \draw (\xmin,0,.7) -- (\xmax,0,.7); 
 \draw (\xmin,\ymax,0) -- (\xmax,\ymax,0);
 \draw (\xmin,0,0) -- (\xmax,0,0);
 
  \foreach \x in {.1,.2,.3,.4,.5,.6,.7} 
 {
 \draw[dotted] (\x,0,0) -- (\x,\ymax,0);
 }

 \node at (\xmax/2-.1, \ymax/2,0) {$A$};

  \node at (0,0,-.7){$M$};

\end{scope}

\end{tikzpicture}

 \caption{Domain $\Omega_h$, total space $\Omegabar$ and resolution $M$ of $\Omegabar$ for adiabatic limit. Compare Figure \ref{fig:proj coords examples}(d). On the right only the part of the blown-up space $[\R^2\times\R_+,\{y=h=0\}]$ where the \lq top\rq\ projective coordinates $h,Y=\frac yh$ are defined is shown. Dotted lines are fibres of the natural fibration of the front face $A$.}
 \label{fig:adiab limit b-up}
\end{figure}

This example, and the generalization needed in Section \ref{sec:adiab variable}, motivates considering the following setting. See Section \ref{subsec:adiab examples} for more examples where this setup occurs.

\subsection{Setup}
Suppose $B,F$ are compact manifolds, possibly with boundary. For the purpose of this article you may simply take $B,F$ to be closed intervals
(but see Subsection \ref{subsubsec:adiab noncompact} for a generalization needed later).
We consider a family of differential operators depending on $h>0$
\begin{equation}
\label{eqn:P adiabatic} 
 P(h) \sim h^{-2}P_F + P_0+hP_1+\dots
\end{equation}
on $A=B\times F$. We assume
\begin{equation}
\label{eqn:adiab assn F} 
 P_F \ \text{ is a self-adjoint elliptic operator on $F$}
\end{equation}
where boundary conditions are imposed if $F$ has boundary.\footnote{Formally it would be more correct to write $\Id_B\otimes P_F$ instead of $P_F$ in \eqref{eqn:P adiabatic}, but here and in the sequel we will use the simplified notation.}
 For example, if $F=[0,1]$ then we could take $P_F=-\partial_Y^2$ with Dirichlet boundary conditions. $P_0,P_1,\dots$ are differential operators on $A$. 
A condition on $P_0$ will be imposed below, see equations \eqref{eqn:PB def}, \eqref{eqn:PB assn}. 

One should think of $A$ as the union of the fibres (preimages of points) of the projection $\pi:A=B\times F\to B$, i.e.\ $A=\bigcup\limits_{x
\in B}\{x\}\times F$, see also Remark \ref{rem:adiab const fibres} below. We call $F$ the fibre and $B$ the base. The analysis below generalizes to the case of fibre bundles $A\to B$, see Section \ref{subsec:adiab-generaliz}. The letter $A$ is used for \lq adiabatic limit\rq.

We will denote coordinates on $B$ by $x$ and on $F$ by $Y$. This may seem strange but serves to unify notation over the whole article, since this notation is natural in the following sections.

\subsection{What to expect: the product case}
To get an idea what happens, we consider the case of a product operator, i.e.
$$ P(h) = h^{-2}P_F + P_B$$
where
$P_B$, $P_F$ are second order  elliptic operators on $B$ and $F$, self-adjoint with given boundary conditions. An example is  \eqref{eqn:adiab ex0}, \eqref{eqn:adiab ex0 Delta} where $B=F=[0,1]$ and $P_B=-\partial_x^2$, $P_F=-\partial_Y^2$. More generally, $P_B$, $P_F$ could be the Laplacians on compact Riemannian manifolds $(B,g_B)$, $(F,g_F)$.  Then $P$ would be the Laplacian on $A$ with respect to the metric $h^2 g_F \oplus g_B$ in which the lengths in $F$-direction are scaled down by the factor $h$. 

By separation of variables $P(h)$ has the eigenvalues $\lambda_{k,l}=h^{-2}\lambda_{F,k}+\lambda_{B,l}$ where $\lambda_{F,k}$, $\lambda_{B,l}$ are the eigenvalues of $P_F$, $P_B$ respectively, with eigenfunctions\footnote{For functions $\phi:B\to\R$ and $\psi:F\to\R$ we write $\phi\otimes\psi:B\times F\to \R$, $(x,Y)\mapsto \phi(x)\psi(Y)$.
}
 $\phi_k\otimes\psi_l$.

Although we have solved the problem, we now rederive the result using  formal expansions, in order to distill from it essential features that will appear in the general case.
We make the ansatz
$$  u= u_0+hu_1+\dots,\ \lambda=h^{-2}\lambda_{-2}+\dots $$
and plug in
\begin{equation}
\label{eqn:adiab expansion} 
 (h^{-2}P_F+P_B  - h^{-2}\lambda_{-2}-h^{-1}\lambda_{-1}-\lambda_0-\dots)
(u_0+hu_1+h^2u_2+\dots) = 0.
\end{equation}
The $h^{-2}$ term gives
\begin{equation}
 \label{eqn:initial adiab}
 (P_F-\lambda_{-2})u_0=0
\end{equation}
so $\lambda_{-2}$ must be an eigenvalue of $P_F$. Suppose it is simple and let $\psi$ be a normalized eigenfunction. It follows that
$$u_0(x,Y) = \phi(x) \psi(Y) $$
for some yet unknown function $\phi$. How can we find $\phi$?
%
%
The $h^{-1}$ term gives $(P_F-\lambda_{-2})u_1 = \lambda_{-1}u_0$. Taking the scalar product with $u_0$ and using self-adjointness of $P_F$ we get $\lambda_{-1}=0$.
The $h^0$ term then gives
$$ (P_F-\lambda_{-2})u_2 = -(P_B-\lambda_0)u_0.$$
By Lemma \ref{lem:orth decomp}, applied to $P_F$ for fixed $x\in B$, this has a solution $u_2$ if and only if
\begin{equation}
 \label{eqn:adiab orth cond}
 (P_B-\lambda_0) u_0(x,\cdot) \perp \psi\quad\text{ in } L^2(F)\ \text{ for each }x\in B.
\end{equation}
Now the left side is $[(P_B-\lambda_0)\phi(x)]\, \psi$, so we get
$$ (P_B-\lambda_0)\phi = 0\ \text{ on }B.$$
Thus, $\lambda_0$ is an eigenvalue of $P_B$ with eigenfunction $\phi$. This solves the problem since $\phi\otimes\psi$ is clearly an eigenfunction of $P(h)$ with eigenvalue $h^{-2}\lambda_{-2}+\lambda_0$.

From these considerations, we see basic features of the adiabatic problem:
\begin{itemize}
 \item $\lambda_{-2}$ is an eigenvalue of the fibre operator $P_F$.
 \item $\lambda_0$ is an eigenvalue of the base operator $P_B$.
 \item The leading term of the eigenfunction, $u_0$, is the tensor product of the eigenfunctions on fibre and base. It is determined from the  two \lq levels', $h^{-2}$ and $h^0$ of \eqref{eqn:adiab expansion}.
\end{itemize}

For a general operator \eqref{eqn:P adiabatic} we cannot separate variables since $P_0$ (and the higher $P_i$) may involve $Y$-derivatives (or $Y$-dependent coefficients). However, the \lq adiabatic\rq\ structure of $P(h)$ still allows separation of variables to leading order:
The $h^{-2}$ term of \eqref{eqn:adiab expansion} still yields $u_0=\phi\otimes\psi$ and the $h^{-1}$ term yields $\lambda_{-1}=0$. The $h^0$ term now 
yields condition \eqref{eqn:adiab orth cond} with $P_B$ replaced by $P_0$. This shows that $\phi$ must be an eigenfunction of the operator
$$U \mapsto (\Pi \circ P_0) (U\otimes \psi)$$
where $\Pi u=\langle u,\psi\rangle_F$ is the $L^2(F)$ scalar product with $\psi$. This motivates the definition of the horizontal operator $P_B$ below.

\subsection{Solution}
The solution of the formal expansion equation \eqref{eqn:adiab expansion} is complicated by the fact that a single $u_i$ is only determined using several $h^k$. 
It is desirable
to avoid this, in order to easily progress to more complex problems afterwards. Thus, we need a procedure where consideration of a fixed $h^k$ gives full information on the corresponding next term in the $u$ expansion.

This can be achieved by redefining the function space containing the remainders $f=(P-\lambda)u$ in the iteration, as well as their notion of leading part. 

As before, we consider a family $(u_h)_{h\geq0}$ of functions on $B\times F$ as one function on the total space 
$$M= B\times F\times \R_+$$ 
and consider a differential operator $P$ acting on functions on $M$ and having an expansion as in \eqref{eqn:P adiabatic}. 
Let
$$ A :=B\times F\times\{0\}$$
be the boundary at $h=0$ of $M$.

\subsubsection{A priori step: Fixing a vertical mode. The horizontal operator.}
{A priori} we fix 
\begin{equation}
 \label{eqn:adiabatic a priori}
\begin{aligned}
 \lambda_{-2} &= \text{a simple eigenvalue  of } P_F\\
 \psi &= \text{an $L^2(F)$-normalized corresponding eigenfunction}.
\end{aligned}
\end{equation}
  We will seek (quasi-)eigenvalues of $P$ of the form $h^{-2}\lambda_{-2}+\Cinf(\R_+)$.

  Every $f\in \Cinf(F)$ may be decomposed into a $\psi$ component and a component perpendicular to $\psi$:
\begin{equation}
 \label{eqn:fibre-decomp}
 f = \langle f,\psi \rangle_{F}\,\psi + f^\perp, \quad f^\perp \perp_F \psi
\end{equation}
where $\langle\ ,\ \rangle_F$ is the $L^2(F)$ scalar product.
The same formula defines a fibrewise decomposition of $f$ in $\Cinf(A)$ or in $h^k\Cinf(M)$, $k\in\Z$.
The coefficient of $\psi$ defines projections
\begin{equation*} 
\begin{aligned}
\Pi:&\ \Cinf(A)\to\Cinf(B) \\
\Pi:&\  h^k\Cinf(M) \to h^k\Cinf(B\times \R_+) 
\end{aligned}
\qquad
 f \mapsto \langle f,\psi\rangle_{F}
\end{equation*}
By self-adjointness of $P_F$
\begin{equation}
\label{eqn:Pi P_F} 
 \Pi \circ (P_F - \lambda_{-2}) = 0
\end{equation}
on the domain of $P_F$.
Motivated by the consideration at the end of the previous section we define the {\bf horizontal operator}\footnote{$P_B$ is also called the {\em effective Hamiltonian}, e.g. in \cite{Teu:APTQD}.}
\begin{equation}
\label{eqn:PB def} 
 P_B : \CinfD(B)\to \Cinf(B),\quad
U \mapsto \Pi P_0 (U\otimes\psi).
\end{equation}
We can now formulate the assumption on $P_0$:
\begin{equation}
\label{eqn:PB assn} 
\text{$P_B$ is a self-adjoint elliptic differential operator on $B$}
\end{equation}
where self-adjointness is with respect to some fixed density on $B$ and given boundary conditions.
This notation is consistent with the use of $P_B$ in the product case.
\subsubsection{Function spaces, leading part and model operator}
We will seek quasimodes $u$  in the {\bf solution space} $\CinfD(M)$, the space of smooth functions on $M$ satisfying the boundary conditions. The {\bf leading part} of $u\in\CinfD(M)$ is defined to be
$$u_A:=u_{|h=0} \in\CinfD(A).$$

The following definition captures the essential properties of the remainders $f=(P-\lambda)u$ arising in the iteration.

\begin{definition}
\label{def:adiab limit rem space} 
The {\bf remainder space} for the adiabatic limit is
\begin{align*}
 \calR (M) &:= \{ f\in h^{-2}\Cinf(M):\, \Pi f\ \text{is smooth at }h=0\}\\
 & \phantom{:}= \{f = h^{-2}f_{-2}+h^{-1}f_{-1}+\dots:\ \ \Pi f_{-2} = \Pi f_{-1}=0\}. 
\end{align*}
 The {\bf leading part} of $f\in \calR(M)$, $f=h^{-2}f_{-2}+h^{-1}f_{-1}+\dots$ is\footnote{The notation $f_{AB}$ is meant to indicate that the leading part has components which are functions on $A$ and on $B$.
}
 $$ f_{AB} := 
\begin{pmatrix}
 f_{-2} \\ \Pi f_0 
\end{pmatrix}
\ 
\in \Cinf(A)_{\Pi^\perp}\oplus \Cinf(B)
$$
where 
$$\Cinf(A)_{\Pi^\perp} :=\{v\in\Cinf(A):\,\Pi v=0\}\,.$$
\end{definition}
\noindent For functions in the solution space we clearly have:
$$ \text{Let }u\in\CinfD(M). \text{ Then }u\in h\CinfD(M) \iff u_A=0\,.$$
The definition of the leading part of $f\in\calR(M)$ is designed to make the corresponding fact for $f$ true:
\begin{leading part and model operator lemma}[adiabatic limit]
\mbox{}
\begin{enumerate}
 \item[a)]
 If $f\in\calR(M)$ then 
 $$f\in h\calR(M)\ \text{ if and only if }\  f_{AB}=0.$$ 
 \item[b)]
For $\lambda\in h^{-2}\lambda_{-2}+\Cinf(\R_+)$ we have
\begin{equation}
\label{eqn:P-l mapping} 
 P-\lambda : \CinfD(M) \to \calR(M)
\end{equation}
and
\begin{equation}
\label{eqn:model op adiabatic}
[(P-\lambda)u]_{AB} = \begin{pmatrix}
 (P_F-\lambda_{-2})u_A\\
 \Pi (P_0-\lambda_0)u_A
\end{pmatrix}
\end{equation}
where  $\lambda_0$ is the constant term of $\lambda$.
\end{enumerate}
\end{leading part and model operator lemma}
The operator $(P-\lambda)_A:= \begin{pmatrix}
 P_F-\lambda_{-2}\\
 \Pi (P_0-\lambda_0)
\end{pmatrix}
$ is called the {\bf model operator} for $P-\lambda$ at $A$.
\begin{proof}
\mbox{}
\begin{enumerate}
 \item[a)]
 Let $f=h^{-2}f_{-2}+h^{-1}f_{-1}+f_0+\dots$ with $\Pi f_{-2} = \Pi f_{-1}=0$.
 Suppose $f_{AB}=0$, so $f_{-2}=0$ and $\Pi f_0=0$. Then  
 $f=h^{-1}f_{-1} + f_0+O(h)$ with $\Pi f_{-1}=\Pi f_0=0$, so $f\in h\calR(M)$. The converse is obvious.
\item[b)]
If $u\in\CinfD(M)$ then $(P-\lambda)u=h^{-2}(P_F-\lambda_{-2})u + (P_0-\lambda_0)u + O(h)$ is in $\calR(M)$ by \eqref{eqn:Pi P_F}, and then the definition of leading part implies \eqref{eqn:model op adiabatic}.
\end{enumerate}
\end{proof}
 
\begin{remark}
 In the uniform notation of Section \ref{sec:main steps}, see \eqref{eqn:exact sequences}, \eqref{eqn:comm diagram model op}, we have $\partial_0 M=A$ and
 $\calE(M) = \CinfD(M)$, $\calE(\partial_0 M)=\CinfD(A)$, $\calR(M)$ is defined in Definition \ref{def:adiab limit rem space}, $\calR(\partial_0 M) = \Cinf(A)_{\Pi^\perp}\oplus \Cinf(B)$, and $\LP(u)=u_A$, $\LP(f)=f_{AB}$, $(P-\lambda)_0=(P-\lambda)_A$.
\end{remark}

\subsubsection{Analytic input for model operator}
For the iterative construction of quasimodes we need the solution properties of the model operator, analogous to Lemma \ref{lem:orth decomp}. The main additional input is the triangular structure of the model operator, equation \eqref{eqn:adiab model triangular} below.

 By definition 
 $$(P-\lambda)_A:\CinfD(A) \to \Cinf(A)_{\Pi^\perp}\oplus \Cinf(B)$$ 
In the proof below it will be important to decompose functions $v\in\CinfD(A)$ into their fibrewise $\Pi^\perp$ and $\Pi$ components. More precisely, the decomposition \eqref{eqn:fibre-decomp} defines an isomorphism
\begin{equation}
 \label{eqn:adiab pi piperp decomp}
\CinfD(A) \cong \CinfD(A)_{\Pi^\perp} \oplus\CinfD(B)\,,\quad v\mapsto (v^\perp, \Pi v)
\end{equation}
so that $v= v^\perp + (\Pi v) \otimes \psi$.
\begin{lemma}
\label{lem:adiab analysis}
 Let $P$ be an operator on $M$ having an expansion as in \eqref{eqn:P adiabatic}, \eqref{eqn:adiab assn F}, and assume $P$ and $\lambda_{-2}\in\R$, $\psi\in\CinfD(F)$ satisfy \eqref{eqn:adiabatic a priori}, \eqref{eqn:PB assn}.
 
 Then for each $g\in \Cinf(A)_{\Pi^\perp} \oplus\Cinf(B)$ and $\lambda_0\in\R$ there is a unique $\gamma\in \Ker(P_B-\lambda_0)\subset\CinfD(B)$ so that the equation
 $$ (P-\lambda)_A v = g + 
\begin{pmatrix}
 0 \\ \gamma
\end{pmatrix}
 $$
 has a solution $v\in\CinfD(A)$. This solution is unique up to adding 
 $w\otimes\psi$ where $w\in\Ker(P_B-\lambda_0)$.
\end{lemma}
\begin{proof}
Decompose $v\in\CinfD(M)$ as in \eqref{eqn:adiab pi piperp decomp}. Then $(P_F-\lambda_{-2})v = (P_F-\lambda_{-2})v^\perp$ and 
$\Pi(P_0-\lambda_0)v = \Pi P_0 v^\perp + (P_B-\lambda_0)\Pi v$ since $\Pi v^\perp=0$ and by definition of $P_B$.  
 Therefore, we may write $(P-\lambda)_A$ as a $2\times2$ matrix:
\begin{equation}
\label{eqn:adiab model triangular}
 (P-\lambda)_A = 
\begin{pmatrix}
 P_F-\lambda_{-2} & 0 \\ & \\
 \Pi P_0 & P_B - \lambda_0
\end{pmatrix}
: \ \ 
\begin{matrix}
\CinfD(A)_{\Pi^\perp}\\ \oplus\\ \CinfD(B) 
\end{matrix}
\ \to \ 
\begin{matrix}
\Cinf(A)_{\Pi^\perp}\\ \oplus\\ \Cinf(B) 
\end{matrix}
\end{equation}
In order to solve $ (P-\lambda)_A v = g + \begin{pmatrix}
 0 \\ \gamma
\end{pmatrix}
$ we write $v=
\begin{pmatrix}
 v^\perp\\ v_\Pi
\end{pmatrix}
$ and $g=
\begin{pmatrix}
 g^\perp\\ g_\Pi
\end{pmatrix}
$ and get the system
\begin{align*}
 (P_F-\lambda_{-2}) v^\perp  &= g^\perp \\
 \Pi P_0 v^\perp + (P_B-\lambda_0)v_\Pi &= g_\Pi + \gamma
\end{align*}
The first equation has a unique solution $v^\perp$ by Lemma \ref{lem:orth decomp} applied to $P_F$.
Then  by Lemma \ref{lem:orth decomp} applied to $P_B$, there is a unique $\gamma\in\Ker(P_B-\lambda_0)$ so that the second equation has a solution $v_\Pi$, and $v_\Pi$ is unique modulo $\Ker(P_B-\lambda_0)$.
\end{proof}

\subsubsection{Inductive construction of quasimodes}

We now set up the iteration.
\begin{indlist}
\item[\bf Initial step:] We want to solve 
\begin{equation}
\label{eqn:adiabatic intl step} 
 (P-\lambda)u \in h \calR(M).
\end{equation}
By the leading part and model operator lemma this is equivalent to \mbox{$[(P-\lambda)u]_{AB}=0$} and then to
$$ (P_F-\lambda_{-2})u_A =0,\quad \Pi(P_0-\lambda_0)u_A=0.$$
By \eqref{eqn:adiabatic a priori} the first equation implies $u_A=\phi\otimes\psi$ for some function $\phi$ on $B$, and then the second equation is equivalent to $(P_B-\lambda_0)\phi=0$ by \eqref{eqn:PB def}, so if we choose
\begin{align*}
 \lambda_0 &= \text{an eigenvalue of } P_B \\
 \phi &= \text{a corresponding eigenfunction of $P_B$}
\end{align*}
then any $u$ having $u_A=\phi\otimes\psi$ satisfies \eqref{eqn:adiabatic intl step}. Again, we make the
\begin{equation}
 \label{eqn:adiab PB simple ev}
 \text{{\bf Assumption:} the eigenvalue $\lambda_0$ of $P_B$ is simple}
\end{equation}
From now on, we fix the following data:
$$ \lambda_{-2},\ \lambda_0 \in\R,\quad u_0 := \phi\otimes\psi\in\CinfD(A).$$
\item[\bf Inductive step:]
\begin{inductivesteplemma}[adiabatic limit]
Let $\lambda_{-2}$, $\lambda_0$ and $u_0$ be as above, and let $k\geq1$.
Suppose  $\lambda\in h^{-2}\Cinf(\R_+)$, $u\in \CinfD(M)$ satisfy
 $$(P-\lambda)u\in h^k \calR(M)$$
 and $\lambda = h^{-2}\lambda_{-2}+\lambda_0+O(h)$, $u_A=u_0$. Then there are $\mu\in\R$, $v\in \CinfD(M)$ so that
  $$  (P-\tilde\lambda)\tilde u \in h^{k+1}\calR(M)$$
 for $\tilde\lambda=\lambda+h^k\mu,\ \tilde u=u+h^kv.$
  The number
 $\mu$ is unique, and $v_A$ is unique up to adding constant multiples of $u_0$.
\end{inductivesteplemma}
\begin{proof}
 Writing $(P-\lambda)u=h^kf$ and $\tilde\lambda=\lambda+h^k\mu,\ \tilde u=u+h^kv$ we have
\begin{align*}
 (P-\lambdatilde)\utilde &= h^k [f - \mu u + (P-\lambda)v - h^k \mu v]
\end{align*}
This is in $h^{k+1}\calR(M)$ if and only if the term in brackets is in $h\calR(M)$, which by the initial step and model operator lemma is equivalent to
$[f - \mu u + (P-\lambda)v]_{AB}=0$ and then to
$$ (P-\lambda)_A v_A = -f_{AB} + 
\begin{pmatrix}
 0 \\ \mu\phi
\end{pmatrix}
$$
 where we used $(h^2 u)_{h=0}=0$ and $\Pi u_A=\Pi u_0=\phi$.
Now Lemma \ref{lem:adiab analysis} gives the result.
\end{proof}
\end{indlist}

We obtain the following theorem.
\begin{theorem}[quasimodes for adiabatic limit]
\label{thm:adiab quasimodes}
Suppose the operator $P$ in \eqref{eqn:P adiabatic} satisfies \eqref{eqn:adiab assn F} and \eqref{eqn:PB assn}, where $P_B$ is defined in \eqref{eqn:PB def}.
Given simple eigenvalues $\lambda_{-2}$, $\lambda_0$ of $P_F$, $P_B$ 
with eigenfunctions $\psi$, $\phi$ respectively, there are $\lambda\in h^{-2}\Cinf(\R_+)$, $u\in\CinfD(M)$ satisfying
$$ (P-\lambda)u \in h^\infty\Cinf (M)$$
and
$$ \lambda = h^{-2}\lambda_{-2}+\lambda_0+O(h),\quad u_A = \phi\otimes\psi \,.$$ 
Furthermore, $\lambda$ and $u$ are unique in Taylor series at $h=0$, up to replacing $u$ by $a(h)u$ where $a$ is smooth and $a(0)=1$.
\end{theorem}
\begin{proof}
This follows from the initial and inductive step as in the proof of Theorem \ref{thm:regular quasimodes}.
\end{proof}

\begin{remark}[Quasimodes vs. modes]
\label{rem:first vertical mode}
This construction works for any simple eigenvalues $\lambda_{-2},\lambda_0$ of $P_F,P_B$ respectively. However, when we ask whether a quasimode $(\lambda,u)$ is close (for small $h$) to an actual eigenvalue/eigenfunction pair we need to be careful:  while $\lambda$ will still be close to a true eigenvalue, $u$ may not be close to an eigenfunction unless $\lambda_{-2}$ is  the {\em smallest} eigenvalue of $P_F$ (\lq first vertical mode\rq). This is in contrast to the case of a regular perturbation where this problem does not arise.

The reason is that closeness of $u$ to an eigenfunction can only be proved (and in general is only true) if we have some a priori knowledge of a spectral gap, i.e. separation of eigenvalues. Such a separation is guaranteed for small $h$ only for the smallest $\lambda_{-2}$. For example, in the case of intervals $B=F=[0,\pi]$ we have eigenvalues $\lambda_{l,m}=h^{-2}l^2+m^2$, $k,l\in\N$. Then 
for each $m$ there are $h_i\to0$ and $m'_i\in\N$ so that $\lambda_{2,m}=\lambda_{1,m'_i}$ for each $i$. Then besides $u_{2,m}$ also $au_{2,m}+bu_{1,m'}$, $a,b\in\R$ are eigenfunctions for these eigenvalues, and in fact under small perturbations (i.e. if $P_1\neq0$) only the latter type may \lq survive\rq.

If one fixes $k$ and considers the $k$th eigenvalue $\lambda_k(h)$ of $\Omega_h$ then, for sufficiently small $h$, it will automatically correspond to the first vertical mode. This is clear for the rectangle but 
follows in general from the arguments that show that such a quasimode is close to an eigenfunction.
\end{remark}

\begin{remark}[Why fibres?]
\label{rem:adiab const fibres}
Why is it natural
to think of the subsets $F_x:=\{x\}\times F$ of $A=B\times F$ as \lq fibres\rq\ (and not the sets $B\times \{Y\}$, for example)? The reason is that these sets are inherently distinguished by the operator $P$: if $u$ is a smooth function on $M=A\times\R_+$ then $Pu$ is generally of order $h^{-2}$. But it is bounded as $h\to0$ if and only if $u$ and $\partial_h u$ are constant on each set $F_x$. 
Put invariantly, $P$ determines the fibres $F_x$ to second order at the boundary $h=0$.

In the geometric setup of the problem, which is sketched in Figure \ref{fig:adiab limit b-up}, the fibres arise naturally as fibres (i.e. preimages of points) of the blow-down map $\beta$ restricted to the front face.
\end{remark}
\begin{exercise}
Find a formula for the first non-trivial perturbation term $\lambda_1$. 
\end{exercise}

\subsection{Examples}\label{subsec:adiab examples}
We already looked at the trivial example of a rectangle. A non-trivial example will be given in Section \ref{sec:adiab variable}. 
Tubes  around curves provide another interesting example: Let $\gamma:I\to\R^2$ be a smooth simple curve in the plane parametrized by arc length, where $I\subset\R$ is a compact interval.
 The tube of width $h>0$ around $\gamma$ is 
$$ T_h= \{\gamma(x) + hYn(x)\,:\,x\in I,Y\in[-\tfrac12,\tfrac12]\}$$
where $n(x)$ is a unit normal at $\gamma(x)$.
For $h$ small the given parametrization is a diffeomorphism, and in coordinates $x,Y$ the euclidean metric on $T_h$ is $a^2dx^2+h^2dY^2$
where $a(x,Y) = 1-hY\kappa(x)$ with $\kappa$ the curvature of $\gamma$, so the Laplacian is
$\Delta=a^{-1}\partial_x a^{-1}\partial_x + h^{-2} a^{-1} \partial_Y a \partial_Y$, which is selfadjoint for the measure $adxdY$. This does not have the desired form. However, the operator $P=-a^{1/2}\Delta a^{-1/2}$ is
 unitarily equivalent to $-\Delta$ and self-adjoint in $L^2(I\times[-\tfrac12,\tfrac12],dxdY)$, and short calculation gives
 $$ P = -h^{-2}\partial_Y^2 - \partial_x^2 - \frac14\kappa^2 + O(h).$$
 Theorem \ref{thm:adiab quasimodes} now yields quasimodes where $\lambda_{-2}=\pi^2k^2$ and $\lambda_0$ is a Dirichlet eigenvalue of the operator $- \partial_x^2 - \frac14\kappa^2$ on $I$.
 See \cite{Gri:TTMPGASG} and \cite{FreKre:LNSTCT} for details.
\medskip

In all previous examples (and also in the example of Section \ref{sec:adiab variable}) the operators $P_F$ and $P_0$ commute. Here is a simple example where this is not the case. Take $B=F=[0,1]$, $P_F=-\partial_Y^2$ and $P_0=-\partial_x^2 + b(x,Y)$ for some smooth function $b$. 
Then $P_B = -\partial_x^2 + c(x)$ where 
 $c(x)=\langle b(x,Y)\psi(Y),\psi(Y)\rangle_{F}=\frac12\int_0^1 b(x,Y)\sin^2 \pi Y\,dY$ if $\lambda_{-2}=\pi^2$ is the lowest eigenvalue of $P_F$. Here 
$P_0$
commutes with $P_F$ iff $b=b(x)$, and then $c=b$. 

\subsection{Generalizations}\label{subsec:adiab-generaliz}
\subsubsection*{Fibre bundles}
The product $B\times F$ can be replaced by a fibre bundle $\pi:A\to B$ with base $B$ and fibres $F_x=\pi^{-1}(x)$. 
We assume $P$ is given as in \eqref{eqn:P adiabatic}, where $P_F$ differentiates only in the fibre directions. That is, for each $x\in B$ there is an operator $P_{F_x}$ on the fibre $F_x$. 
We assume that $P_{F_x}$ has {\em the same} eigenvalue $\lambda_{-2}$ for each $x\in B$, with one-dimensional eigenspace $K_x$.
Under this assumption there are no essential changes, mostly notational ones:
 
The $K_x$ form a line bundle $K$ over $B$. Sections of $K\to B$ may be identified with functions on $A$ which restricted to $F_x$ are in $K_x$, for each $x$, so 
$$ \Cinf(B,K) \subset \Cinf(A).$$
The line bundle $K\to B$ may not have a global non-vanishing section (replacing $\psi$). We deal with this by replacing functions on $B$ by sections of $K\to B$. The projections $\Cinf(F_x)\to K_x$ fit together to a map
$$ \Pi:\Cinf(M) \to \Cinf(B\times \R_+,K) $$
and then 
$$P_B = \Pi P_0 i:\CinfD(B,K)\to\Cinf(B,K)$$
where $i:\CinfD(B,K)\to\CinfD(A)$ is the inclusion. 
We replace $\Cinf(B)$ by $\Cinf(B,K)$ and $\phi\otimes\psi$ by $u_0\in \CinfD(B,K)\subset\CinfD(A)$, an eigensection of $P_B$, everywhere. Then the construction of formal eigenvalues and eigenfunctions works as before.

The adiabatic limit for fibre bundles has been considered frequently in the global analysis literature, see for example \cite{MazMel:ALHCLSSF}, \cite{DaiMel:ALHKAT}.
\subsubsection*{Multiplicities}
The construction can be generalized to the case where $\lambda_{-2}$ and $\lambda_0$ are multiple eigenvalues. In the case of fibre bundles it is important that the multiplicity of $\lambda_{-2}$ is independent of the base point, otherwise new analytic phenomena arise. 
\subsubsection*{Noncompact base}
\label{subsubsec:adiab noncompact}
The base (or fibre) need not be compact as long as $P_B$ (resp. $P_F$) has compact resolvent (hence discrete spectrum) and the higher order (in $h$) terms of $P$ behave well at infinity. 

For example, the case $B=\R$ with $P_B=-\partial_x^2+V(x)$ where $V(x)\to\infty$ as $|x|\to \infty$ arises in Section \ref{sec:adiab variable}.
%

\section{Adiabatic limit with variable fibre eigenvalue}
\label{sec:adiab variable}

\def\yoffsetoben{2}
\def\yoffsetunten{4}

\def\h{.3}

\pgfmathsetmacro{\radiusoben}{sqrt((\h/2+\yoffsetoben)^2+1)}
\pgfmathsetmacro{\radiusunten}{sqrt((\h/2+\yoffsetunten)^2+1)}

\pgfmathsetmacro{\startangletop}{acos(1/(\radiusoben))}
\pgfmathsetmacro{\endangletop}{acos(-1/(\radiusoben))}

\pgfmathsetmacro{\startanglebot}{acos(1/(\radiusunten))}
\pgfmathsetmacro{\endanglebot}{acos(-1/(\radiusunten))}

\begin{figure}
\centering
\begin{tikzpicture}

\begin{scope}[xscale=1.5,yscale=.8]

\draw[dashed,->] (-1.2,0,0) -- (1.2,0,0) node[right]{$x$};
\draw[dashed,->] (0,-1.5,0) -- (0,1.5,0) node[left]{$y$};

\draw (-1,\h/2) -- (-1,-\h/2) arc [start angle=\startanglebot+180, end angle=\endanglebot+180, radius=\radiusunten];;
\draw (1,-\h/2) -- (1, \h/2) arc [start angle=\startangletop, end angle=\endangletop, radius=\radiusoben];

\end{scope} 
\end{tikzpicture}

\caption{Thin domain $\Omega_h$ of variable thickness}
\label{fig:adiab var domain}
\end{figure}
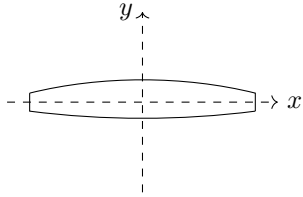
In this section we consider thin domains of variable thickness, see Figure \ref{fig:adiab var domain}. We will see that the nonconstancy of the thickness makes a big difference to the behavior of eigenfunctions and hence to the construction of quasimodes. However, using a suitable rescaling, reflected in the second blow-up in Figure \ref{fig:adiab var resolution}, we can reduce the problem to the case considered in the previous section. 

We consider a family of domains $\Omega_h\subset\R^2$ defined as follows.
Let $I\subset\R$ be a bounded open interval and $a_-,a_+:I\to\R$ be functions satisfying $a_-(x)<a_+(x)$ for all $x\in I$. Let
\begin{equation}
\label{eqn:adiab var Omega_h def} 
 \Omega_h = \{(x,y)\in\R^2:\, ha_-(x) < y < ha_+(x),\ x\in I\}
\end{equation}
for $h>0$.
We assume that the height function $a:=a_+-a_-$ has a unique, non-degenerate maximum, which we may assume to be at $0\in I$. More precisely
\begin{equation}
 \label{eqn:var adiab limit assumption}
\begin{gathered}
\text{\parbox{11cm}{
 for each $\eps>0$ there is a $\delta>0$ so that $|x|>\eps\Rightarrow a(x)<a(0)-\delta$, and\\[1mm]
$a$ is smooth near $0$ and  $a''(0) < 0$
}} 
\end{gathered}
\end{equation}
The conditions in the second line sharpen the first condition near $0$. 
See Section \ref{subsec:generalizations adiab limit var} for generalizations.

As before, we want to construct quasimodes $(\lambda_h,u_h)$ for the Laplacian on $\Omega_h$ with Dirichlet boundary conditions, as $h\to0$. Our construction will apply to \lq low\rq\ eigenvalues, see Remark \ref{rem:adiab var low ev} below.

As in the previous section we rescale the $y$-variable to lie in a fixed interval, independent of $x$:
Let
\begin{equation}
\label{eqn:adiab var Y} 
 Y = \frac{y-ha_-(x)}{ha(x)} \in (0,1).
\end{equation}
The change of variables $(x,y)\to (x,Y)$ transforms the vector fields
$\partial_x$, $\partial_y$ to%
\footnote{This is common but terrible notation. For calculational purposes it helps to write $(x',Y)$ for the new coordinates, related to $(x,y)$ via $x'=x$ and \eqref{eqn:adiab var Y}. Then $\frac{\partial}{\partial x}=\frac{\partial x'}{\partial x}\frac{\partial}{\partial x'} + \frac{\partial Y}{\partial x}\frac{\partial}{\partial Y} = \frac{\partial}{\partial x'} + b(x',Y)\frac{\partial}{\partial Y}$ and similarly for $\frac{\partial}{\partial y}$.
In the end replace $x'$ by $x$ to simplify notation.

Put differently, $\rightsquigarrow$ means push-forward under the map $F(x,y)=(x,Y(x,y))$.}
\begin{align*}
 \partial_x &\rightsquigarrow \partial_x + b(x,Y)\partial_Y,\quad
 b= \frac{\partial Y}{\partial x} = -\frac{a_-'}a - Y \frac{a'}a \\ 
 \partial_y &\rightsquigarrow \frac{\partial Y}{\partial y}\partial_Y = h^{-1}a^{-1}\partial_Y
\end{align*}
Therefore
$$ \Delta = h^{-2}a^{-2}\partial_Y^2 + (\partial_x + b\partial_Y)^2$$
This is reminiscent of the adiabatic limit considered in Section \ref{sec:adiab limit const}, but the fibre operator $a^{-2}\partial_Y^2$ has first eigenvalue $\pi^2 a(x)^{-2}$ depending on $x$, so the analysis developed there is not directly applicable.

We deal with this by expanding around $x=0$ and rescaling the $x$-variable.
\subsection{Heuristics: Finding the relevant scale}
The assumption $a''(0)<0$ implies that the Taylor series of $a^{-2}$ around $0$ is
\begin{equation}
\label{eqn:taylor a-2} 
 a^{-2}(x) \sim c_0 + c_2x^2+\dots, \quad c_0>0,\ c_2>0
\end{equation}
so 
\begin{equation}\label{eqn:Delta adiab limit var}
 \Delta =  c_0h^{-2}\partial_Y^2 + c_2 h^{-2} x^2 \partial_Y^2 + \dots + (\partial_x + b\partial_Y)^2
\end{equation}
near $x=0$. 

Which behavior do we expect for the eigenfunctions with small eigenvalues, say the first? 
Such an eigenfunction $u$ will minimize the Rayleigh-quotient
$$ R(u) = \frac{\langle -\Delta u,u\rangle}{\|u\|^2}$$
among functions satisfying Dirichlet boundary conditions.
 Let us see how the different terms in  \eqref{eqn:Delta adiab limit var} contribute to $R(u)$:
\begin{itemize}
 \item
The $h^{-2}\partial_Y^2$ term contributes at least $c_0\pi^2 h^{-2}$, since $\langle -\partial_Y^2 \psi,\psi\rangle_{[0,1]} \geq \pi^2\|\psi\|^2_{[0,1]}$ for any    $\psi:[0,1]\to\R$ having boundary values zero.\footnote{This is just the fact that the smallest eigenvalue of the Dirichlet Laplacian on $[0,1]$ is $\pi^2$.}
\item 
The $h^{-2}x^2\partial_Y^2$ term contributes a positive summand which is $O(h^{-2})$, but can be much smaller if the eigenfunction is large only for $x$ near zero. Specifically, if $u$ concentrates near $x=0$ on a scale of $L$, i.e.
$$ u(x,Y) \approx \phi(\frac xL) \psi(Y) $$
for a function $\phi$ on $\R$ that is rapidly decaying at infinity
then this term will be of order $$h^{-2}L^2$$
since $x^2 \phi(\frac xL) = L^2\tilde\phi(\frac xL)$ for $\tilde\phi(\xi)=\xi^2\phi(\xi)$ and $\tilde\phi$ is bounded\footnote{It is useful to think of this as follows: \lq $\phi(\frac xL)$ contributes only for $x\approx L$, and then $x^2\approx L^2$'.}. If  $L\to0$ for $h\to0$ then this is much smaller than $h^{-2}$.
\item
On the other hand, the $\partial_x^2$ term will be of order $L^{-2}$ if $u$ concentrates on a scale of $L$ near $x=0$.
\item
The other terms are smaller.
\end{itemize}
We can now determine the scale $L$ (as function of $h$) for which the sum of the $h^{-2}x^2\partial_Y^2$ and $\partial_x^2$ terms is smallest: For fixed $h$ the sum $h^{-2}L^2+L^{-2}$ is smallest when $h^{-2}L^2=L^{-2}$ (since the product of $h^{-2}L^2$ and $L^{-2}$ is constant), i.e.
$$ L = h^{1/2}.$$
The expectation of concentration justifies using the Taylor expansions around $x=0$. 

The heuristic considerations of this section are justified by the construction of quasimodes in the next section.

\subsection{Solution by reduction to the adiabatic limit with constant fibre}
The scaling considerations suggest to introduce the variable 
\begin{equation}
 \label{eqn:rescale x}
\xi=\frac x{h^{1/2}}
\end{equation}
 in \eqref{eqn:Delta adiab limit var}. 
Expanding also $b(x,Y)$ in Taylor series around $x=0$ and substituting $x=\xi h^{1/2}$ we obtain
\begin{equation}
 \label{eqn:Delta adiab limit rescaled}
 \Delta \sim h^{-2}c_0\partial_Y^2 + h^{-1}\left(\partial_\xi^2+ \xi^2 c_2\partial_Y^2\right) + \sum_{j=-1}^\infty h^{j/2} P_j
\end{equation}
where $P_j$ are second order differential operators in $\xi,Y$ 
whose coefficients are polynomial in $\xi$ (of degree at most $j+4$) and linear in $Y$.

The right hand side of \eqref{eqn:Delta adiab limit rescaled} is a formal series of
differential operators which are defined for $Y\in (0,1)$ and $\xi\in\R$. Now we may apply the constructions of Section \ref{sec:adiab limit const}, with $F=[0,1]$ and $B=\R$.
More precisely, $-\Delta=h^{-1}P$ where, with $t=h^{1/2}$,
$$P \sim t^{{-2}}P_F + P_0+tP_1+\dots$$
with $P_F=-c_0\partial_Y^2$ and $P_0=-\partial_\xi^2 - \xi^2 c_2\partial_Y^2$. These operators act on bounded functions satisfying Dirichlet boundary conditions at $Y=0$ and $Y=1$.
Using  the first eigenvalue, $\lambda_{-2}=c_0\pi^2$, of $P_F$ we get the horizontal operator (see \eqref{eqn:PB def})
$$ P_B = -\partial_\xi^2 + \omega^2 \xi^2,\quad \omega=\sqrt{c_0c_2}\pi.$$
This is the well-known quantum harmonic oscillator, with eigenvalues $\mu_m=\omega(2m+1)$, $m=0,1,2,\dots$, and eigenfunctions
\begin{equation}
 \label{eqn:adiab var harm osc}
\psi_m(\xi) = H_m(\sqrt\omega \xi) e^{-\frac12 \omega\xi^2} 
\end{equation}
where $H_m$ is the $m$th Hermite polynomial.

The exponential decay of $\psi_m$ as $|\xi|\to\infty$ justifies a posteriori the scaling limit considerations above. It means that quasimodes concentrate on a strip around $x=0$ whose width is of order $h^{1/2}$ . 

By Theorem \ref{thm:adiab quasimodes} in Section \ref{sec:adiab limit const} the operator $P$ has quasimodes
$\sin \pi Y \psi_m(\xi) + O(t)$. To get quasimodes for $-\Delta$ on $\Omega_h$ we simply substitute the coordinates 
$Y,\xi$ as in \eqref{eqn:adiab var Y}, \eqref{eqn:rescale x}. In addition, we should introduce a cutoff near the ends of the interval $I$ so that Dirichlet boundary conditions are satisfied there.

We state the result in terms of resolutions.
Introducing the singular coordinates $Y$ and $\xi$ corresponds to 
a resolution of the total space $\Omega=\bigcup_{h>0}\Omega_h\times\{h\}$ by
two blow-ups as shown in Figure \ref{fig:adiab var resolution}: 
$$ \Omegabar \longleftarrow [\Omegabar,\{y=h=0\}] =: M_0 \longleftarrow
 [M_0,\{x=h=0\}]_q =: M\,.$$

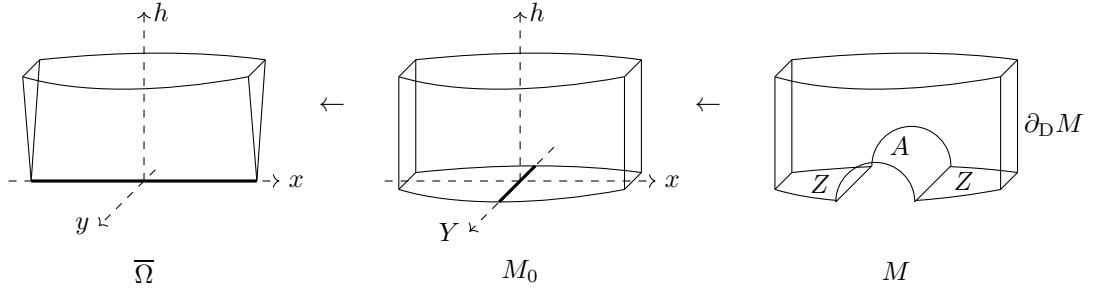
\begin{figure}
 
\begin{tikzpicture}

\begin{scope}[scale=1.5, x={(1cm,0)},y={(-.5cm,-.5cm)}, z={(0,1cm)}]

 \draw[dashed,->] (-1.2,0,0) -- (1.2,0,0) node[right]{$x$};
 \draw[dashed,->] (0,-.2,0) -- (0,.8,0) node[left]{$y$};
 \draw[dashed,->] (0,0,0) -- (0,0,1.5) node[right]{$h$};

\draw[very thick] (-1,0,0) -- (1,0,0); 

\def\heights{1}

 \draw (-1,\h/2, \heights) -- (-1,-\h/2, \heights) arc [start angle=\startanglebot+180, end angle=\endanglebot+180, radius=\radiusunten];;
\draw (1,-\h/2, \heights) -- (1, \h/2, \heights) arc [start angle=\startangletop, end angle=\endangletop, radius=\radiusoben]; 

\draw (-1,0,0) -- (-1,\h/2,\heights);
\draw (-1,0,0) -- (-1,-\h/2,\heights);
\draw (1,0,0) -- (1,\h/2,\heights);
\draw (1,0,0) -- (1,-\h/2,\heights);

\node at (0,0,-.8){$\Omegabar$}; 
 \end{scope}
   \node at (2.5,1) {$\leftarrow$};

 \begin{scope}[xshift=5cm, scale=1.5,  x={(1cm,0)},y={(-.5cm,-.5cm)}, z={(0,1cm)}]
 
 \draw[dashed,->] (-1.2,0,0) -- (1.2,0,0) node[right]{$x$};
 \draw[dashed,->] (0,-.6,0) -- (0,.9,0) node[left]{$Y$};
 \draw[dashed,->] (0,0,0) -- (0,0,1.5) node[right]{$h$}; 
 
 \draw (-1,\h/2, 0) -- (-1,-\h/2, 0) arc [start angle=\startanglebot+180, end angle=\endanglebot+180, radius=\radiusunten];;
\draw (1,-\h/2, 0) -- (1, \h/2, 0) arc [start angle=\startangletop, end angle=\endangletop, radius=\radiusoben]; 

\def\height{1}

 \draw (-1,\h/2, \height) -- (-1,-\h/2, \height) arc [start angle=\startanglebot+180, end angle=\endanglebot+180, radius=\radiusunten];;
\draw (1,-\h/2, \height) -- (1, \h/2, \height) arc [start angle=\startangletop, end angle=\endangletop, radius=\radiusoben]; 

\draw (-1,\h/2, 0) -- (-1,\h/2, \height);
\draw (-1,-\h/2, 0) -- (-1,-\h/2, \height);
\draw (1,-\h/2, 0) -- (1, -\h/2, \height);
\draw (1,\h/2, 0) -- (1, \h/2, \height);

\draw[very thick] (0,-\yoffsetoben+\radiusoben,0) -- (0,\yoffsetunten-\radiusunten,0);

\node at (0,0,-.8){$M_0$}; 

 \end{scope}
 
   \node at (7.5,1) {$\leftarrow$};

 \begin{scope}[xshift=10cm, scale=1.5,  x={(1cm,0)},y={(-.5cm,-.5cm)}, z={(0,1cm)}]
 

\def\height{1}
\def\radiuscylinder{0.35}

\pgfmathsetmacro{\angleuntenlinks}{acos((\radiuscylinder/\radiusunten)}
\pgfmathsetmacro{\angleuntenrechts}{acos((-\radiuscylinder/\radiusunten)}

\pgfmathsetmacro{\angleobenlinks}{acos((\radiuscylinder/\radiusoben)}
\pgfmathsetmacro{\angleobenrechts}{acos((-\radiuscylinder/\radiusoben)}

 \draw (-1,\h/2, 0) -- (-1,-\h/2, 0) arc [start angle=\startanglebot+180, end angle=\angleuntenlinks+180, radius=\radiusunten];
 \draw (1,-\h/2, 0) arc [start angle=\endanglebot+180, end angle=\angleuntenrechts+180, radius=\radiusunten];
 
 \draw (1,-\h/2, 0) -- (1,\h/2, 0) arc [start angle=\startangletop, end angle=\angleobenlinks, radius=\radiusoben];
 \draw (-1,\h/2, 0) arc [start angle=\endangletop, end angle=\angleobenrechts, radius=\radiusoben];

\pgfmathsetmacro{\endYPLUS}{-\yoffsetoben+\radiusoben}
\pgfmathsetmacro{\endYMINUS}{\yoffsetunten-\radiusunten}

 \draw (-1,\h/2, \height) -- (-1,-\h/2, \height) arc [start angle=\startanglebot+180, end angle=\endanglebot+180, radius=\radiusunten];;
\draw (1,-\h/2, \height) -- (1, \h/2, \height) arc [start angle=\startangletop, end angle=\endangletop, radius=\radiusoben]; 

\draw (-1,\h/2, 0) -- (-1,\h/2, \height);
\draw (-1,-\h/2, 0) -- (-1,-\h/2, \height);
\draw (1,-\h/2, 0) -- (1, -\h/2, \height);
\draw (1,\h/2, 0) -- (1, \h/2, \height);

\draw (\radiuscylinder,\endYPLUS, 0) -- (\radiuscylinder,\endYMINUS, 0);
\draw (-\radiuscylinder,\endYPLUS, 0) -- (-\radiuscylinder,\endYMINUS, 0);

\draw [canvas is xz plane at y=\endYMINUS] (\radiuscylinder, 0, 0) arc [start angle = 0, end angle=180, radius=\radiuscylinder];
\draw [canvas is xz plane at y=\endYPLUS] (\radiuscylinder, 0, 0) arc [start angle = 0, end angle=180, radius=\radiuscylinder];

\node at (0.07,.08,\radiuscylinder) {$A$};
\node at (.6,0,0) {$Z$};
\node at (-.65,.05,0) {$Z$};
\node at (1.4,0,\height/2) {$\dDir M$};
\node at (0,0,-.8){$M$}; 

 \end{scope}
 
\end{tikzpicture}
 
 \caption{Total space for adiabatic limit with variable thickness and its resolution}
 \label{fig:adiab var resolution}
\end{figure}

The  blow-up of $\Omegabar$ in the $x$ axis corresponds to introducing $Y$, as in Section \ref{sec:adiab limit const}, and results in the space $M_0$.
The quasihomogeneous blow-up (see Subsection \ref{subsubsec:blow-up qh}) of $M_0$ in the $Y$-axis corresponds to introducing $\xi=\frac x{\sqrt h}$. Compare Figure \ref{fig:proj coord qh} (with $y$ replaced by $h$):  $\xi$ and $\sqrt h=t$, the variables used for the operator $P$, are precisely the \lq top\rq\ projective coordinates defined away from the right face. Denote the total blow-down map by
$$ \beta: M\to \Omegabar\,.$$
Each of the two blow-ups creates a boundary hypersurface of $M$ at $h=0$: the first blow-up creates $Z$, the  second blow-up creates $A$ (for \lq adiabatic'). In addition, $M$ has the Dirichlet boundary $\partial_DM$ which is the lift of $\bigcup_{h>0}(\partial\Omega_h)\times\{h\}\subset\Omegabar$.

The essence of these blow-ups is that we can construct quasimodes
as smooth functions on $M$.\footnote{Of course this means that we construct quasimodes on $\Omega$ so that their pull-backs to $M$ extend smoothly to the boundary of $M$.} Their expansion at $A$ is the one obtained using the analysis of $P$. Since the quasimodes of $P$ are exponentially decaying as $\xi\to\pm\infty$, we may just take the zero expansion at $Z$ (hence the letter $Z$).

Summarizing, we obtain the following theorem.
We denote
$$\Cinf_{1/2}(\R_+) =\{\mu:\R_+\to\R:
\mu(h) = \tilde\mu(\sqrt h) \text{ for some }\tilde\mu\in\Cinf(\R_+)
\}\,.$$
\begin{theorem}[quasimodes for adiabatic limit with variable fibre eigenvalue]
\label{thm:adiabatic variable}
Consider the family of domains $\Omega_h$ defined in \eqref{eqn:adiab var Omega_h def} and  satisfying \eqref{eqn:var adiab limit assumption}.
Define $M$ as above.
Then for each $m\in\N$ there are $\lambda_m\in h^{-2}\Cinf_{1/2}(\R_+)$, $u_m\in\CinfD(M)$
satisfying 
$$ (-\Delta-\lambda_m)u_m \in h^\infty\Cinf (M)$$
and
\begin{align*}
\displaystyle \lambda_m &\sim c_0\pi^2 h^{-2} + \sqrt{c_0c_2}\pi(2m+1) h^{-1} + O(h^{-1/2}) \\
u_{m} &= \sin \pi Y \psi_m(\xi) \text{ at }A,\quad u_m = 0 \text{ at }Z
\end{align*}
 where $c_0=a(0)^{-2}$, $c_2 = -a''(0) a(0)^{-1}$.
 In addition, $u_m$ vanishes to infinite order at $Z$.
\end{theorem}
In the original coordinates on $\Omega_h$ the conditions on $u_m$ translate to
\begin{equation}
 \label{eqn:adiab variable explicit asymp}
u_m(h,x,y) = \sin\pi \frac {y-ha_-(x)}{ha(x)}\,\, \psi_m (\frac x{h^{1/2}})\, +\, O (h^{1/2} \left(1+\tfrac {x^2}h\right)^{-N})
\end{equation}
for all $N$. There is also a uniqueness statement similar to the one in Theorem \ref{thm:adiab quasimodes}.
\begin{proof}
 Choose a function $u_m$ on $M$ satisfying the following conditions: The expansion of $u_m$ at the face $A$ is given by the expansion for the quasimodes of $P$ discussed above. The expansion of $u_m$ at the face $Z$ is identically zero; and $u_m$ is zero at the Dirichlet boundary of $M$. Since $\psi_m$ is exponentially decaying and all $P_j$ have coefficients which are polynomial in $\xi$, all terms in the expansion at $A$ are exponentially decaying as $\xi\to\infty$. Since $\xi=\infty$ corresponds to the corner $A\cap Z$, the matching conditions of the Borel Lemma \ref{lem:borel} are satisfied, so $u_m\in\CinfD(M)$ exists having the given expansions.  Since both expansions satisfy the eigenvalue equation to infinite order, so does $u_m$. 
The extra decay factor in the error term of $u_m$ in \eqref{eqn:adiab variable explicit asymp} corresponds to the infinite order vanishing at $Z$, since $\frac {x^2}h$ defines $Z$ near $A\cap Z$, see Figure \ref{fig:proj coord qh}.
\end{proof}
\begin{remark}
\label{rem:adiab var low ev}
The scaling considerations depended on the assumption that $u$ concentrates near $x=0$ as $h\to 0$, and this was justified a posteriori
by Theorem \ref{thm:adiabatic variable}.  On the other hand, it can also be shown a priori  using Agmon estimates 
that eigenfunctions for eigenvalues $\lambda_k(h)$, where $k$ is fixed as $h\to0$, behave in this way (and this can be used to prove closeness of quasimodes to eigenfunctions, see \cite{Ag:BEDESO}, \cite{DauRay:PWCSAL} for example).

Quasimodes can also be constructed for higher vertical modes, i.e. taking $\lambda_{-2}=l^2c_0\pi^2$ for any $l\in\N$. However, the same caveat as in Remark \ref{rem:first vertical mode} applies. 
\end{remark}
\begin{exercise}
Compute the next term in the expansion of $\lambda_m$, i.e. the coefficient of $h^{-1/2}$.
\end{exercise}
\subsection{Generalizations}
\label{subsec:generalizations adiab limit var}
\subsubsection*{Degenerate maximum}
A very similar procedure works if $a$ has a finitely degenerate maximum, i.e. if the condition $a''(0)<0$ in \eqref{eqn:var adiab limit assumption} is replaced by
\begin{equation}
\label{eqn:adiab var deg max} 
 a^{(j)}(0)=0\ \text{ for } j < 2p,\quad a^{(2p)}(0)< 0
\end{equation}
for some $p\in\N$.
The order is even by smoothness. The expansion \eqref{eqn:taylor a-2} is replaced by $a^{-2}(x)\sim c_0 + c_{2p} x^{2p} + \dots$ with $c_{2p}>0$, and then the correct scaling is found from the equation $h^{-2}L^{2p} = L^{-2}$, so $L=h^{\frac1{p+1}}$. So we set $\xi=\frac xt$ where $t=h^{\frac1{p+1}}$, then $-\Delta = t^{-2}P$ where
$$ P = t^{-2p}(-c_0\partial_Y^2) + (-\partial_\xi^2 - c_{2p}\xi^{2p}\partial_Y^2) + tP_1+\dots $$
The adiabatic limit analysis works just as well with $t^{-2p}$ as with $t^{-2}$ in the leading term (do it!), and the eigenfunctions of the operator $-\partial_\xi^2 + \omega^2\xi^{2p}$ are still rapidly decaying at infinity, so we obtain
$$ \lambda_m \sim c_0\pi^2h^{-2} + \sum_{j=-2}^\infty d_{jm}h^{\frac j{p+1}}$$
and a similar statement for $u_m$.

This problem with weaker regularity assumptions (and also allowing half-integer $p$ in \eqref{eqn:adiab var deg max}) was analyzed in \cite{FriSol:SDLNS}, by a different method.

\subsubsection*{Several maxima} 
If the height function $a$ has several isolated maxima then each one will contribute quasimodes. For instance, consider the case of two maxima at $x=x_1$ and $x=x_2$, with $a_i=a(x_i)$, and  let  $\lambda_k(h)$ be the $k$th eigenvalue of $\Omega_h$
for fixed $k\in\N$. 
If $a_1>a_2$ then the leading term of the quasi-eigenvalue constructed at $a_1$ is smaller than the one at $a_2$, and therefore the aymptotics of $\lambda_k(h)$ as $h\to0$ is determined from the Taylor series of $a$ around $x_1$, and the eigenfunction concentrates near $x_1$ alone. 
On the other hand, if $a_1=a_2$ then both maxima will generally contribute, and it is interesting to analyze their interaction (so-called tunnelling). A special case of this was analyzed in \cite{Our:DEAFT}, and a detailed study of tunnelling for Schr\"odinger operators with potentials was carried out in \cite{Hel:SCASOA} and \cite{HelSjo:MWSCLI}. 

\subsubsection*{Other approaches}
A different, more operator-theoretic approach to the problem considered here (and more general ones, e.g. higher dimensions) is taken in \cite{LamTeu:EHTDTVCS}, \cite{HaaLamTeu:GQW}, \cite{LamTeu:ALSOFB}, see also the book \cite{Teu:APTQD}.
\section{Adiabatic limit with ends}
\label{sec:adiabatic with ends}
%
%

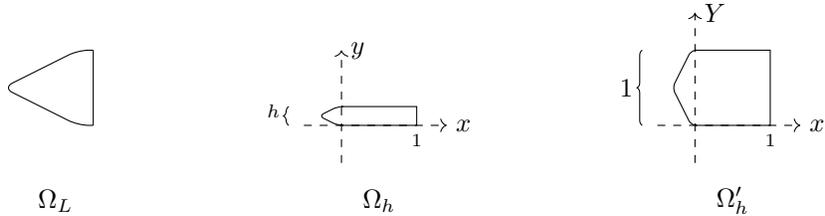
\begin{figure}
\centering
\begin{tikzpicture}
\begin{scope}
\draw (0,0) -- (0,1);
\draw[rounded corners=4pt] (0,1) -- (-.2,1) --  (-1.2,0.5) -- (-.2,0) -- (0,0);
\node at (-.5,-1) {$\Omega_L$};
\end{scope}
\begin{scope}[xshift=3.3cm]
 \draw  (0,0) -- (1,0) node[below]{$\scriptstyle 1$} -- (1,.25) -- (0,.25);
 \draw[rounded corners=2pt] (0,.25) -- (-.05,.25) -- (-.3,.125) -- (-.05,0) -- (0,0);
 
 \draw[dashed,->] (-.5,0) -- (1.4,0) node[right] {$x$};
 \draw[dashed,->] (0,-.5) -- (0,1) node[right] {$y$};
  \draw[decorate,decoration={brace,amplitude=2pt}]
(-.7,0) -- (-.7,.25) node [near end,left]{$\scriptstyle h$};

 \node at (.5,-1) {$\Omega_h$};
\end{scope}
\begin{scope}[xshift=8cm]
 \draw  (0,0) -- (1,0) node[below]{$\scriptstyle 1$} -- (1,1) -- (0,1);
 \draw[rounded corners=2pt] (0,1) -- (-.05,1) -- (-.3,.5) -- (-.05,0) -- (0,0);
 
 \draw[dashed,->] (-.5,0) -- (1.4,0) node[right] {$x$};
 \draw[dashed,->] (0,-.5) -- (0,1.5) node[right] {$Y$};
  \draw[decorate,decoration={brace,amplitude=2pt}]
(-.7,0) -- (-.7,1) node [midway,left]{$1$};

 \node at (.5,-1) {$\Omega_h'$};
\end{scope}

\end{tikzpicture}
\caption{Example of domains $\Omega_L$ and $\Omega_h$, and rescaling after first blow-up}
\label{fig:example OmegaL}
\end{figure}

We consider the following problem, see Figure \ref{fig:example OmegaL} left and center: Let $\Omega_L\subset\R^2$ be a bounded domain contained in the left half plane $x<0$, having $\{0\}\times[0,1]$ as part of its boundary. For $h>0$ consider the domain
\begin{equation}
 \label{eqn:adiab ends Omegah def}
 \Omega_h = h\Omega_L \cup R_h \subset \R^2,\quad R_h = [0,1) \times (0,h)
\end{equation}
i.e. a $1\times h$ rectangle with the \lq end' $\Omega_L$, scaled down by the factor $h$, attached at its left boundary. 
To simplify notation we assume that $\Omega_L$ is such that the boundary of $\Omega_h$ is smooth, except for the right angles at the right end; however, 
this is irrelevant for the method.

We denote coordinates on $\Omega_h$ by $x,y$.
We will construct quasimodes $(\lambda_h,u_h)$ for the Laplacian $\Delta_h = \partial_x^2+\partial_y^2$ on $\Omega_h$, with Dirichlet boundary conditions, as $h\to 0$.

The central difficulty, and new aspect compared to the adiabatic limit, is the fact that there are two different scalings in the problem:
\begin{itemize}
\item in the rectangular part of $\Omega_h$ only the $y$-direction scales like $h$,
\item  in the left end both $x$- and $y$-directions scale like $h$. 
\end{itemize}
This leads to different ways in which these two parts of $\Omega_h$ influence eigenvalues and eigenfunctions. 

This is a simple case of a much more general setup arising in contexts such as surgery in global analysis and \lq fat graph' analysis, see Section \ref{subsec:general ends added}. The essential structures, however, already appear in this simple case. An explicit analysis using matched asymptotic expansions was carried out in \cite{GriJer:AEPD}. We will rederive the quasimode expansions in a more conceptual way using the idea of resolutions. 

\subsection{Resolution}
\label{subsec:adiab ends resol}
First, we construct a space on which we may hope the eigenfunctions (and quasimodes) to be smooth. We start with the total space on which these are functions, which is
$$\Omega=\bigcup_{h>0} \Omega_h\times \{h\} \subset\R^3$$
see the left picture in Figure \ref{fig:blow-up}.
Really we want to consider the closure
$\Omegabar$
since we are interested in the behavior of quasimodes as $h\to 0$, compare Remark \ref{rem:functions int}. This set is not a manifold with corners, let alone a d-submanifold of $\R^2\times\R_+$ (compare Footnote \ref{footnote:not dsubman}).
At $y=h=0$ the set $\Omegabar$ has an adiabatic limit type
singularity as in the case of Section \ref{sec:adiab limit const}. In addition, it has a conical singularity (with singular base) at the point $x=y=h=0$.

So we blow up these two submanifolds of $\R^3$ and find the lift (see Definition  \ref{def:resol subsets}) of $\Omegabar$: The blow up of $\{y=h=0\}$ results in the space $M_0$ in the center of Figure 
\ref{fig:blow-up}.  Projective coordinates are $x$, $Y=\frac yh$ and $h$, globally on $M_0$ since $|y|\leq C h$ on $\Omegabar$.\footnote{To make sense of the picture for $M_0$ it may help to note that $M_0$ is  the closure of $\bigcup_{h>0} \Omega_h'\times\{h\}$ where $\Omega_h'=\{(x,Y):\, (x,y)\in\Omega_h, Y=\frac yh\}$ is depicted on the right in Figure \ref{fig:example OmegaL}.} 
The bottom face of $M_0$ is $h=0$, and the preimage of the point $x=y=h=0$ is the bold face line $x=h=0$ in $M_0$. So we blow up this line and define $M$ to be the lift of $M_0$.\footnote{%
You may wonder if we would have obtained a different space if had first blown up the point $x=y=h=0$ and then the (lift of the) line $y=h=0$. It can easily be checked that this results in the same space $M$ -- more precisely that the identity on the interiors of this space and of $M$ extends to the boundary as a diffeomorphism. This also follows from the fact that $\{x=y=h=0\}\subset\{y=h=0\}$ and a general theorem about commuting blow-ups, see \cite{Mel:DAMWC}.} 

As always we will use $x,y,h$ to denote the pull-backs of the coordinate functions $x,y,h$ on $\Omegabar$ to $M$. Projective coordinate systems for the second blow-up give coordinates $h, X=\frac xh, Y$ on $M\setminus A$ and $\frac hx, x, Y$ in a neighborhood of  $A$.
%
\begin{figure}
\centering
\begin{tikzpicture}[x=1.8cm,y={(-1cm,-1cm)}, z={(0cm,2cm)}]
\begin{scope}
  \draw (0,0,1) -- (1,0,1)  -- (1,.5,1) -- (0,.5,1);
  \draw[rounded corners=2pt] (0,.5,1) -- (-.1,.5,1) -- (-.6,.25,1) -- (-.1,0,1) -- (0,0,1);
 \draw[dotted] (0,0,1) -- (0,.5,1);
 \draw[very thick] (0,0,0) -- (1,0,0); 
 \draw[dashed] (0,0,0) -- (0,0.5,1);
 \draw[dashed] (0,0,0) -- (-.58,.25,1);
 \draw (1,0,0) -- (1,0,1);
 \draw (1,0,0) -- (1,.5,1);

 \draw (0.2,-0.5,1) node {$\Omegabar$};
 
 \draw[dashed,->] (0,0,0) -- (1.2,0,0) node[below]{$x$};
 \draw[dashed,->] (0,0,0) -- (0,0,1.2) node[left]{$h$};
 \draw[dashed,->] (0,0,0) -- (0,.5,0) node[left]{$y$};
 
\end{scope}

 \draw (1.4,0,0.5) node {$\leftarrow$};

 
\begin{scope}[xshift=4.5cm]

  \draw (0,0,1) -- (1,0,1)  -- (1,.5,1) -- (0,.5,1);
  \draw[rounded corners=2pt] (0,.5,1) -- (-.1,.5,1) -- (-.6,.25,1) -- (-.1,0,1) -- (0,0,1);
 \draw[dotted] (0,0,1) -- (0,.5,1);
 \draw (0,0,0) -- (1,0,0)  -- (1,.5,0) -- (0,.5,0) -- (0,0,0);
 \draw[very thick] (0,0,0) -- (0,0.5,0); 
 \draw[dashed] (0,0.5,0) -- (0,0.5,1);
 \draw[dashed] (0,0.25,0) -- (-.58,.25,1);
 \draw (1,0,0) -- (1,0,1);
 \draw (1,0.5,0) -- (1,.5,1);
  \draw (0.2,-0.5,1) node {$M_0$};

 \draw[dashed,->] (0,0,0) -- (1.2,0,0) node[below]{$x$};
 \draw[dashed,->] (0,0,0) -- (0,0,1.2) node[left]{$h$};
 \draw[dashed,->] (0,0,0) -- (0,.8,0) node[left]{$Y$};

\end{scope}

 \draw (4,0,.5) node {$\leftarrow$};

\begin{scope}[xshift=9.5cm]
  \draw (0,0,1) -- (1,0,1)  -- (1,.5,1) -- (0,.5,1);
  \draw[rounded corners=2pt] (0,.5,1) -- (-.1,.5,1) -- (-.6,.25,1) -- (-.1,0,1) -- (0,0,1);  
 \draw[dotted] (0,0,1) -- (0,.5,1);

  \draw (0.3,0,0) -- (1,0,0)  -- (1,.5,0) -- (0.3,.5,0) -- (0.3,0,0);
 \draw[y={(0cm,2cm)}] (.3,0) arc (0:90:.3);
 \draw[y={(0cm,2cm)}] (0.02,-.25) arc (0:90:.3);

 \draw[rounded corners=2pt] (0,.5,.3) -- (-.1,.5,.3) -- (-.6,.25,.3) -- (-.1,0,.3) -- (0,0,.3);
 \draw[dotted] (0,0,.3) -- (0,.5,.3);
 \draw[dashed] (0,0,0.3) -- (0,0,1);
 \draw[dashed] (0,0.5,0.3) -- (0,0.5,1);
 \draw[dashed] (-.58,0.25,0.3) -- (-.58,.25,1);
 \draw (1,0,0) -- (1,0,1);
 \draw (1,0.5,0) -- (1,.5,1);
 \draw (.6,.25,0) node{$A$};
 \draw (0.2,0.3,0.25) node {$S$};
 
 \draw (1.25,0,.5) node{$\dDir M$};
 \draw (0.2,-0.5,1) node {$M$};
 \draw (.3,1.5,0) node[left]{$B$}-- (1,1.5,0);
 \draw[->] (.65,1,0)  -- (.65,1.2,0) node[right,xshift=.1cm]{$\pi_a$};
 \foreach \x in {.4,.5,.6,.7,.8,.9}
 {
 \draw[dotted] (\x,0,0) -- (\x,.5,0);
 }
\end{scope}

\end{tikzpicture}

 \caption{Total space and its resolution for adiabatic limit with ends, with fibration of the adiabatic face $A$; solid lines are codimension 2 corners of $M$, dashed or dotted lines are not}
\label{fig:blow-up}
\end{figure}
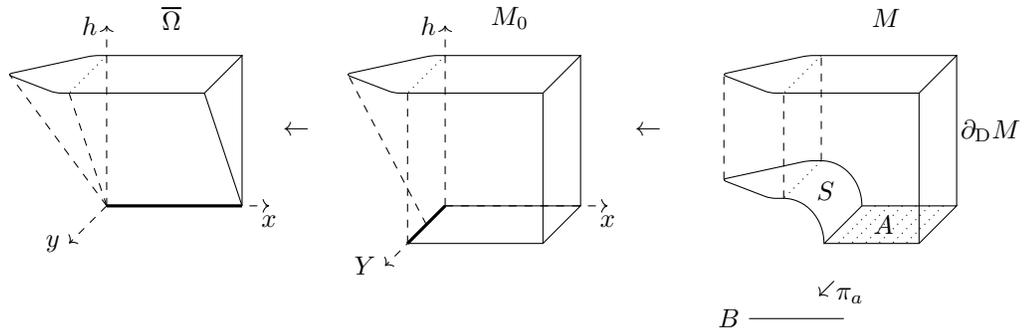
\begin{figure}
\centering
\begin{tikzpicture}[x=1.5cm,y=1.5cm]
   \draw (0,0) -- (1,0)  -- (1,1) -- (0,1);
   \draw[rounded corners=4pt] (0,0) -- (-.2,0) -- (-1,.5) -- (-.2,1) -- (0,1); 
   \draw (1,0) -- (2,0) -- (2,1) -- (1,1); 
   \draw (.2,.5) node{$S$};
   \draw (1.5,.5) node{$A$};
   \draw[->] (-2,0) node[left]{$0$}-- (-2,1) node[left]{$1$ }node[right]{$Y$};
   \draw (-2.05,0) -- (-1.95,0); 
   \draw[->] (-1,-.5) node[above right]{$X$} -- (1,-.5) node[below left]{$\infty$};
   \draw (0,-.5) -- (0,-.54) node[below]{$0$};
   \draw[->] (1,-.5) node[above right]{$x$} node[below right]{$0$}-- (2,-.5) node[below]{$1$};
   \draw[dashed] (1,0) -- (1,-.7);
\end{tikzpicture}
 \caption{\lq Flattened' picture of $h=0$ boundary of resolved total space $M$, with coordinates for each face}
 \label{fig:flat M}
\end{figure}
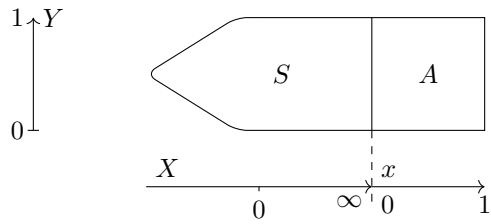
The space $M$ has two types of boundary hypersurfaces:
\begin{itemize}
 \item The \lq Dirichlet boundary' $\dDir M$, which corresponds to the boundary of $\Omega_h$. This is the union of the two \lq vertical' faces in the right picture of Figure \ref{fig:blow-up}:
 $$ \dDir M = \overline{\beta^{-1}\left(\bigcup_{h>0}\partial\Omega_h \times\{h\}\right)}$$
where $\beta:M\to \Omegabar$ is the total blow-down map.
 \item The boundary at $h=0$,
 $$ \partial_0 M = S \cup A $$
where $A$ and $S$ are the front faces of the two blow-ups, which meet in the corner $S\cap A$.%
 \footnote{$A$ is for adiabatic and $S$ is for surgery, see Section \ref{subsec:general ends added} for an explanation.} 
\end{itemize}
Our interest lies in the behavior of quasimodes at $A$ and $S$. All functions will be smooth at the Dirichlet boundary. 

The faces $A$ and $S$ are rescaled limits of $\Omega_h$, see the discussion at the end of Section \ref{subsec:resol-ex}. The adiabatic face $A$ is naturally a rectangle
$$ A \equiv [0,1] \times [0,1] \text{ with coordinates $x$ and $Y=\frac yh$.}$$
 It is the limit as $h\to0$ of $\{(x,\frac yh):\, (x,y)\in 
 \overline R_h\}$
 -- this is precisely what the blow-up means, in terms of projective coordinates.
The Laplacian in these coordinates is
$$ \Delta = h^{-2}\partial_Y^2 + \partial_x^2 .$$
Thus, we have an adiabatic problem, 
with base $B=[0,1]_x$ and fibre $F=[0,1]_Y$ and $P_F=-\partial_Y^2$, $P_B = -\partial_x^2$. The corresponding
projection is
\begin{equation}
\label{eqn:A face projection} 
 \pi_A: U_A\to U_B,\quad (x,Y,h) \mapsto (x,h)
\end{equation}
where $U_B=B\times[0,\eps)$ for some $\eps>0$ and $U_A$ is a neighborhood of $A$.
There is a difference to the setup in Section \ref{sec:adiab limit const} in that $h$ is not a defining function for $A$. This leads to various issues below.

The interior of the surgery face $S$ can be identified with the plane domain $\Omega^{\infty}$ obtained by taking $h^{-1}\Omega_h$ and letting $h\to 0$ (again, by definition of the blowup):
\begin{equation}
\label{eqn:surgery face} 
 \interior S \equiv \Omega^\infty:=\Omega_L \cup \left( [0,\infty) \times (0,1) \right) \text{ with coordinates $X=\frac xh$ and $Y=\frac yh$,}
\end{equation}
 and the Laplacian is
$$ \Delta = h^{-2}(\partial_X^2+\partial_Y^2) .$$
The corner $S\cap A$ is the interval $[0,1]$ and corresponds to $x=0$ in $A$ and to $X=\infty$ in $S$.
Coordinates near the corner are $x$, defining $S$ locally, and $\frac{h}x = X^{-1}$, defining $A$ locally and even globally.

Note that the face $A$ carries naturally a non-trivial fibration, compare Remark \ref{rem:adiab const fibres}, but the face $S$ does not: locally near any point of $S$ no direction is distinguished.

\subsection{Solution}
The construction of quasimodes builds on the construction for the adiabatic limit in Section \ref{sec:adiab limit const}. The presence of the extra scale, i.e. the left end of $\Omega_h$, leads to a number of new features.

To emphasize the relation with previous sections and motivated by the considerations above we will use the notation
\begin{equation}
 \label{eqn:adiab ends notation operators}
\begin{gathered}
P = -\Delta = -\partial_x^2 - \partial_y^2 \\
P_F = -\partial_Y^2,\ P_B = -\partial_x^2 ,\ P_S = -\partial_X^2-\partial_Y^2
\end{gathered}
\end{equation}

\subsubsection{A priori step: Fixing the vertical mode.}
Since an adiabatic limit is involved, we fix a priori
\begin{align*}
 \lambda_{-2} &= \text{a simple eigenvalue  of $P_F$ on $[0,1]$, with Dirichlet boundary conditions}\\
 \psi &= \text{an $L^2$-normalized corresponding eigenfunction}.
\end{align*}
Here we take the lowest fibre eigenvalue%
\footnote{One could also consider higher fibre modes, but this would change the analysis at $S$, see also Remark \ref{rem:first vertical mode}.}
 $$\lambda_{-2} = \pi^2,\quad \psi(Y) = \sqrt2 \sin\pi Y$$
  We will seek (quasi-)eigenvalues of $P$ of the form $\lambda(h)\in h^{-2}\lambda_{-2}+\Cinf(\R_+)$.

\subsubsection{Function spaces, leading parts and model operators}
We want to define spaces $\calE(M)$ and $\calR(M)$ which will contain the eigenfunctions/quasimodes and remainders in the construction, respectively.

Our resolution was chosen so that eigenfunctions have a chance of being smooth on $M$, so $\calE(M)\subset\Cinf(M)$. Since $\frac hx$ is a defining function for $A$, functions $u\in\Cinf(M)$ have an expansion at $A$
$$ u \sim \sum_{j=0}^\infty \left(\frac hx\right)^j \utilde_j(x,Y),\quad \utilde_j\in\Cinf(A).$$
In the sequel it will be convenient\footnote{In order to have $Pu\sim\sum\limits_j h^j Pu_j$. 
But note that $h$ is not a defining function of $A$.} to write this as
\begin{equation}
 \label{eqn:A expansion}
 u \sim \sum_{j=0}^\infty h^j u_j(x,Y)
\end{equation}
where $u_j = x^{-j}\utilde_j$. Note that $u_j$ may be not smooth at $x=0$, i.e. at $S\cap A$, even though $u\in \Cinf(M)$.
We posit that quasimodes satisfy the stronger condition that $u_j$ be smooth on $A$ (including $S\cap A$) and define
\begin{equation}
\label{eqn: triangular condition} 
 \Cinftr(M) = \{u\in\Cinf(M):\ u_j\in\Cinf(A)\ \text{ in the expansion \eqref{eqn:A expansion}}\}
\end{equation}
See Remark \ref{rem:why triangular} below for an explanation why we expect quasimodes to satisfy this condition. 
This can be reformulated as a \lq triangular' condition on the indices in  the expansion  at the corner $S\cap A$:
\begin{equation}
 \label{eqn:triangular condition}
\begin{gathered}
\text{If }u\in\Cinf(M),\ u\sim\sum\limits_{j,l=0}^\infty a_{jl}(Y) \left(\frac hx\right)^j x^l\quad\text{ near }S\cap A\\
 \text{ then } u\in\Cinftr(M) \iff \left(a_{jl}\neq0\Rightarrow l\geq j\right).
\end{gathered}
\end{equation}

In addition, quasimodes should vanish at the Dirichlet boundary $\dDir M$. As before, we indicate this by the index $D$ in the function spaces. 
For functions on the faces $A$, $S$, $S\cap A$ we use a similar notation. For example, $\CinfD(A)$ is the space of smooth functions on $A$ vanishing on the Dirichlet boundary of $A$, which consists of the three sides $x=1$, $Y=0$, $Y=1$. 

\begin{definition} \label{def:adiab ends qm space}
 The {\bf space of quasimodes} for the adiabatic limit with ends  is defined as
 $$\calE(M) = \CinfDtr(M) $$
 i.e. smooth functions on $M$ satisfying Dirichlet boundary conditions and the triangular condition explained above.
The {\bf leading parts} of $u\in\calE(M)$ are defined as
$$ u_S := u_{|S},\quad u_A := u_{|A}.$$
\end{definition}
What are the restrictions of elements of $\calE(M)$ to $\partial_0 M=S\cup A$? Define
\begin{align}
\label{eqn:def Cinftr S}
\CinfDtr(S) &:= \{u_s\in\CinfD(S):\, u_s=a(Y)+O(X^{-\infty})\text{ as }X\to\infty,\
a\in \CinfD(S\cap A)\}\\
\label{eqn:def E SA}
\calE(\partial_0 M) &:= 
 \{(u_s,u_a):\,  u_s\in\CinfDtr(S),\ u_a\in\CinfD(A),\ u_s=u_a\text{ at }S\cap A\}
\end{align}
Here we use the coordinate $X$ on $S$. Recall that $X^{-1}$ defines the face $S\cap A$ of $S$. 
\begin{lemma}[leading parts of quasimodes, adiabatic limit with ends]
\label{lem:adiab ends qm exact}
If $u\in\calE(M)$ then $(u_S,u_A) \in \calE(\partial_0 M)$. Conversely, given $(u_s,u_a)\in \calE(\partial_0 M)$ there is $u\in\calE(M)$ satisfying
$(u_S,u_A)=(u_s,u_a)$, and $u$ is unique modulo $h\calE(M)$.
\end{lemma}
This could be formulated as existence of a short exact sequence:
\begin{equation}
\label{eqn:short exact seq qm adiab ends} 
 0 \to h\calE(M) \to \calE(M) \to \calE(\partial_0 M) \to 0
\end{equation}
where the left map is inclusion and the right map is restriction.
\begin{proof}
 It is clear that the restrictions of $u\in\calE(M)$ to $S,A$ are smooth and agree at $S\cap A$. Write the expansion of $u$ at the corner as in \eqref{eqn:triangular condition}. The $l=0$ terms give the expansion of $u_S$ at $S\cap A$, i.e. as $\frac hx\to0$. The only such term is $j=0$, so $u_S=a_{00}(Y) + O(\left(\frac hx\right)^\infty)$. From $\frac hx=X^{-1}$ we get $(u_S,u_A)\in\calE(\partial_0 M)$.
 
Given $(u_s,u_a)\in\calE(\partial_0 M)$ one constructs $u\in\calE(M)$ having this boundary data using the Borel Lemma \ref{lem:borel}, as follows.
We write $\eta=X^{-1}$ for the function defining $A$ and suppress the $Y$-coordinate. Write $u_a(x)\sim \sum_{l\geq0} a_{0l} x^l$, $x\to0$.
We choose $u$ having complete expansions
\begin{align*}
 u(x,\eta) &\sim u_s(\eta) + \sum_{l=1}^\infty a_{0l} x^l & \text{ as }x\to 0,\quad \text{ i.e.\ at }S\\
 u(x,\eta) &\sim u_a(x) & \text{ as }\eta\to0,\quad \text{ i.e.\ at }A
\end{align*}
(with error $O(\eta^\infty)$ in the second case). 
Such a $u\in\CinfD(M)$ exists by the Borel Lemma -- the matching conditions at $x=\eta=0$ are satisfied since $\frac{\partial a_{0l}}{\partial\eta}=0$ for all $l$.
Also, the expansions satisfy the triangular condition in \eqref{eqn:triangular condition}, hence $u\in\CinfDtr(M)$.

Finally, we need to show that if $u\in\calE(M)$, $(u_S,u_A)=0$ then $u\in h\calE(M)$. Now $h=\frac hx x$ is a total boundary defining function for $\{S,A\}$, so $u=h\utilde$ for some $\utilde\in\CinfD(M)$ by Lemma \ref{lem:divide by bdf}. In the expansion \eqref{eqn:A expansion} all $u_j$ are smooth and $u_0=u_A=0$, so 
$\utilde=\sum_{j\geq 1} h^{j-1}u_j$ is in $\CinfDtr(M)$.
\end{proof}

The definition of the remainder space combines the triangular condition with the remainder space for the adiabatic limit. First, the choice of 
$\lambda_{-2}$ defines a projection type map related to the projection $\pi:U_A\to U_B$, see \eqref{eqn:A face projection},
$$ \Pi: \Cinf(M)\to\Cinf(U_B),\quad f \mapsto \langle f_{|U_A},\psi\rangle_F\,. $$
Then
\begin{equation}
 \label{eqn:Pi Pi_F ends}
 \Pi \circ (P_F-\lambda_{-2}) = 0
\end{equation}
where this is defined.
%

\begin{definition} \label{def:adiab ends rem space}
 The {\bf remainder space} for the adiabatic limit with ends is defined as
 $$ \calR(M) = \{f\in h^{-2}\Cinftr(M):\ \Pi f\ \text{ is smooth at }B\}.$$
where $B:=B\times\{0\}\subset U_B$.
 The {\bf leading parts} of $f\in\calR(M)$ are
 $$ f_{-2,S} := (h^2 f)_{|S},\quad f_{AB} := 
\begin{pmatrix}
f_{-2,A} \\
\Pi f_{0,A}
\end{pmatrix}
$$
where $f_{-2,A} =  (h^2 f)_{|A}$, $\Pi f_{0,A} =  (\Pi f)_{|B}$.
\end{definition}
Thus, a function $f\in h^{-2}\Cinf(M)$ is in $\calR(M)$ iff it has an expansion $f\sim h^{-2}f_{-2,A}+h^{-1}f_{-1,A} + f_{0,A} + \dots$ at $A$ with $f_{j,A}\in\Cinf(A)$ analogous to \eqref{eqn:A expansion}, and $\Pi f_{-2,A} = \Pi f_{-1,A} = 0$. This defines $f_{-2,A}$ and $f_{0,A}$ in the definition of $f_{AB}$.
Note that $f\in\calR(M)$ implies that
\begin{equation}
 \label{eqn:adiab ends rem cond}
f_{-2,S}\in\Cinftr(S),\quad \Pi f_{-2,S} = 0 \text{ at } S\cap A.
\end{equation}
 The first statement follows as in the proof of Lemma \ref{lem:adiab ends qm exact} and the second from $\Pi (h^2 f)_{A}=0$ and $(h^2f)_A = (h^2f)_S$ at $S\cap A$.

As usual, the remainder space and leading part definitions are justified by the following properties. Recall the definition of the operators $P, P_F, P_B, P_S$ in \eqref{eqn:adiab ends notation operators}.
\begin{leading part and model operator lemma}[adiabatic limit with ends] 
\mbox{}
\begin{enumerate}
 \item[a)]
If $f\in\calR(M)$ then 
$$f\in h\calR(M)\ \text{ if and only if }\ f_{-2,S} = 0, \ f_{AB}=0.$$ 
 \item[b)]
  For $\lambda \in h^{-2}\lambda_{-2} + \Cinf(\R_+)$ we have
$$ P-\lambda: \calE(M) \to \calR(M) $$
and
\begin{align}
\label{eqn:model S ends}
 [(P-\lambda)u]_{-2,S} &= (P-\lambda)_S u_S &\text{where } (P-\lambda)_S&= P_S - \lambda_{-2}\\
\label{eqn:model A ends} 
[(P-\lambda)u]_{AB} &= (P-\lambda)_A u_A &\text{where } (P-\lambda)_A &= 
\begin{pmatrix}
 P_F-\lambda_{-2} \\
 \Pi (P_B - \lambda_0)
\end{pmatrix}
\end{align}
where $\lambda_0$ is the constant term of $\lambda$.

\end{enumerate}
\end{leading part and model operator lemma}
The operators $(P-\lambda)_S$, $(P-\lambda)_A$ are called the {\bf model operators} of $P-\lambda$ at $S$ and at $A$.
There is also a short exact sequence like \eqref{eqn:short exact seq qm adiab ends} for $\calR(M)$, but we need only what is stated as a).
\begin{proof}\mbox{}
\begin{enumerate}
 \item[a)]
 ``$\Rightarrow$'' is obvious. ``$\Leftarrow$'': 
If $f\in\calR(M)$ then $h^2f\in\Cinftr(M)$, and $f_{-2,S}=0$, $f_{AB}=0$ imply $(h^2f)_S=0$, $(h^2f)_A=0$, so Lemma \ref{lem:adiab ends qm exact} gives $h^2f\in h\Cinftr(M)$, so $f=h\ftilde$ with $\ftilde\in h^{-2}\Cinftr(M)$. Furthermore, $f_{AB}=0$ implies
$f_{-2,A}=0$ and $\Pi f_{0,A}=0$, so $\Pi\ftilde$ is smooth at $B$, hence $\ftilde\in\calR(M)$.
\item[b)]
If $u\in\calE(M)$ then  $(P-\lambda)u=h^{-2}(P_S-\lambda_{-2} + O(h^2))u$ near $S$ and $(P-\lambda)u=h^{-2}(P_F-\lambda_{-2})u + (P_B-\lambda_0)u + O(h)$ near $A$. This is clearly in $h^{-2}\Cinftr(M)$, and even in $\calR(M)$ by \eqref{eqn:Pi Pi_F ends}.
The definition of leading parts directly implies \eqref{eqn:model S ends}, \eqref{eqn:model A ends}.
 
\end{enumerate}
\end{proof}

\subsubsection{Analytic input for model operators}
At the face $A$, i.e. for the operators $P_F=-\partial_Y^2$ and $P_B=-\partial_x^2$ on $F=[0,1]_Y$ resp. $B=[0,1]_x$,  with Dirichlet boundary conditions, we have the standard elliptic solvability result, Lemma \ref{lem:orth decomp}.

The solvability properties of the model operator $(P-\lambda)_S$ are of a different nature, essentially since this operator has essential spectrum. 

Recall from \eqref{eqn:surgery face} that the interior of $S$ can be identified with the unbounded domain $\Omega^\infty\subset\R^2$, see Figure \ref{fig:flat M}. This set is the union of a compact set and an infinite strip, hence an example of a space with infinite cylindrical ends, and we can use the standard theory for such spaces. We assume:
\begin{equation}
 \label{eqn:adiab ends non-resonance}
\begin{aligned}
 \text{{\bf Non-resonance assumption: } The resolvent $z\mapsto (P_S-z)^{-1}$} \\
  \text{ of the Laplacian $P_S$ on $\Omega^\infty$ has no pole at $z=\lambda_{-2}=\pi^2$.}
\end{aligned}
\end{equation}
It is well-known that this condition is equivalent to the non-existence of bounded solutions of $(P_S-\lambda_{-2})v=0$, and also to the unique solvability of $(P_S-\lambda_{-2})v=f$ for compactly supported $f$, with bounded $v$, see \cite[Proposition 6.28]{Mel:APSIT}. Also, this condition is satisfied for convex sets $\Omega^\infty$, see \cite[Lemma 7]{GriJer:AEPD}, and holds for generic $\Omega_L$.
\begin{lemma}\label{lem:surgery solving}
Assume $\Omega^\infty\subset\R^2$ satisfies the non-resonance assumption \eqref{eqn:adiab ends non-resonance}.

If $f_s\in\Cinftr(S)$, $\Pi f_s = 0$ at $S\cap A$ then the equation
\begin{equation}
 \label{eqn:S equation}
 (P_S-\lambda_{-2}) v_s = f_s
\end{equation}
has a unique bounded solution $v_s$, and $v_s\in\CinfDtr(S)$.
\end{lemma}

\begin{proof}
Uniqueness holds since $(P_S-\lambda_{-2})v=0$ has no bounded solution. For existence, we first reduce to the case of compactly supported $f_s$, then use the non-resonance assumption to get a bounded solution $v_s$ and then show that $v_s\in\CinfDtr(S)$. 
The first and third step can be done by developing $f_s$ and $v_s$ for each fixed $X>0$ in eigenfunctions 
$\psi_k(Y)=\sqrt2 \sin k\pi Y$ of the \lq vertical\rq\ operator $-\partial_Y^2$ on $[0,1]$ with Dirichlet conditions: 
$f_s(X,Y) = \sum_{k=1}^\infty f_{k}(X) \psi_k(Y)$. Then \eqref{eqn:S equation} is equivalent, in $X>0$, to the ODEs
$(-\frac{d^2}{d X^2} + \mu_k) v_k = f_k$
where  $\mu_k=(k^2-1)\pi^2$, and these can be analyzed explicitly. For example, if $k>1$ and $f_k(X)=0$ for large $X$ then any bounded solution $v_k$ must be exponentially decaying. For details see \cite[Lemma 6 and Lemma 9 (with $p=0$)]{GriJer:AEPD}.
\end{proof}
\begin{remark}
\label{rem:why triangular}
This lemma explains why we expect the \lq triangular' condition on the Taylor series of quasimodes: for compactly supported $f_s$ the solution $v_s$ lies in $\CinfDtr(S)$. This leads to the definition of $\calE(\partial_0 M)$. Then $\calE(M)$ must be defined so that the sequence \eqref{eqn:short exact seq qm adiab ends} is exact.
\end{remark}

\subsubsection{Inductive construction of quasimodes}
\begin{indlist}
 \item[\bf Initial step]
We want to solve
$$ (P-\lambda)u \in h\calR(M),\quad u\in\calE(M).$$
By the leading part and model operator lemma this means
\begin{align}
\label{eqn:adiab ends uS}
 (P_S-\lambda_{-2}) u_S &= 0\\
 \label{eqn:adiab ends uA}
(P_F-\lambda_{-2}) u_A &=0\\
\label{eqn:adiab ends Pi uA}
(P_B-\lambda_0) \Pi u_A &=0 
\end{align}
where we used that $\Pi$ commutes with $P_B=-\partial_x^2$.

Also, $(u_S,u_A)\in\calE(\partial_0 M)$, defined in \eqref{eqn:def E SA}.

First, Lemma \ref{lem:surgery solving} implies $u_S=0$.

Therefore $u_A=0$ at $S\cap A$, hence $u_A$ satisfies Dirichlet boundary conditions at all four sides of the square $A$. Thus we have an adiabatic problem as treated in Section \ref{sec:adiab limit const}, so $\lambda_0$ must be a Dirichlet eigenvalue of $-\partial_x^2$ on $B=[0,1]$, i.e.
$$ \lambda_0 = \pi^2m^2, \quad u_A = u_0:= \phi\otimes\psi,\ \phi(x)= \sqrt2\sin\pi m x $$
for some $m\in\N$. Since $B$ is one-dimensional, $\lambda_0$ is a simple eigenvalue as required in \eqref{eqn:adiab PB simple ev}.

From now on we fix $m$, $\phi$, $\lambda_{-2}=\pi^2$, $\lambda_0$, and $u_0=\phi\otimes\psi$.
 \item[\bf Inductive step]
 \begin{inductivesteplemma}[adiabatic limit with ends]
Let $\lambda_{-2}$, $\lambda_0$ and $u_0$ be chosen as above in the initial step, and let $k\geq1$.
Suppose $\lambda\in h^{-2}\Cinf(\R_+)$, $u\in \calE(M)$ satisfy
 $$(P-\lambda)u\in h^k \calR(M)$$
 and $\lambda= h^{-2}\lambda_{-2}+\lambda_0+O(h)$ and $u_S=0, u_A=u_0$. Then there are $\mu\in\R$, $v\in \calE(M)$ so that
  $$  (P-\tilde\lambda)\tilde u \in h^{k+1}\calR(M)$$
 for $\tilde\lambda=\lambda+h^k\mu$, $\tilde u=u+h^kv.$ The number
 $\mu$ and the restriction $v_S$ are unique, and $v_A$ is unique up to adding constant multiples of $u_0$.
\end{inductivesteplemma}
\begin{proof}
Writing  $(P-\lambda)u=h^kf$, $f\in\calR(M)$ we have
$$ (P-\lambdatilde)\utilde = h^k[f-\mu u + (P-\lambda)v - h^k\mu v]$$ 
This is in $h^{k+1}\calR(M)$ if and only if the term in brackets is in $h\calR(M)$, which by the leading part and model operator lemma is equivalent to
\begin{align}
\label{eqn:adiab ends S inhomog}
 (P_S-\lambda_{-2})v_S &= - f_{-2,S} \\
\label{eqn:adiab ends A inhomog}
 (P_F-\lambda_{-2})v_A &= - f_{-2,A} \\
 \label{eqn:adiab ends Pi A inhomog}
 (P_B-\lambda_0)\Pi v_A &= - \Pi f_{0,A} + \mu\phi
\end{align}
where we have used that $(h^2 u)_S=0$, $(h^2u)_A=0$ and $\Pi u_A = \Pi u_0=\phi$.

We first solve at $S$: We have $f\in\calR(M)$, so by \eqref{eqn:adiab ends rem cond} we can apply Lemma \ref{lem:surgery solving} with $f_s=-f_{-2,S}$ and obtain $v_S\in\CinfDtr(S)$ solving \eqref{eqn:adiab ends S inhomog}. This determines in particular $v_{|S\cap A}$, i.e. the boundary value of $v_A$ at $S\cap A$.

Now at $A$ we need to solve an adiabatic problem, but with an inhomogeneous boundary condition at $S\cap A$. To this end we extend
$v_{|S\cap A}$ to $v'\in\CinfD(A)$. 
Writing $v_A=v'+v''$ we then need to find $v''$ satisfying homogeneous boundary conditions also at $S\cap A$, and solving \eqref{eqn:adiab ends A inhomog}, \eqref{eqn:adiab ends Pi A inhomog} with $f_{-2,A}$ modified to $f_{-2,A} + (P_F-\lambda_{-2})v'$ and
$f_{0,A}$ modified to $f_{0,A} + (P_B-\lambda_0)v'$.

This is an adiabatic problem, so Lemma \ref{lem:adiab analysis} guarantees the existence of a solution $v''$. Note that since $\Pi$ commutes with $P_B$ there is no off-diagonal term in \eqref{eqn:adiab model triangular} (where $P_0=P_B$ in current notation).

The uniqueness follows directly from \eqref{eqn:adiab ends S inhomog}-\eqref{eqn:adiab ends Pi A inhomog}: The difference between two solutions $v$ would satisfy the same equations with $f=0$, so would have to vanish at $S$ and therefore solve the adiabatic problem at $A$ with homogeneous boundary condition at $S\cap A$, for which we have already shown that $\mu$ is unique and $v_A$ is unique up to multiples of $u_0$.
\end{proof}

\end{indlist}

Now by the same arguments as for Theorem \ref{thm:regular quasimodes} we obtain from the initial and inductive steps:
\begin{theorem}[quasimodes for adiabatic limit with ends]
\label{thm:adiab ends}
Consider the  family of domains $\Omega_h$ defined in \eqref{eqn:adiab ends Omegah def}.  Suppose the non-resonance assumption \eqref{eqn:adiab ends non-resonance} is satisfied. Then for each $m\in\N$ there are $\lambda_m\in h^{-2}\Cinf(\R_+)$, $u_m\in\CinfDtr(M)$ satisfying
 $$ (P-\lambda_m)u_m \in h^\infty \Cinf(M)$$
and
\begin{align*}
  \lambda_m & = h^{-2}\pi^2 + m^2\pi^2 + O(h)\\
   u_A&=2\sin m\pi x\,\sin \pi Y\,,\quad u_S = 0
\end{align*}
\end{theorem}
There is also a uniqueness statement similar to the one in Theorem \ref{thm:adiab quasimodes}.
\begin{remark}[Quasimodes vs. modes]
 It is shown in \cite{GriJer:AEPD} that for convex $\Omega_h$ all eigenfunctions are captured by this construction. That is, for each $m\in\N$ there is $h_0>0$ so that for $h<h_0$ the $m$th eigenvalue of $\Omega_h$ is simple, and both eigenvalue and (suitably normalized) eigenfunction  are approximated  by $\lambda_m$, $u_m$ with error $O(h^\infty)$. However, if $\Omega_h$ is not convex then there may be an additional finite number of eigenvalues not captured by this construction. Essentially, these arise from $L^2$-eigenvalues of the Laplacian on $\Omega^\infty$ below the essential spectrum. See \cite{Gri:SGNS} and references there for a detailed discussion.
\end{remark}
\subsection{Explicit formulas}
The inductive step yields a method for finding any number of terms in the expansions of $\lambda_m$ and $u_m$ as $h\to0$ in terms of solutions of the model problems. In \cite{GriJer:AEPD} the next two terms for $\lambda_m$ are computed:
$$ \lambda_m = h^{-2}\pi^2 + m^2\pi^2 (1+ah)^{-2} + O(h^3) $$
where $a>0$ is determined by the scattering theory of $-\Delta$ on $\Omega^\infty$ at the infimum  of the essential spectrum, which equals $\pi^2$.  
More precisely, there is, up to scalar multiples, a unique polynomially bounded solution $v$ of $(\Delta+\pi^2)v=0$ on $\Omega^{\infty}$, and
it has the form
$$ v(X,Y) = (X+a) \sin \pi Y + O (e^{-X}) \text{ as }X\to\infty.$$
This fixes $a$. Another description is $a=\frac12\gamma'(0)$ where $\gamma(s)$ is the scattering phase at frequency $\pi^2+s^2$. 

\subsection{Generalizations}\label{subsec:general ends added}
The structure of $\Omega_h$ may be described as \lq thin cylinder with end attached'. A natural general setup for this structure is obtained by replacing the $Y$-interval $[0,1]$ by a compact Riemannian manifold, of dimension $n-1$, and the end $\Omega_L$ by another compact Riemannian manifold, of dimension $n$, which has an isometric copy of $Y$ as part of its boundary. One may also add another Riemannian manifold as right end. This is studied in global analysis (where it is sometimes called \lq analytic surgery\rq) as a tool to study the glueing behavior of spectral invariants, see  \cite{Has:ASAT}, \cite{HasMazMel:ASAE}, \cite{MazMel:ASEI} for example.
Other degenerations which have been studied by similar methods include conic degeneration \cite{GuiShe:LERHLARTSEAT}, \cite{Row:SGACC}, \cite{She:CDDL} and degeneration to a (fibred) cusp \cite{AlbRocShe:ATRWRMWC}, \cite{AlbRocShe:RHKTDFC}.

Another generalization is to have several thin cylinders meeting in prescribed ways, so that in the limit $h=0$ one obtains a graph-like structure instead of an interval. This is called a \lq fat graph\rq. For example, consider a finite graph embedded in $\R^n$ with straight edges, and let $\Omega_h$ be the set of points of $\R^n$ having distance at most $h$ from this graph.
This was studied in detail in \cite{ExnPos:CSGLTM}, \cite{Gri:SGNS} and \cite{MolVai:SSNTFSDA}, see also \cite{Gri:TTMPGASG} for a discussion and many more references.
The methods in these papers actually yield a stronger result: 
$\lambda_m(h)$ is given by a power series in $h$ which converges for small $h$, plus an exponentially small error term.

\section{Summary of the quasimodes constructions}
\label{sec:summary}
We summarize the essential points of the quasimode constructions, continuing the outline given in Section \ref{sec:main steps}.

In the case of a regular perturbation we introduced the iterative setup that allowed us to reduce the quasimode construction to the solution of a model problem (Lemma \ref{lem:orth decomp}). It involves spaces of quasimodes and remainders and notions of leading part. In this case these are simply smooth functions and their restriction to $h=0$.

For the adiabatic limit problem with constant fibre eigenvalue this needs to be refined: the different scaling in fibre and base directions requires a new definition of remainder space and leading part of remainders (Definition \ref{def:adiab limit rem space}). The model operator combines fibre and horizontal operators, and its triangular structure with respect to the decomposition of functions in fibrewise $\lambda_{-2}$ modes and other modes, Equation \eqref{eqn:adiab model triangular}, enables us to solve the model problem.

The adiabatic limit problem with variable fibre eigenvalue can be reduced to the previous case by expanding the fibre eigenvalue (as function on the base) around its maximum and by rescaling the base variable. This rescaling balances the leading non-constant term in the expansion of the eigenvalue with the leading term of the base operator. The rescaling is encoded geometrically by a blow-up of the total space.

The adiabatic limit problem with ends carries the new feature of having two regions with different scaling behavior. Geometrically this corresponds to two boundary hypersurfaces, $A$ and $S$, at $h=0$ in the resolved total space. The model problem at $A$ is the same as for the adiabatic limit with constant fibre eigenvalue. The model problem at  $S$ is a scattering problem, i.e. a spectral problem on a non-compact domain. The properties of the solutions of the scattering problem lead to the triangular condition on the Taylor series at the corner $S\cap A$ in the spaces of quasimodes and remainders. Once this setup is installed the construction proceeds in a straight-forward way as in the other cases.


%

\def\cprime{$'$} \def\cprime{$'$}

\end{document}